\documentclass[reqno]{amsart}
\usepackage{amscd}
\usepackage{amsmath}
\usepackage[dvips]{graphics}
\usepackage{color}
\usepackage{balance}

\def\beq{\begin{equation}}
\def\eeq{\end{equation}}
\def\ba{\begin{array}}
\def\ea{\end{array}}

\newcommand{\g}{{\mathbf g}}

\newenvironment{abs}{\textbf{Abstract}\mbox{  }}{ }
\newenvironment{key words}{\textbf{Keywords}\mbox{  }}{ }
\newtheorem{theorem}{Theorem}[section]
\newtheorem*{theorem*}{Theorem A }
\newtheorem{lemma}[theorem]{Lemma}
\newtheorem{prop}[theorem]{Proposition}
\newtheorem{crl}[theorem]{Corollary}

\theoremstyle{definition}
\renewenvironment{proof}{\noindent{\textbf{Proof.}}}{\hfill$\Box$}
\newtheorem{rem}[theorem]{Remark}

\newtheorem{df}[theorem]{Definition}

\theoremstyle{remark}
 
\numberwithin{equation}{section}
\begin{document}
\pagestyle{plain}

\title{Reversed Hardy-Littlewood-Sobolev inequalities with vertical weights on the upper half space}
\author{Jingbo Dou, Yunyun Hu, Jingjing Ma}

\address{Jingbo Dou, School of Mathematics and Statistics,
Shaanxi Normal University,
Xi'an, 710119, P. R. China}
\email{jbdou@snnu.edu.cn}

\address{Yunyun Hu, School of Mathematics and Statistics,
Shaanxi Normal University,
Xi'an, 710119, P. R. China}
\email{yhu@snnu.edu.cn}
\medskip

\address{Jingjing Ma, School of Mathematics and Statistics,
Shaanxi Normal University,
Xi'an, 710119, P. R. China}
\email{mjjwnm@126.com}

\thanks{Corresponding author: Yunyun Hu at yhu@snnu.edu.cn.}

\date{}

\maketitle

\begin{abs}
In this paper, we obtain the reversed Hardy-Littlewood-Sobolev inequality with vertical weights on the upper half space and discuss the extremal functions. We show that the sharp constants in this inequality are attained by introducing a renormalization method. The  classification of corresponding extremal functions is discussed via the method of moving spheres. Moreover, we prove the sufficient and necessary conditions of existence for positive solutions to the Euler-Lagrange equations  by using Pohozaev identities in weak sense. This renormalization method is rearrangement free, which can be also applied to prove the existence of extremal functions for sharp (reversed) Hardy-Littlewood-Sobolev inequality with extended kernels and  other similar inequalities.
\end{abs}

\begin{key words} Reversed Hardy-Littlewood-Sobolev inequality, extremal function, radical  symmetry, regularity, moving sphere method.
 \end{key words}
\indent

\section{Introduction\label{Section 1}}

The classical Hardy-Littlewood-Sobolev (HLS) inequality established in \cite{Hardy1928,Sobolev1938}, states that
\begin{equation}\label{HL2}
\int_{\mathbb{R}^n}\int_{\mathbb{R}^n}|x-y|^{-\lambda} f(x)g(y)dxdy\leq C_{n,p,\lambda}\|f\|_{L^{p}(\mathbb{R}^n)}\|g\|_{L^{q}(\mathbb{R}^n)}
\end{equation}
holds for $f\in L^{p}(\mathbb{R}^n), g\in L^{q}(\mathbb{R}^n)$, $1<p,q<\infty, 0<\lambda<n$ and $\frac{1}{p}+\frac{1}{q}+\frac{\lambda}{n}=2$.
Lieb \cite{Lieb1983} proved the existence of extremal functions for the inequality \eqref{HL2} by using rearrangement inequalities. He also classified extremal functions and computed the sharp constant in the case of $p=2$, or $q=2$, or $p=q=\frac{2n}{2n-\lambda}$.

HLS inequality \eqref{HL2} plays prominent role in many geometric problems,  such as Yamabe problem, Ricci flow problem, etc.  In particular, from global analysis, the integral (curvature) equation was studied via HLS inequality by Zhu \cite{Z2016}, Dou and Zhu \cite{DZ2019}, Dou, Guo and Zhu \cite{DGZ2020}. It also implies many \emph{important geometrical inequalities}, such as the \emph{Gross logarithmic Sobolev inequality} \cite{Gr1976} and the \emph{Moser-Onofri-Beckner inequality} \cite{Beckner20081}. These inequalities have many applications in analysis, geometric problems and quantum field theory equations.

In the past years, HLS inequality \eqref{HL2} has been studied by many mathematicians.
Using the competing symmetry method, Carlen and Loss \cite{CarL1990} provided an alternative way to reprove Lieb's result.  Frank and Lieb \cite{FL2010} offered a new proof for diagonal cases by employing the reflection positivity of inversions in spheres, and employed the method of moving spheres (Li-Zhu Lemma in \cite{LZ1995}) to characterize the minimizing functions. Later, Frank and Lieb \cite{FL20121} gave a rearrangement-free method to show the sharp inequality \eqref{HL2} by exploiting their conformal covariance.  This new proof also leads to a proof of analogous inequalities on the CR spheres and Heisenberg group (see \cite{FL20122}).

Recently, Dou and Zhu \cite{DZ2015b} established the reversed HLS inequality
\begin{equation} \label{DZ-reverse}
\int_{\mathbb{R}^n}\int_{\mathbb{R}^n}\frac{ f(x)g(y)}{|x-y|^{\lambda}}dxdy\geq S_{n,p,\lambda}\|f\|_{L^{p}(\mathbb{R}^n)}\|g\|_{L^{q}(\mathbb{R}^n)}
\end{equation}
for nonnegative functions $f\in L^{p}(\mathbb{R}^n), g\in L^{q}(\mathbb{R}^n)$, where $n\ge1,\lambda<0, \frac{n}{2n-\lambda}<p,q<1,  $ and $\frac{1}{p}+\frac{1}{q}+\frac{\lambda}{n}=2$. In particular, for $p=q=\frac{2n}{2n-\lambda}$,
they proved the existence of extremal functions and classified all extremal functions via the method of moving spheres, and computed the best constant $S_{n,\frac{2n}{2n-\lambda},\lambda}$.  Inequality \eqref{DZ-reverse} is a  complement of HLS inequality for case $1<p,q<\infty$.

Closely related to the HLS inequality, Stein and Weiss \cite{SW1958} established the following weighted HLS inequality (also call Stein-Weiss inequality)
\begin{equation}\label{HL1}
\int_{\mathbb{R}^n}\int_{\mathbb{R}^n}|x|^{-\alpha}|x-y|^{-\lambda} f(x)g(y)|y|^{-\beta}dxdy\leq C_{n,\alpha,\beta,p,q}\|f\|_{L^{p}(\mathbb{R}^n)}\|g\|_{L^{q}(\mathbb{R}^n)},
\end{equation}
where $1<p,q<+\infty$, $\alpha,\beta,\lambda$ satisfy the following conditions
$$\frac{1}{p}+\frac{1}{q}+\frac{\alpha+\beta+\lambda}{n}=2,\ \ \ \frac{1}{p}+\frac{1}{q}\geq1,$$
$$\alpha+\beta\geq0, \ \ \alpha<\frac{n}{p'}, \ \ \ \beta<\frac{n}{q'},\ \ \ 0<\lambda<n.$$
In \cite{Lieb1983}, Lieb proved the existence of extremal function for the inequality \eqref{HL1} by using rearrangement inequalities in the case $p<q$ and $\alpha,\beta\geq0$. In addition, he pointed out the extremals cannot be expected to exist when $p=q=2$, $\lambda=n-1$, $\alpha=0$ and $\beta=1$, see also \cite{Her1977}. In the case $p=q$, Beckner \cite{Beckner20081,Beckner20082} obtained the sharp constant of inequality \eqref{HL1} by establishing general Stein-Weiss lemma.
For $p<q$, Chen, Lu and Tao \cite{CLLT2019} proved the existence of extremal functions for the weighted HLS inequality under assumption $\alpha+\beta\geq0$, which extended Lieb's results by relaxing the restriction $\alpha,\beta\geq0$. Dou \cite{D2016} established weighted  HLS  type inequality on the upper half space by Hardy inequality and  discussed the existence of extremal functions by using rearrangement inequalities for $p<q$ and $\alpha,\beta\geq0$.
Later, Chen, Liu, Lu and Tao \cite{CLLT2018} established a reverse version of inequalities \eqref{HL1} and proved the existence of extremal functions. Chen, Lu and Tao \cite{CLT2019} also established a general reversed weighted HLS  inequality on the upper half space and proved the existence of their extremal functions.

In the last two decades, various extensions of HLS inequality have been investigated. For example, one has sharp HLS inequality on Heisenberg group, on compact Riemannian manifolds, on reversed forms, and on weighted forms; for interested readers, we refer to \cite{FL20122,FS1974,Han2016,DZ2015b,Ngo20171,Ngo20172,CLLT2021, HLZ2012, Beckner2021}.

\medskip

Apart from these extensions, Gluck \cite{Gl2020} established the following sharp HLS inequality involving general kernels (extended kernels) on the upper half space $\mathbb{R}_+^n$:
\begin{equation}\label{equality00}
\Big|\int_{\mathbb{R}_+^n}\int_{\partial\mathbb{R}^n_+} K_{\alpha,\beta}(x'-y',x_n)f(y)g(x) dydx\Big|\leq C_{n,\alpha,\beta}
\|f\|_{L^p(\partial\mathbb{R}_+^n)}\|g\|_{L^{q}(\mathbb{R}_+^n)},
\end{equation}
where the kernel
$$K_{\lambda,\beta}(x',x_n):=\frac{x_n^\beta}{(|x'|^2+x_n^2)^{\frac{\lambda}{2}}}, \quad  x=(x',x_n)\in \mathbb{R}^{n-1}\times (0,\infty)=\mathbb{R}_+^n$$
with $\beta\geq0$, $0<n-\lambda+\beta<n-\beta$, $\frac{\lambda-2\beta}{2n}+\frac{\lambda}{2(n-1)}<1.$

The family of kernels $K_{\lambda,\beta}$ includes the classical Possion kernel $K_{n,1}$, Riesz kernel $K_{\lambda,0}$ and the Poisson kernel $K_{\lambda,1-n+\lambda}$ for the divergence form operator $u\mapsto div(x_n^{n-\lambda} \nabla u)$ on the half space. For the kernels $K_{n,1}$, $K_{\lambda,0}$, $K_{n-\lambda,1-n+\lambda}$ and $K_{n-\lambda,1}$, the corresponding inequalities of the form of \eqref{equality00} have been investigated, see \cite{HWY2008}, \cite{DZ2015a},\cite{Chens2014} and \cite{DGZ2017}, respectively.
In the case $\lambda<0$,
 Dai, Hu and Liu \cite{DHL2021} proved a reverse version of inequality \eqref{equality00}. They also classified all extremal functions and computed the sharp constants.

In the following, we write the upper half space  as $\mathbb{R}^{n+1}_+=\{(x,t)\in \mathbb{R}^n\times \mathbb{R}\,:\, t>0\}$.
Recently, Dou and Ma \cite{DM2021} established the following weighted HLS type inequality on $\mathbb{R}^{n+1}_+$
\begin{equation}\label{WHLS inequality-t}
\int_{\mathbb{R}^{n+1}_+}\int_{\mathbb{R}^{n+1}_+}\frac{f(x,t) g(y,z)}{t^{\alpha}|(x,t)-(y,z)|^{\lambda}z^{\beta} } dxdtdydz
\le C \|f\|_{L^p(\mathbb{R}_+^{n+1})}\|g\|_{L^{q}(\mathbb{R}_+^{n+1})},
\end{equation}
where $f\in L^p(\mathbb{R}_+^{n+1}), \, g\in L^{q}(\mathbb{R}_+^{n+1})$ and $\lambda,\alpha,\beta,p,q$ satisfy
$$0<\lambda<n+1,\ \ 1< p,\ q<\infty,\
\alpha<\frac{1}{p'},\ \beta<\frac {1}{q'},$$
$$ \alpha+\beta\ge0,\ \ \frac 1p +\frac 1{q}+\frac{\alpha+\beta+\lambda}{n+1}=2.$$
They proved the existence of the extremal functions by the concentration-compactness principle. For the conformal invariant case, they showed the explicit form of the extremal functions  on $\partial\mathbb{R}^{n+1}_+$ by the method of moving spheres. Furthermore, some weighted Sobolev inequalities on $\mathbb{R}^{n+1}_+$ were established by inequality \eqref{WHLS inequality-t}.

Our goal in this paper continues to investigate reverse weighted HLS type inequality related to inequality \eqref{WHLS inequality-t}.
We first establish the following reversed HLS  type inequality with vertical weights.
Assume that $\alpha,\beta,\lambda,p,q$ satisfy
\begin{equation}\label{RWH-exp-0}
\begin{cases}
&-n-1<\lambda<0, 0< p,q<1,\\
&0\le\alpha<-\frac{1}{p'},0\le\beta<-\frac{1}{q'},\\
&\frac1p+\frac1{q}+\frac{\lambda-\alpha-\beta}{n+1}=2,\\
\end{cases}
\end{equation} where $\frac{1}{p}+\frac{1}{p'}=1$ and $\frac{1}{q}+\frac{1}{q'}=1$. We always assume that $p$ and $p'$ are conjugate numbers in the whole paper.

\begin{theorem}\label{RWHLSB}
Let $\alpha,\beta,\lambda,p,q$ satisfy \eqref{RWH-exp-0}. Then there exists a constant $N_{\alpha,\beta,\lambda}:=N(n,\alpha,\beta, \lambda, p)>0$ such that for any nonnegative functions $f\in L^p(\mathbb{R}_+^{n+1})$ and $g\in L^{q}(\mathbb{R}_+^{n+1})$,
\begin{equation}\label{RHLSD-O}
\int_{\mathbb{R}^{n+1}_+}\int_{\mathbb{R}^{n+1}_+}\frac{t^\alpha z^\beta f(x,t)g(y,z)}{|(x,t)-(y,z)|^{\lambda}}dxdtdydz
\ge N_{\alpha,\beta,\lambda}\|f\|_{L^p(\mathbb{R}_+^{n+1})}\|g\|_{L^{q}(\mathbb{R}_+^{n+1})}
\end{equation}
holds.  Moreover, the upper and lower bounds of constant  $ N_{\alpha,\beta,\lambda}$ satisfy
\begin{eqnarray}\label{bestrange}
(\frac{pq-p}{2pq-p-q})^{\frac{1-q}{q}}(\frac{pq-q}{2pq-p-q})^{\frac{1-p}{p}}\mbox{min}\{ D_1, D_2\} \le N_{\alpha,\beta,\lambda}\le \mbox{min}\{ D_1, D_2\},
\end{eqnarray}
where $$D_1=\big[\frac{\pi^{\frac{n}{2}}(p-1)}{(n+1+\alpha) p-n-1}\frac{\Gamma(\frac{(\alpha+1) p-1}{2(p-1)})}{\Gamma(\frac{(n+\alpha+1)p-n-1 }{2(p-1)})}\big]^{\frac{p-1}{p}}\big[\frac{\pi^{\frac{n}{2}}(q-1)}{(n+1+\beta-\lambda) q-n-1}\frac{\Gamma(\frac{(\beta+1) q-1}{2(q-1)})}{\Gamma(\frac{(n+\beta +1)q-n-1}{2(q-1)})}\big]^{\frac{q-1}{q}},$$
$$D_2=\big[\frac{\pi^{\frac{n}{2}}(p-1)}{(n+1+\alpha-\lambda)p-n-1}\frac{\Gamma(\frac{(\alpha+1) p-1}{2(p-1)})}{\Gamma(\frac{(n+\alpha+1)p-n-1 }{2(p-1)})}\big]^{\frac{p-1}{p}}\big[\frac{\pi^{\frac{n}{2}}(q-1)}{(n+1+\beta) q-n-1}\frac{\Gamma(\frac{(\beta+1) q-1}{2(q-1)})}{\Gamma(\frac{(n+\beta +1)q-n-1}{2(q-1)})}\big]^{\frac{q-1}{q}}.$$
\end{theorem}

For simplicity, let $r=q'<0$, we  write the assumptions of exponents $\alpha,\beta,\lambda,p,r$ as
\begin{equation}\label{RWH-exp-1}
\begin{cases}
&-n-1<\lambda<0, \ r<0< p<1,\\
&0\le\alpha<-\frac{1}{p'},\ 0\le\beta<-\frac{1}{r},\\
&\frac 1{r}=\frac1p-\frac{n+1+(\alpha+\beta-\lambda)}{n+1}.
\end{cases}
\end{equation}
Define weighted singular operator  for $f(x,t)\in C^\infty_0(\mathbb{R}^{n+1}_+)$ as
\[I_{\alpha,\beta}f(y,z)=\int_{\mathbb{R}^{n+1}_+}\frac{t^\alpha z^\beta  f(x,t)}{|(x,t)-(y,z)|^{\lambda}}dxdt,
\]
 and a  singular operator

\begin{equation*}
E_\lambda f(y,z)=\int_{\mathbb{R}^{n+1}_+}\frac{f(x,t)}{|(x,t)-(y,z)|^\lambda}dxdt, \qquad \forall\ (y,z)\in\mathbb{R}^{n+1}_+.
\end{equation*}
By duality and the reversed H\"{o}lder inequality (see e.g. Lemma 2.1 in \cite{DZ2015b}), it is easy to verify that the inequality \eqref{RHLSD-O} is equivalent to the following  two form inequalities.

\begin{crl}\label{RWHLSD-theo}
Let $\alpha,\beta,\lambda,p,r$ satisfy \eqref{RWH-exp-1}.
There is a constant $C (n,\alpha,\beta,\lambda,p)>0$ depending on $n,\alpha,\beta, \lambda, p$ such that
 \begin{equation}\label{RHLSD-1}
\| I_{\alpha,\beta}f\|_{L^r( \mathbb{R}^{n+1}_+)}\ge C (n,\alpha,\beta,\lambda,p)\|f\|_{L^p(
\mathbb{R}^{n+1}_+)}
\end{equation}
or
\begin{equation}\label{WHLSD-1}
\|(E_\lambda f)z^{\beta}\|_{L^{r}( \mathbb{R}^{n+1}_+)}\ge C (n,\alpha,\beta,\lambda,p)\|t^{-\alpha} f\|_{L^p(
\mathbb{R}^{n+1}_+)}
\end{equation}
holds for any nonnegative
function $f\in L^p(\mathbb{R}^{n+1}_+).$
\end{crl}

We will employ the  reversed Hardy inequality to establish inequality  \eqref{RHLSD-O}.
The Hardy type inequality is a basic integral inequality in harmonic analysis, such as classic weighted Hardy inequality (see e.g., \cite{DHK1997, Maz2011}), reversed Hardy inequality (see e.g., \cite{AKM2019}). It is a powerful and immediate tool to prove weighted HLS inequality (see e.g., \cite{KM2005, D2016}). In contrast, not much is known about that reversed HLS inequality can be proved by using reversed Hardy inequality. Our approach is different from that of Dou, Zhu \cite{DZ2015b} and Chen et al \cite{CLLT2018}.

\medskip

Once we have established the reversed HLS inequality with vertical weights, it is natural to ask whether the extremal functions for inequality \eqref{RHLSD-O} exist or not. To answer this question, we consider the following extremal function of inequality \eqref{RHLSD-O} as follows:
\begin{align}\label{WHLSE}
N_{\alpha,\beta,\lambda}&=\inf\{\int_{\mathbb{R}^{n+1}_+}\int_{\mathbb{R}^{n+1}_+}\frac{t^\alpha z^\beta f(x,t)g(y,z)}{|(x,t)-(y,z)|^{\lambda}}dxdtdydz: f\geq0, g\geq0,\|f\|_{L^{p_\alpha}(\mathbb{R}^{n+1}_+)}= \|g\|_{L^{q_{\beta}}(\mathbb{R}^{n+1}_+)}=1\}\nonumber\\
&=\inf\{\frac{\int_{\mathbb{R}^{n+1}_+}\int_{\mathbb{R}^{n+1}_+}\frac{t^\alpha z^\beta f(x,t)g(y,z)}{|(x,t)-(y,z)|^{\lambda}}dxdtdydz}{\|f\|_{L^{p_\alpha}(\mathbb{R}^{n+1}_+)} \|g\|_{L^{q_{\beta}}(\mathbb{R}^{n+1}_+)}}: f\geq0, g\geq0, f\in L^{p_\alpha}(\mathbb{R}^{n+1}_+),g\in L^{q_{\beta}}(\mathbb{R}^{n+1}_+)\},
\end{align}
where $p_\alpha=\frac{2(n+1)}{2(n+1)+2\alpha-\lambda}$ and $q_{\beta}=\frac{2(n+1)}{2(n+1)+2\beta-\lambda}$.

\medskip

We will prove the attainability of minimizers for variational problem \eqref{WHLSE} with $p=p_\alpha,q=q_{\beta}$.

\begin{theorem}\label{theorem2}Let $\alpha,\beta,\lambda,p,q$ satisfy \eqref{RWH-exp-0},  and $\alpha,\beta>0$. $N_{\alpha,\beta,\lambda}$ is attained by a pair of positive functions $(f,g)\in L^{p_\alpha}(\mathbb{R}^{n+1}_+)\times L^{q_{\beta}}(\mathbb{R}^{n+1}_+)$ satisfying
$\|f\|_{L^{p_\alpha}(\mathbb{R}^{n+1}_+)}=1$ and $\|g\|_{L^{q_{\beta}}(\mathbb{R}^{n+1}_+)}=1$.
\end{theorem}

\begin{rem}
When $\alpha=\beta=0$, Dou, Guo and Zhu \cite{DGZ2020} proved the minimizer of
$N_{\alpha,\beta,\lambda}$ can never be attained with $p=p_0,q=q_{0}.$ We only need to show
 the attainability of minimizers for $\alpha,\beta>0$ in here.
\end{rem}

In most previous literature, to prove the existence of extremal functions of reversed HLS-type inequalities, one usually used the rearrangement inequalities. However, the rearrangement inequalities are no longer valid to prove the existence of extremal functions for inequality \eqref{RHLSD-O} since vertical weights may cause there exists no radically symmetric minimizing sequence.

In this paper, we develop a renormalization method to deal with the situation where the presence of the factor $t^\alpha z^\beta$ in the numerator causes the failure of the rearrangement inequalities for inequalities \eqref{RHLSD-O}. This renormalization method is a rearrangement-free technique.
In order to prove Theorem \ref{theorem2}, we use some ideas of Dou, Guo, Zhu \cite{DGZ2020} and Han, Zhang \cite{HZ2022, ZH2022}.
While the works in \cite{DGZ2020, HZ2022,ZH2022} only considered the special case $p_0=q_0$ for $\alpha=\beta=0$ and $f=g\in L^{p_0}(\mathbb{R}^{n})$ in the proof of attainability of best constant for HLS inequality, one natural question remains open. Can we extend these results to general case $p_\alpha\neq q_\beta$, $f\in L^{p_\alpha}(\mathbb{R}^{n+1}_+)$ and $g\in L^{q_\beta}(\mathbb{R}^{n+1}_+)$ for any $0\le\alpha<-\frac{1}{p'_\alpha}$ and $0\le\beta<-\frac{1}{q'_\beta}$?

We completely answered this question by introducing some new ideas. We first study the extremal problems for subcritical exponents $p$ and $q$ and obtain the existence of extremal functions under subcritical case in a unit ball. We then pass to the limit to obtain the expected attainability of best constant with critical exponents by introducing the renormalization method. We think either the result and the idea of the proof may be applicable to other problems.

The Euler-Lagrange equation for extremal functions to inequality \eqref{RHLSD-O}, up to a constant multiplier, is given by
\begin{equation}\label{Eulereq0}
\begin{cases}
f^{p-1}(x,t)&=\int_{\mathbb{R}^{n+1}_+}\frac{t^\alpha z^\beta g(y,z)}{|(x,t)-(y,z)|^\lambda}dydz\ \quad(x,t)\in\mathbb{R}^{n+1}_+,\\
g^{q-1}(y,z)&=\int_{\mathbb{R}^{n+1}_+}\frac{t^\alpha z^\beta f(x,t)}{|(x,t)-(y,z)|^\lambda}dxdt\ \quad(y,z)\in\mathbb{R}^{n+1}_+.
\end{cases}
\end{equation}

To classify the extremal function of  inequality \eqref{RHLSD-1}, we discuss the regularity, radical symmetry of positive solutions to system \eqref{Eulereq0} as follows.

\begin{theorem}\label{regularity1}
Let $\alpha,\beta,\lambda,p,q$ satisfy  \eqref{RWH-exp-0}
and $(f,g)$ be a pair of positive Lebesgue measurable solution to \eqref{Eulereq0}. Then $f,g\in C^\infty(\mathbb{R}^{n+1}_+)\cap C^\gamma(\overline{\mathbb{R}^{n+1}_+})$ for $\gamma\in(0,1)$.
\end{theorem}

\begin{theorem}\label{theoremfenlei1}
Suppose the assumptions in Theorem \ref{regularity1} hold with $\alpha,\beta>0$, then
$f(x,t)$ and $g(x,t)$ are radically symmetric with respect to $x$ about some $x_0$.
Moreover, for $p=p_\alpha=\frac{2(n+1)}{2(n+1)+2\alpha-\lambda}$ and $q=q_{\beta}=\frac{2(n+1)}{2(n+1)+2\beta-\lambda}$, if $f,g\in C(\overline{\mathbb{R}^{n+1}_+})$, then there exist $c_1,c_2>0$ and $d>0$ such that
\begin{equation*}
f(x,0)=c_1(\frac{d}{1+d^{2}|x-\xi_0|^2})^{\frac{2(n+1)+2\alpha-\lambda}{2}},\quad g(x,0)=c_2(\frac{d}{1+d^{2}|x-\xi_0|^2})^{\frac{2(n+1)+2\beta-\lambda}{2}}
\end{equation*}
for some $\xi_0\in\partial\mathbb{R}^{n+1}_{+}$ and $x\in\partial\mathbb{R}^{n+1}_{+}$.
\end{theorem}


Note that equation \eqref{Eulereq0} is conformal invariant for $p=p_\beta$ and $q=q_\beta$, and hence one can employ
the method of moving spheres to classify Lebesgue measurable solutions $(f,g)$ of \eqref{Eulereq0}.

The method of moving spheres, introduced by Li and Zhu \cite{LZ1995}, and improved by Li  and Zhang \cite{LZ2003} and Li \cite{L2004}, is a variant of the well known method of moving planes. It has been widely used and has become a powerful tool to give the explicit form of solutions for some conformal invariant elliptic equations. For more results related to the method of moving spheres, we refer to \cite{DHL2021, DGZ2017, DZ2015a, DZ2015b, Gl2020, HL2021}.

By Theorem \ref{theorem2}, we deduce that the sufficient condition of the existence of positive solutions $(f,g)$ satisfying $(f,g)\in L^{p}(\mathbb{R}^{n+1}_+)\times L^{q}(\mathbb{R}^{n+1}_+)$ to system \eqref{Eulereq0} is
\begin{equation}\label{SC}
\frac 1p +\frac 1{q}=\frac{\alpha+\beta-\lambda}{n+1}+2.
\end{equation}

We will use the Pohozaev identities to prove that the sufficient condition \eqref{SC} is also a necessary condition for the existence of positive solutions to equation \eqref{Eulereq0}. In most previous literature, one usually needed to assume that the solutions $f$ and $g$ are $C^1$ and then investigate the necessary condition for the existence of positive solutions by establishing the Pohozaev identities. Here we introduce a new
idea, constructing some cut-off functions to obtain the Pohozaev identities in weak sense. This way, we only require $f$ and $g$ are Lebesgue measurable.

We first give the definition of weak positive solutions of equation \eqref{Eulereq0}.

\begin{df} We say that $(f,g)$  is a pair of weak positive solutions of equation \eqref{Eulereq0}, if  $f,g$ are the nonnegative Lebesgue measurable solutions, and satisfy
\begin{eqnarray*}
& &\int_{\mathbb{R}^{n+1}_+}f^{p-1}(x,t)\varphi(x,t)dxdt+ \int_{\mathbb{R}^{n+1}_+}g^{q-1}(x,t)\varphi(x,t)dxdt\nonumber\\
&=&\int_{\mathbb{R}^{n+1}_+}\int_{\mathbb{R}^{n+1}_+}
  \frac{t^{\alpha}z^{\beta} g(y,z)\varphi(x,t)}{|(x,t)-(y,z)|^\lambda }dydzdxdt
  +\int_{\mathbb{R}^{n+1}_+}\int_{\mathbb{R}^{n+1}_+}
  \frac{z^{\alpha}t^{\beta}f(y,z)\varphi(x,y)}{|(x,t)-(y,z)|^\lambda}dydz dxdt~~~~~~~~~~~~~~
\end{eqnarray*}
 for any $\varphi\in C^\infty_0(\mathbb{R}^{n+1}_+)$.
\end{df}

\begin{theorem}\label{theorem4}
For $-n-1<\lambda<0$, $\alpha,\beta\geq0$, $p,q\in (0,1)$, assume that there exists a pair of weak positive solutions $(f,g)$ satisfying \eqref{Eulereq0}. Then a necessary condition for $p$ and $q$ is
$$\frac 1p +\frac 1{q}=\frac{\alpha+\beta-\lambda}{n+1}+2.$$
\end{theorem}

As a consequence, we obtain the following Liouville type theorem for equation \eqref{Eulereq0}.
\begin{crl}\label{coro01}
For $-n-1<\lambda<0$, $\alpha,\beta\geq0$, $p,q\in (0,1)$, assume that
$$\frac 1p +\frac 1{q}\neq\frac{\alpha+\beta-\lambda}{n+1}+2,$$
then there does not exist a pair of positive Lebesgue measurable solutions $(f,g)$ satisfying \eqref{Eulereq0}.
\end{crl}

Let $(x,x')\in \mathbb{R}^{n+m}$ with $x\in\mathbb{R}^{n}, x'\in\mathbb{R}^{m}$, we also extend the reversed HLS inequality with general partial variable weights $|x'|$. Assume that $\alpha,\beta,\lambda,p,q$ satisfy
\begin{equation}\label{RWH-exp-gen0}
\begin{cases}
&-n-m<\lambda<0, 0< p,q<1,\\
&0\le\alpha<-\frac{m}{p'},0\le\beta<-\frac{m}{q'},\\
&\frac1p+\frac1{q}+\frac{\lambda-\alpha-\beta}{n+m}=2,\\
\end{cases}
\end{equation}
we have

\begin{theorem}\label{thmgen}
Let $\lambda,\alpha,\beta,p,q$ satisfy \eqref{RWH-exp-gen0}.
Then there exists a constant $N_{m,\alpha,\beta,\lambda}:=N(n,m,\alpha,\beta,\lambda, p)>0$ such that for any nonnegative functions $f\in L^p(\mathbb{R}^{n+m})$ and $g\in L^{q}(\mathbb{R}^{n+m})$,
\begin{equation}\label{gen01}
\int_{\mathbb{R}^{n+m}}\int_{\mathbb{R}^{n+m}}\frac{|x'|^\alpha |y'|^\beta g(x,x') f(y,y')}{|(x,x')-(y,y')|^{\lambda} } dxdx'dydy'
\geq N_{m,\alpha,\beta,\lambda}\|f\|_{L^p(\mathbb{R}^{n+m})}\|g\|_{L^{q}(\mathbb{R}^{n+m})}
\end{equation}
holds. Moreover, the upper and lower bounds of constant  $ N_{m,\alpha,\beta,\lambda}$ satisfy
\begin{eqnarray*}
(\frac{pq-p}{2pq-p-q})^{\frac{1-q}{q}}(\frac{pq-q}{2pq-p-q})^{\frac{1-p}{p}} \mbox{min}\{ D_1, D_2\} \le N_{m,\alpha,\beta,\lambda}\le \mbox{min}\{ D_1, D_2\},
\end{eqnarray*}
where
\begin{align*}
D_1&=\big[\frac{2\pi^{\frac{n+m}{2}}(p-1)}{(n+m+\alpha) p-n-m}\frac{\Gamma(\frac{(\alpha+m) p-m}{2(p-1)})}{\Gamma(\frac{(n+m+\alpha)p-n-m }{2(p-1)})\Gamma(\frac{m}{2})}\big]^{\frac{p-1}{p}}\times\\
&\qquad\big[\frac{2\pi^{\frac{n+m}{2}}(q-1)}{(n+m+\beta-\lambda) q-n-m}\frac{\Gamma(\frac{(\beta+m) q-m}{2(q-1)})}{\Gamma(\frac{(n+m+\beta )q-n-m}{2(q-1)})\Gamma(\frac{m}{2})}\big]^{\frac{q-1}{q}},
\end{align*}
\begin{align*}
D_2&=\big[\frac{2\pi^{\frac{n+m}{2}}(p-1)}{(n+m+\alpha-\lambda)p-n-m}\frac{\Gamma(\frac{(\alpha+m) p-m}{2(p-1)})}{\Gamma(\frac{(n+m+\alpha)p-n-m }{2(p-1)})\Gamma(\frac{m}{2})}\big]^{\frac{p-1}{p}}\times\\
&\qquad\big[\frac{2\pi^{\frac{n+m}{2}}(q-1)}{(n+m+\beta) q-n-m}\frac{\Gamma(\frac{(\beta+m) q-m}{2(q-1)})}{\Gamma(\frac{(n+m+\beta )q-n-m}{2(q-1)})\Gamma(\frac{m}{2})}\big]^{\frac{q-1}{q}}.
\end{align*}
\end{theorem}

For $f(x,x')\in C^\infty_0(\mathbb{R}^{n+m})$, we define
\[H_{\alpha,\beta}f(y,y')=\int_{\mathbb{R}^{n+m}}\frac{|x'|^\alpha |y'|^\beta  f(x,x')}{|(x,x')-(y,y')|^{\lambda}}dxdx'
\]
and
\begin{equation*}
R_\lambda f(y,y')=\int_{\mathbb{R}^{n+m}}\frac{f(x,x')}{|(x,x')-(y,y')|^\lambda}dxdx' \qquad\  ~~\forall\ (y,y')\in\mathbb{R}^{n+m}.
\end{equation*}

By duality and the reversed H\"{o}lder inequality, it is easy to see that the inequality \eqref{gen01} is equivalent to the following weighted HLS inequality with $r=q'$. For simplicity, let $\alpha,\beta,\lambda,p,r$ satisfy
\begin{equation}\label{gen02}
\begin{cases}
&-n-m<\lambda<0, r<0< p<1,\\
&0\le\alpha<-\frac{m}{p'},0\le\beta<-\frac{m}{r},\\
&\frac 1{r}=\frac{1}{p}-\frac{n+m+(\alpha+\beta-\lambda)}{n+m}.
\end{cases}
\end{equation}
\begin{crl}\label{thmgen}
Let $\alpha,\beta,\lambda,p,r$ satisfy \eqref{gen02}.
There exists a constant $C(n,m,\alpha,\beta,\lambda,p)>0$ depending on $n,m,\alpha,\beta, \lambda, p$ such that
\begin{equation*}
\| H_{\alpha,\beta}f\|_{L^{r}( \mathbb{R}^{n+m})}\ge C(n,m,\alpha,\beta,\lambda,p)\|f\|_{L^p(
\mathbb{R}^{n+m})}
\end{equation*}
or
\begin{equation*}
\|(R_\lambda f)|y'|^{\beta}\|_{L^{r}( \mathbb{R}^{n+m})}\ge C(n,m,\alpha,\beta,\lambda,p)\||x'|^{-\alpha} f\|_{L^p(
\mathbb{R}^{n+m})}
\end{equation*}
holds for any nonnegative function $f\in L^p(\mathbb{R}^{n+m}).$
\end{crl}

If $\alpha=\beta=0$, inequality \eqref{gen01} becomes inequality \eqref{DZ-reverse}. This work provides a different approach to establish  inequality \eqref{DZ-reverse}.

Define
\begin{align}\label{genmin}
N_{m,\alpha,\beta,\lambda}&=\inf\{\int_{\mathbb{R}^{n+m}}\int_{\mathbb{R}^{n+m}}\frac{|x'|^\alpha |y'|^{\beta} f(x,x')g(y,y')}{|(x,x')-(y,y')|^{\lambda}}dxdx'dydy': f\geq0, g\geq0,\nonumber\\
&\qquad\qquad \|f\|_{L^{p_\alpha}(\mathbb{R}^{n+m})}= \|g\|_{L^{q_\beta}(\mathbb{R}^{n+m})}=1\}\nonumber\\
&=\inf\{\frac{\int_{\mathbb{R}^{n+m}}\int_{\mathbb{R}^{n+m}}\frac{|x'|^\alpha |y'|^{\beta} f(x,x')g(y,y')}{|(x,x')-(y,y')|^{\lambda}}dxdx'dydy'}{\|f\|_{L^{p_\alpha}(\mathbb{R}^{n+m})} \|g\|_{L^{q_\beta}(\mathbb{R}^{n+m})}}: f\geq0, g\geq0, f\in L^{p_\alpha}(\mathbb{R}^{n+m}),g\in L^{q_\beta}(\mathbb{R}^{n+m})\},
\end{align}
where $p_\alpha=\frac{2(n+m)}{2(n+m)+2\alpha-\lambda}$ and $q_\beta=\frac{2(n+m)}{2(n+m)+2\beta-\lambda}$.

\medskip

Arguing as Theorem \ref{theorem2},  we can prove the attainability of minimizers for minimizing problem \eqref{genmin} using the renormalization method. Different from the proof of Theorem \ref{theorem2}, here we study the corresponding inequalities with subcritical exponents on the sphere $\mathbb{S}^{n+m}$ by using stereographic projection, which is equivalent to the inequality \eqref{gen01}. Then we prove the existence of extremal functions on the sphere $\mathbb{S}^{n+m}$ and hence obtain the existence of extremal functions of inequalities \eqref{gen01}.
The result is stated as follows.

\begin{theorem}\label{gentheorem2}
Let $\alpha,\beta,\lambda,p,q$ satisfy \eqref{RWH-exp-gen0}. $N_{m,\alpha,\beta,\lambda}$ is attained by a pair of positive functions $(f,g)\in L^{p_\alpha}(\mathbb{R}^{n+m})\times L^{q_\beta}(\mathbb{R}^{n+m})$ satisfying
$\|f\|_{L^{p_\alpha}(\mathbb{R}^{n+m})}=1$ and $\|g\|_{L^{q_\beta}(\mathbb{R}^{n+m})}=1$.
\end{theorem}

The Euler-Lagrange equation for extremal functions to inequality \eqref{gen01}, up to a constant multiplier, is given by
\begin{equation}\label{genEulereq0}
\begin{cases}
f^{p-1}(x,x')=\int_{\mathbb{R}^{n+m}}\frac{|x'|^\alpha |y'|^\beta g(y,y')}{|(x,x')-(y,y')|^\lambda}dydy'\ \quad(x,x')\in\mathbb{R}^{n+m},\\
g^{q-1}(y,y')=\int_{\mathbb{R}^{n+m}}\frac{|x'|^\alpha |y'|^\beta f(x,x')}{|(x,x')-(y,y')|^\lambda}dxdx'\ \quad(y,y')\in\mathbb{R}^{n+m}.
\end{cases}
\end{equation}

\begin{theorem}\label{gentheoremfenlei1}
Let $\alpha,\beta,\lambda,p,q$ satisfy  \eqref{RWH-exp-gen0}
and $(f,g)$ be a pair of positive Lebesgue measurable solutions to \eqref{genEulereq0}. Then $f,g\in C^\gamma(\mathbb{R}^{n+m})$ for $\gamma\in (0,1)$,
$f(x,x')$ and $g(x,x')$ are radially symmetric and monotone decreasing with respect to $x$ about some $x_0$ in  $\mathbb{R}^n,$ and are radially symmetric and monotone decreasing with respect to $x'$ about the origin in $\mathbb{R}^m$, respectively. That is, $f(x,x')=f(|x-x_0|,|x'|)$ and $g(x,x')=g(|x-x_0|,|x'|)$.
Moreover, for $p=p_\alpha=\frac{2(n+m)}{2(n+m)+2\alpha-\lambda}$ and $q_\beta=\frac{2(n+m)}{2(n+m)+2\beta-\lambda}$, if $f,g\in C(\mathbb{R}^{n+m})$, then there exist $c_1, c_2>0$ and $d>0$ such that
\begin{equation*}
f(x,0)=c_1(\frac{d}{d^{2}+|x-\xi_0|^2})^{\frac{2(n+m)+2\alpha-\lambda}{2}}, \ g(x,0)=c_2(\frac{d}{d^{2}+|x-\xi_0|^2})^{\frac{2(n+m)+2\beta-\lambda}{2}}
\end{equation*}
for some $\xi_0\in \mathbb{R}^{n}$ and  $x\in\mathbb{R}^{n}$.
\end{theorem}

Notice that we use $|x'|^\alpha$ and $|y'|^\beta$ instead of $x_{n+1}^\alpha$ and $y_{n+1}^\beta$ for inequalities \eqref{gen01} on $\mathbb{R}^{n+m}\times\mathbb{R}^{n+m}$, the methods in Section 2, Section 3 and  Section 4 are still holds here with small modifications, so we present only the results and omit the details of the proofs. 

This paper is organized as follows. In section 2, we establish the reverse weighted HLS inequality \eqref{RHLSD-O}. Section 3 is devoted to proving the existence of extremal function of the reverse weighted HLS inequality \eqref{RHLSD-O} by introducing the renormalization method.  In section 4, employing the method of moving spheres, we show the radial symmetry of positive solutions for equation \eqref{Eulereq0} and explicit form on the boundary $\partial\mathbb{R}^{n+1}_+=\mathbb{R}^{n}$ of  all extremal functions in the conformal invariant cases. In section 5, we obtain a necessary condition for the existence of positive solutions to equation \eqref{Eulereq0} in weak sense.

In all paper, positive constants are denoted by $c, C$ (with subcript in some cases) and are allowed to vary within a single line or formula.

\section{The rough reversed Hardy-Littlewood-Sobolev inequality with vertical weights\label{Section 2} }
In this section, we shall establish the rough reverse weighted  HLS inequalities on the upper half space by weighted Hardy inequality.

For $R>0$, we denote
$B_R(0)=\{(y,z)\in\mathbb{R}^{n}\times\mathbb{R}\,: \,|(y,z)|<R\} $ and $B_R^+(0)=\mathbb{R}^{n+1}_+\cap B_R(0) $.
The following lemma is crucial in our proof of Theorem \ref{RWHLSD-theo}, which can be proved by employing the same method as in \cite{AKM2019}.
\begin{lemma}\label{Weighted Hardy ineq}
Let $W$ and $U$ be nonnegative locally integrable weighted functions on $\mathbb{R}^{n+1}_+$. For $p\in (0,1)$ and $r<0$, the inequality
\begin{equation}\label{WH-1}
\big(\int_{\mathbb{R}^{n+1}_+}W(x,t)\big(\int_{ B_{|(x,t)|}^{+}(0)}f(y,z)dydz\big)^{r}dxdt\big)^\frac{1}{r}
\ge C_1(n,p,r)\big(\int_{\mathbb{R}^{n+1}_+}f^p(y,z)U(y,z)dydz\big)^\frac{1}{p}
\end{equation}
 holds for some $C_1(n, p,r)>0$ and all nonnegative measurable functions $f$, if and only if
\begin{equation*}
0<D_1=\inf_{(y,z)\in\mathbb{R}^{n+1}_+ }\big\{\big(\int_{\mathbb{R}^{n+1}_+\backslash B_{|(x,t)|}^{+}(0)}W(x,t)dxdt\big)^\frac {1}{r}\big(\int_{B_{|(x,t)|}^{+}(0)}U^{1-p'}(y,z)dydz\big)^\frac1{p'}\big\}.
\end{equation*}
Moreover, the biggest constant $C_1(n, p,r)>0$ in \eqref{WH-1} has the following relation to $D_1$
\[D_1\ge C_1(n, p, r)\ge (\frac{p'}{p'+r})^{-\frac1{r}}(\frac{r}{p'+r})^{-\frac1{p'}}D_1.
\]
On the other hand,
\begin{align}\label{WH-1-rev}
\big(\int_{\mathbb{R}^{n+1}_+}W(x,t)\big(\int_{\mathbb{R}^{n+1}_+\backslash {B_{|(x,t)|}^{+}(0)}}f(y,z)dydz\big)^{r}dxdt\big)^\frac{1}{r}
\ge C_2(n, p,r)\big(\int_{\mathbb{R}^{n+1}_+}f^p(y,z)U(y,z)dydz\big)^\frac1p
\end{align}
holds for some $C_2(n, p,r)>0$ and all nonnegative measurable functions $f,$ if and only if
\begin{align*}
0<D_2
=\inf_{(x,t)\in\mathbb{R}^{n+1}_+ }\big\{\big(\int_{B_{|(x,t)|}^{+}(0)}W(x,t)dxdt\big)^\frac{1}{r}\big(\int_{\mathbb{R}^{n+1}_+\backslash B_{|(x,t)|}^{+}(0)}U^{1-p'}(y,z)dydz\big)^\frac1{p'}\big\}.
\end{align*}
Moreover,  the biggest constant $C_2(n, p,r)>0$ in  \eqref{WH-1-rev} has the following relation to $D_2$,
\[D_2\ge C_2(n, p,r)\ge (\frac{p'}{p'+r})^{-\frac1{r}}(\frac{r}{p'+r})^{-\frac1{p'}}D_2.
\]
\end{lemma}

\textbf{Proof of Theorem \ref{RWHLSB}.} By reversed H\"{o}lder inequality, we have
\begin{eqnarray*}
\int_{\mathbb{R}^{n+1}_+}\int_{\mathbb{R}^{n+1}_+}\frac{t^\alpha z^\beta f(x,t)g(y,z)}{|(x,t)-(y,z)|^{\lambda}}dxdtdydz
&=&\int_{\mathbb{R}^{n+1}_+}(\int_{\mathbb{R}^{n+1}_+}\frac{t^\alpha z^\beta f(x,t)}{|(x,t)-(y,z)|^{\lambda}}dxdt) g(y,z)dydz\\
&\ge&(\int_{\mathbb{R}^{n+1}_+}(\int_{\mathbb{R}^{n+1}_+}\frac{t^\alpha z^\beta f(x,t)}{|(x,t)-(y,z)|^{\lambda}}dxdt)^{r}dydz)^{\frac1{r}}\|g\|_{L^{r'}(\mathbb{R}^{n+1}_+)}.
\end{eqnarray*}
Let $u(x,t)=t^\alpha f(x,t)$ and $r=q'<0$, then \eqref{RHLSD-1}  is equivalent to
\[\int_{\mathbb{R}^{n+1}_+}(\int_{\mathbb{R}^{n+1}_+} \frac{z^\beta u(x,t)}{|(x,t)-(y,z)|^{\lambda}}dxdt)^{r}dydz\le N_{\alpha,\beta,\lambda}^{r}\|t^{-\alpha}u\|^{r}_{L^p(
\mathbb{R}^{n+1}_+)}.\]
It is easy to see that
\[(\int_{\mathbb{R}^{n+1}_+}\frac{ z^\beta u(x,t)}{|(x,t)-(y,z)|^{\lambda}}dxdt)^{r}\le(\int_{B_{\frac{|(y,z)|}2}^+(0)}\frac{ z^\beta u(x,t)}{|(x,t)-(y,z)|^{\lambda}}dxdt)^{r}.\]
Therefore,
\begin{eqnarray}\label{RHLSD-2}
& &(\int_{\mathbb{R}^{n+1}_+}z^{\beta r}(\int_{\mathbb{R}^{n+1}_+} \frac{u(x,t)}{|(x,t)-(y,z)|^{\lambda}}dxdt)^{r}dydz)^{\frac1{r}}\nonumber\\
&\ge&(\int_{\mathbb{R}^{n+1}_+}z^{\beta r}(\int_{B_{\frac{|(y,z)|}2}^+(0)} \frac{u(x,t)}{|(x,t)-(y,z)|^{\lambda}}dxdt)^{r}dydz)^{\frac1{r}}\nonumber\\
&:=&I_1^{\frac1{r}}(u).
\end{eqnarray}
Similarly, we derive
\begin{eqnarray}\label{RHLSD-3}
& &(\int_{\mathbb{R}^{n+1}_+}z^{\beta r}(\int_{\mathbb{R}^{n+1}_+} \frac{u(x,t)}{|(x,t)-(y,z)|^{\lambda}}dxdt)^{r}dydz)^{\frac1{r}}\nonumber\\
&\ge&(\int_{\mathbb{R}^{n+1}_+}z^{\beta r}(\int_{\mathbb{R}^{n+1}_+\setminus B_{2|(y,z)|}^+(0)} \frac{u(x,t)}{|(x,t)-(y,z)|^{\lambda}}dxdt)^{r}dydz)^{\frac1{r}}\nonumber\\
&:=&I_2^{\frac1{r}}(u).
\end{eqnarray}
Combining \eqref{RHLSD-2} with \eqref{RHLSD-3}, we have
\[(\int_{\mathbb{R}^{n+1}_+}z^{\beta r}(\int_{\mathbb{R}^{n+1}_+}\frac{u(x,t)}{ |(x,t)-(y,z)|^{\lambda}}dxdt)^{r}dydz)^{\frac1{r}}\ge \frac{I_1^{\frac1{r}}(u)}{2}+\frac{I_2^{\frac1{r}}(u)}{2}.\]
Hence, we only need to show
\begin{eqnarray*}
I_i(u)\le N_{\alpha,\beta,\lambda}^{r}\|t^{-\alpha}u\|^{r}_{L^p(\mathbb{R}^{n+1}_+)},\quad i=1,2.
\end{eqnarray*}
We divide into two cases to discuss.

$(1)$ If $|(x,t)|\le \frac{|(y,z)|}2$, then $|(x,t)-(y,z)|\ge\frac{|(y,z)|}2$. For any $\lambda<0,$ we have
\[2^{\lambda}\int_{B_{\frac{|(y,z)|}2}^+(0)} \frac{u(x,t)}{|(y,z)|^{\lambda}}dxdt\le\int_{B_{\frac{|(y,z)|}2}^+(0)} \frac{u(x,t)}{|(x,t)-(y,z)|^{\lambda}}dxdt.\]
Since $r<0$, we obtain
\[(\int_{B_{\frac{|(y,z)|}2}^+(0)} \frac{u(x,t)}{|(x,t)-(y,z)|^{\lambda}}dxdt)^{r}\le 2^{\lambda r}(\int_{B_{\frac{|(y,z)|}2}^+(0)} \frac{u(x,t)}{|(y,z)|^{\lambda}}dxdt)^{r},\]
 and then
\begin{eqnarray}\label{theo001}
 I_1(u)
 &\le&2^{\lambda r}\int_{\mathbb{R}^{n+1}_+}\frac{z^{\beta r}}{|(y,z)|^{\lambda r}}(\int_{B_{\frac{|(y,z)|}2}^+(0)} u(x,t)dxdt)^{r}dydz.
\end{eqnarray}
Choosing $W(y,z)=z^{\beta r}|(y,z)|^{-\lambda r}$ and $U(x,t)=t^{-\alpha p}$ in \eqref{WH-1}, it follows from \eqref{theo001} that
\begin{eqnarray}\label{E1-es-1}
I_1(u)\le N_{\alpha,\beta,\lambda}^{r}\|t^{-\alpha} u\|^{r}_{L^p(\mathbb{R}^{n+1}_+)},
\end{eqnarray}
where $N_{\alpha,\beta,\lambda}$ can be seen as $C_1(n,p,r)$ in \eqref{WH-1}.
In fact,  from \eqref{RWH-exp-1} we know that
\begin{equation}\label{RWH-exp-1-eqv}
\frac {n+1}{r}=-(\alpha+\beta-\lambda)- \frac{n+1}{p'}.
\end{equation}
Since  $\alpha<-\frac{n+1}{p'}$,  we have $-\beta+\lambda<\frac {n+1}{r}$. That is, ${n+1}+(\beta-\lambda){r}<0.$  Since $\beta<-\frac{1}{r}$, 
we have
\begin{eqnarray}\label{WH-6}
\int_{\mathbb{R}^{n+1}_+\backslash B_{|(y,z)|}^+(0)}W(y,z)dydz&=&\int_{\mathbb{R}^{n+1}_+\backslash B_{|(y,z)|}^+(0)}z^{\beta r}|(y,z)|^{-\lambda r}dydz\nonumber\\
&=&J_\beta \int_{|(y,z)|}^\infty \rho^{(\beta-\lambda) r}\rho^{n}d\rho\nonumber\\
&=& C_3(n,\beta,\lambda,r)|(y,z)|^{n+1+(\beta-\lambda)r},
\end{eqnarray}
where $C_3(n,\beta,\lambda,r)=\frac{J_\beta}{n+1+(\beta-\lambda)r}$ and
\begin{eqnarray*}
J_\beta&=&\int_{0}^{\pi}(sin\theta_1)^{n-1+\beta r}d\theta_1\int_{0}^{\pi}(sin\theta_2)^{n-2+\beta r }d\theta_2\cdots\int_{0}^{\pi}(sin\theta_{n})^{\beta r} d\theta_{n}=\pi^{\frac{n}{2}}\frac{\Gamma(\frac{\beta r+1}{2})}{\Gamma(\frac{n+\beta r+1}{2})}.
\end{eqnarray*}

Moreover, since $\alpha<-\frac{n+1}{p'}$, we have $n+1+\alpha p'>0$. Since $\alpha<-\frac{1} {p'}$, one has
\begin{eqnarray}\label{WH-7}
\int_{B_{|(y,z)|}^+(0)}U^{1-p'}(x,t)dxdt&=&\int_{B_{|(y,z)|}^+(0)} t^{-\alpha p(1-p')}dxdt\nonumber\\
&=&J_\alpha\int_0^{|(y,z)|} \rho^{-\alpha p(1-p')}\rho^n d\rho\nonumber\\
&=&C_4(n,\alpha,p) |(y,z)|^{n+1+\alpha p'},
\end{eqnarray}
where $C_4(n,\alpha,p)=\frac{J_\alpha}{n+1+\alpha p'}$ and
\begin{eqnarray*}
J_\alpha&=&\int_{0}^{\pi}(sin\theta_1)^{n-1+\alpha p'}d\theta_1\int_{0}^{\pi}(sin\theta_2)^{n-2+\alpha p'}d\theta_2\cdots\int_{0}^{\pi}(sin\theta_{n})^{\alpha p'} d\theta_{n}=\pi^{\frac{n}{2}}\frac{\Gamma(\frac{\alpha p'+1}{2})}{\Gamma(\frac{n+\alpha p'+1}{2})}.
\end{eqnarray*}

The equalities \eqref{WH-6} and \eqref{WH-7} yield
\begin{eqnarray*}
D_1&=&\big(\int_{\mathbb{R}^{n+1}_+\backslash B_{|(y,z)|}^+(0)}W(y,z)dydz\big)^\frac1{r}\big(\int_{B_{|(y,z)|}^+(0)}U^{1-p'}(x,t)dxdt\big)^\frac1{p'}\\
&=&C_3^\frac1{r}(n,\beta,\lambda,r)C_4^\frac1{p'}(n,\alpha,p)|(y,z)|^{\frac{n+1} {r}+\frac{n+1}{p'}+(\alpha+\beta-\lambda)}\\
&=&C_3^\frac1{r}(n,\beta,\lambda,r)C_4^\frac1{p'}(n,\alpha,p),
\end{eqnarray*}
where $\frac{n+1} {r}+\frac{n+1}{p'}+(\alpha+\beta-\lambda)=0$ due to \eqref{RWH-exp-1}. Therefore, we prove \eqref{E1-es-1}.

$(2)$ If $|(y,z)|\le \frac{|(x,t)|}2$, then $|(x,t)-(y,z)|\ge\frac{|(x,t)|}2$. Similar to the case $I_1(u)$,  choosing $W(y,z)=z^{\beta r}$ and
$U(x,t)=t^{ -\alpha p}|(x,t)|^{\lambda p}$ in \eqref{WH-1-rev}, we have
\begin{eqnarray}\label{I2}
I_2(u)
&\le&
2^{\lambda r}\int_{\mathbb{R}^{n+1}_+}z^{\beta r}(\int_{\mathbb{R}^{n+1}_+\backslash B_{2|(y,z)|}^+(0)} \frac{u(x,t)}{|(x,t)|^\lambda}dxdt)^{r}dydz\nonumber\\
&\le&
N_{\alpha,\beta,\lambda}^{r}\|t^{-\alpha} u\|^{r}_{L^p(\mathbb{R}^{n+1}_+)},
\end{eqnarray}
where $N_{\alpha,\beta,\lambda}$ can be seen as $C_2(n,p,r)$ in \eqref{WH-1-rev}.
In fact, it follows from \eqref{RWH-exp-1-eqv} that $\frac{n+1}{p'}>\lambda-\alpha$ for any $\beta<-\frac{n+1}{r}$. That is, $n+1+(\alpha-\lambda)p'<0.$ Thus, we have
\begin{eqnarray}\label{WH-8}
\int_{\mathbb{R}^{n+1}_+\backslash B_{|(y,z)|}^+(0)}U^{1-p'}(x,t)dxdt&=&\int_{\mathbb{R}^{n+1}_+\backslash  B_{|(y,z)|}^+(0)}(t^{-\alpha}|(x,t)|^{ \lambda })^{p(1-p')}dxdt\nonumber\\
&=& J_\alpha\int_{|(y,z)|}^\infty \rho^{-(\alpha-\lambda)p(1-p')}\rho^nd\rho\nonumber\\
&=& C_5(n,\alpha,\lambda,p)|(y,z)|^{n+1+(\alpha-\lambda)p'}
\end{eqnarray}
and
\begin{eqnarray}\label{WH-9}
\int_{B_{|(y,z)|}^+(0)}W(y,z)dydz&=&\int_{ B_{|(y,z)|}^+(0)}z^{\beta r}dydz\nonumber\\
&=&J_\beta\int_0^{|y,z|} \rho^{\beta r}\rho^n d\rho\nonumber\\
&=&C_6(n,\beta,q) |(y,z)|^{n+1+\beta r},
\end{eqnarray}
where $C_5(n,\alpha,\lambda,p)=\frac{J_\alpha}{n+1+(\alpha-\lambda)p'}$ and $C_6(n,\beta,q)=\frac{J_\beta}{n+1+\beta r}$.
Combining \eqref{WH-8} with \eqref{WH-9}, we show that \eqref{WH-1-rev} holds and hence \eqref{I2} is proved.

Finally,  we show the estimate of constant $N_{\alpha,\beta,\lambda}$. From
 Lemma \ref{Weighted Hardy ineq}, we know
$$D_1=C_3^\frac1{r}(n,\beta,\lambda,r)C_4^\frac1{p'}(n,\alpha,p)\ \mbox{and}\ D_2=C_5^\frac1{p'}(n,\alpha,\lambda,p)C_6^\frac1{r}(n,\beta,q).$$
Hence, we give the upper and lower bounds of constant $N_{\alpha,\beta,\lambda}$ as follows
\begin{eqnarray*}
(\frac{p'}{p'+r})^{-\frac1{r}}(\frac{r}{p'+r})^{-\frac1{p'}}\mbox{min}\{ D_1, D_2\} \le N_{\alpha,\beta,\lambda}\le \mbox{min}\{ D_1, D_2\},
\end{eqnarray*}
where $$D_1=\big[\frac{\pi^{\frac{n}{2}}}{n+1+\alpha p'}\frac{\Gamma(\frac{\alpha p'+1}{2})}{\Gamma(\frac{n+\alpha p'+1}{2})}\big]^{\frac{1}{p'}}\big[\frac{\pi^{\frac{n}{2}}}{n+1+(\beta-\lambda) r}\frac{\Gamma(\frac{\beta r+1}{2})}{\Gamma(\frac{n+\beta r+1}{2})}\big]^{\frac{1}{r}},$$
$$D_2=\big[\frac{\pi^{\frac{n}{2}}}{n+1+(\alpha-\lambda)p'}\frac{\Gamma(\frac{\alpha p'+1}{2})}{\Gamma(\frac{n+\alpha p'+1}{2})}\big]^{\frac{1}{p'}}\big[\frac{\pi^{\frac{n}{2}}}{n+1+\beta r}\frac{\Gamma(\frac{\beta r+1}{2})}{\Gamma(\frac{n+\beta r+1}{2})}\big]^{\frac{1}{r}}.$$
This completes the proof of Theorem \ref{RWHLSB}.
 \hfill$\Box$

\section{Sharp reversed Hardy-Littlewood-Sobolev inequality with vertical weights\label{Section 3} }

In this section, we prove the existence of extremal functions for inequality \eqref{RHLSD-O} by introducing a renormalization method.

Define conformal transformation $\mathcal{T}: \zeta\in B^{n+1}\rightarrow \zeta^{x^0}\in\mathbb{R}^{n+1}_+$ given by
\begin{equation}\label{cftrans}
\zeta^{x^0}:=\frac{2^2(\zeta-x^0)}{|\zeta-x^0|^2}+x^0\in\mathbb{R}^{n+1}_+,
\end{equation}
where $x^0=(0,-2)\in\mathbb{R}^n\times(0,-\infty)=:\mathbb{R}^{n+1}_-$ and $B^{n+1}=B(x^1,1)$ with $x^1=(0,-1)\in\mathbb{R}^{n+1}_-$ and radius $1$. Let $\mathcal{T}^{-1}: \mathbb{R}^{n+1}_+ \rightarrow B^{n+1}$ be the inverse of $\mathcal{T}$.

For conformal case $p_\alpha=\frac{2(n+1)}{2(n+1)+2\alpha-\lambda}$ and $q_\beta=\frac{2(n+1)}{2(n+1)+2\beta-\lambda}$, it is easy to verify that the inequality \eqref{RHLSD-O} is equivalent to the following integral inequality on the ball $B^{n+1}$:
\begin{align}\label{02}
&\int_{B^{n+1}}\int_{B^{n+1}}\big(\frac{1-|\zeta-x^1|^2}{2}\big)^\alpha\big(\frac{1-|\eta-x^1|^2}{2}\big)^\beta\frac{ f(\zeta)g(\eta)}{|\zeta-\eta|^\lambda}d\zeta d\eta\geq N_{\alpha,\beta,\lambda} \|f\|_{L^{p_\alpha}(B^{n+1})}\|g\|_{L^{q_\beta}(B^{n+1})}.
\end{align}

The sharp constant to inequality \eqref{02} is classified by
\begin{align*}
N_{\alpha,\beta,\lambda}&=\inf_{\substack{\|f\|_{L^{p_\alpha}(B^{n+1})}=1\\ \|g\|_{L^{q_\beta}(B^{n+1})}=1}}\int_{B^{n+1}}\int_{B^{n+1}}\big(\frac{1-|\zeta-x^1|^2}{2}\big)^\alpha
\big(\frac{1-|\eta-x^1|^2}{2}\big)^\beta\frac{ f(\zeta)g(\eta)}{|\zeta-\eta|^{\lambda}}d\zeta d\eta\nonumber\\
&=\inf_{\substack{f\in L^{p_\alpha}(B^{n+1}),\\g\in L^{q_\beta}(B^{n+1})}}\frac{\int_{B^{n+1}}\int_{B^{n+1}}(\frac{1-|\zeta-x^1|^2}{2})^\alpha
(\frac{1-|\eta-x^1|^2}{2})^\beta\frac{ f(\zeta)g(\eta)}{|\zeta-\eta|^{\lambda}}d\zeta d\eta}{\|f\|_{L^{p_\alpha}(B^{n+1})} \|g\|_{L^{q_\beta}(B^{n+1})}}.
\end{align*}

\begin{theorem}\label{equi01'}
$N_{\alpha,\beta,\lambda}$  is attained by a pair of positive functions $(f,g)\in L^{p_\alpha}(B^{n+1})\times L^{q_\beta}(B^{n+1})$ satisfying $\|f\|_{L^{p_\alpha}(B^{n+1})}=1$ and $\|g\|_{L^{q_\beta}(B^{n+1})}=1$, respectively.
\end{theorem}

\subsection{\textbf{Subcritical reverse weighted HLS inequality on the ball}}
\

In this subsection, we first establish the subcritical reverse weighted HLS inequality on the ball $B^{n+1}$. The existence of the corresponding extremal functions is also established.

\begin{theorem}\label{equi1}
Assume $\lambda\in(-n-1,0)$, $p\in (0,p_\alpha)$ and $q\in (0,q_\beta)$. Then there exists a constant $C_{n,\alpha,\beta,\lambda,p}>0$ such that
\begin{align}\label{2}
&\int_{B^{n+1}}\int_{B^{n+1}}\big(\frac{1-|\zeta-x^1|^2}{2}\big)^\alpha\big(\frac{1-|\eta-x^1|^2}{2}\big)^\beta\frac{ f(\zeta)g(\eta)}{|\zeta-\eta|^\lambda}d\zeta d\eta\geq C_{n,\alpha,\beta,\lambda,p} \|f\|_{L^p(B^{n+1})}\|g\|_{L^{q}(B^{n+1})}
\end{align}
holds for any nonnegative functions  $f\in L^p(B^{n+1})$ and $g\in L^{q}(B^{n+1})$.
\end{theorem}

For $\eta\in B^{n+1}$, we introduce the operator
\begin{align*}
T_{\alpha,\beta,p,q}f(\eta)&=\int_{B^{n+1}}\big(\frac{1-|\zeta-x^1|^2}{2}\big)^\alpha\big(\frac{1-|\eta-x^1|^2}{2}\big)^\beta
\frac{ f(\zeta)}{|\zeta-\eta|^{\lambda}}d\zeta.
\end{align*}
It is easy to see that the inequality \eqref{2} is equivalent to
\begin{align}\label{critical}
&\|T_{\alpha,\beta,p,q}f\|_{L^{q'}(B^{n+1})}\geq C_{n,\alpha,\beta,\lambda,p} \|f\|_{L^p(B^{n+1})}
\end{align}
for any nonnegative function $f\in L^p(B^{n+1})$.

Now we consider the extremal problem of inequality \eqref{2}:
\begin{align}\label{3}
C^*_{n,\alpha,\beta,\lambda,p}&=\inf\{\int_{B^{n+1}}\int_{B^{n+1}}(\frac{1-|\zeta-x^1|^2}{2}\big)^\alpha(\frac{1-|\eta-x^1|^2}{2}\big)^\beta
\frac{ f(\zeta)g(\eta)}{|\zeta-\eta|^{\lambda}}d\zeta d\eta : f\geq0, g\geq0,\nonumber\\
&\qquad\qquad\|f\|_{L^p(B^{n+1})}= \|g\|_{L^{q}(B^{n+1})}=1\}\nonumber\\
&=\inf\{\frac{\int_{B^{n+1}}\int_{B^{n+1}}(\frac{1-|\zeta-x^1|^2}{2})^\alpha(\frac{1-|\eta-x^1|^2}{2})^\beta\frac{ f(\zeta)g(\eta)}{|\zeta-\eta|^{\lambda}}d\zeta d\eta}{\|f\|_{L^p(B^{n+1})}\|g\|_{L^{q}(B^{n+1})}}:f\geq0,g\geq0,\nonumber\\
&\qquad\qquad f\in L^p(B^{n+1}),g\in L^{q}(B^{n+1})\}.
\end{align}
It follows from \eqref{2} and \eqref{3} that $C^*_{n,\alpha,\beta,\lambda,p}\geq C_{n,\alpha,\beta,\lambda,p}>0$.

The following proposition plays a crucial role in our proof of existence of the minimizer. The method we shall use is similar to that of \cite{DGZ2017} and \cite{DGZ2020}, but we need to
introduce some new ideas such that this method can be applied to reverse weighted HLS inequality.

\begin{prop}\label{pro1}
\ \ $(i)$ There exists a pair of nonnegative functions $(f,g)\in L^1(B^{n+1})\times L^1(B^{n+1}) $ such that $\|f\|_{L^p(B^{n+1})}=\|g\|_{L^{q}(B^{n+1})}=1$ and
\begin{align*}
C^*_{n,\alpha,\beta,\lambda,p}&=\int_{B^{n+1}}\int_{B^{n+1}}\big(\frac{1-|\zeta-x^1|^2}{2}\big)^\alpha\big(\frac{1-|\eta-x^1|^2}{2}\big)^\beta
\frac{ f(\zeta)g(\eta)}{|\zeta-\eta|^{\lambda}}d\zeta d\eta.
\end{align*}
\ \ \ $(ii)$ For $\zeta\in \overline{B^{n+1}}$, the functions $f,g$ satisfy the Euler-Lagrange equations
\begin{equation}\label{integral}\begin{cases}
C^*_{n,\alpha,\beta,\lambda,p}f^{p-1}(\zeta)=\int_{B^{n+1}}(\frac{1-|\zeta-x^1|^2}{2}\big)^\alpha(\frac{1-|\eta-x^1|^2}{2})^\beta
\frac{g(\eta)}{|\zeta-\eta|^{\lambda}} d\eta,\\
C^*_{n,\alpha,\beta,\lambda,p}g^{q-1}(\zeta)=\int_{B^{n+1}}(\frac{1-|\eta-x^1|^2}{2})^\alpha(\frac{1-|\zeta-x^1|^2}{2})^\beta
\frac{ f(\eta)}{|\zeta-\eta|^{\lambda}} d\eta.
\end{cases}\end{equation}
\ \ \ $(iii)$ There exists a constant $C=C(n,\alpha,\beta,\lambda,p,q)>0$ such that
\begin{equation}\label{bdd1}
\frac{1}{C}\leq f,g\leq C.
\end{equation}
Furthermore, $f, g\in C^\gamma(\overline{B^{n+1}})$ with $\gamma\in(0,1)$.
\end{prop}
\begin{proof}
${\bf (i).}$  We first show $C^*_{n,\alpha,\beta,\lambda,p}$ is attained by a pair of nonnegative functions $(f,g)\in L^1(B^{n+1})\times L^1(B^{n+1})$.

By density argument, we choose a pair of nonnegative minimizing sequence $\{f_j,g_j\}_{j=1}^{+\infty}\in C^\infty(B^{n+1})\times C^\infty(B^{n+1})$ such that
$$\|f_j\|_{L^p(B^{n+1})}=\|g_j\|_{L^{q}(B^{n+1})}=1$$
and
\begin{align*}
C^*_{n,\alpha,\beta,\lambda,p}&=\lim_{j\rightarrow+\infty}\int_{B^{n+1}}\int_{B^{n+1}}\big(\frac{1-|\zeta-x^1|^2}{2}\big)^\alpha
\big(\frac{1-|\eta-x^1|^2}{2}\big)^\beta\frac{f_j(\zeta)g_j(\eta)}{|\zeta-\eta|^{\lambda}}d\zeta d\eta,\ \ j=1,2,....
\end{align*}
We carry out the proof of part (i) in three steps.

\ \ {\bf Step 1.} We show that
\begin{equation}\label{4}
\|f_j\|_{L^1(B^{n+1})}\leq C_1,\ \|g_j\|_{L^{1}(B^{n+1})}\leq C_2\ \ \mbox{uniformaly}.
\end{equation}
In fact, we know from \eqref{bestrange} that there exist constants $C_3$ and $C_4$ such that
\begin{align*}
0<&C_3\leq \int_{B^{n+1}}\int_{B^{n+1}}\big(\frac{1-|\zeta-x^1|^2}{2}\big)^\alpha
\big(\frac{1-|\eta-x^1|^2}{2}\big)^\beta
\frac{f_j(\zeta)g_j(\eta)}{|\zeta-\eta|^{\lambda}}d\zeta d\eta\leq C_4<\infty.
\end{align*}
Then it follows from reversed H\"{o}lder inequality that
$$\|T_{\alpha,\beta,p,q}f_j\|_{L^{q'}(B^{n+1})}=\|T_{\alpha,\beta,p,q}f_j\|_{L^{q'}(B^{n+1})}\|g_j\|_{L^{q}(B^{n+1})}\leq C_4,$$
$$\|T_{\alpha,\beta,p,q}g_j\|_{L^{p'}(B^{n+1})}=\|T_{\alpha,\beta,p,q}g_j\|_{L^{p'}(B^{n+1})}\|f_j\|_{L^{p}(B^{n+1})}\leq C_4.$$
Since $q'_\beta<q'<0$, we derive that for some constant $M>0$ (to be determined later),
\begin{align}\label{5}
C_4^{q'}&\leq \int_{B^{n+1}}|T_{\alpha,\beta,p,q}f_j|^{q'} d\eta\nonumber\\
&=\int_{\{T_{\alpha,\beta,p,q}f_j\geq M\}}|T_{\alpha,\beta,p,q}f_j|^{q'} d\eta+\int_{\{T_{\alpha,\beta,p,q}f_j< M\}}|T_{\alpha,\beta,p,q}f_j|^{q'} d\eta\nonumber\\
&\leq M^{q'} |B^{n+1}|+|\{T_{\alpha,\beta,p,q}f_j< M\}|^{1-\frac{q'}{q'_\beta}}\|T_{\alpha,\beta,p,q}f_j\|^{q'}_{L^{q'_\beta}(B^{n+1})},
\end{align}
where $\frac{1}{q'_\beta}+\frac{1}{q_\beta}=1$.
Using \eqref{critical} and reversed H\"{o}lder inequality, we have
\begin{align}\label{5'}
\|T_{\alpha,\beta,p_\alpha,q_\beta}f_j\|_{L^{q'_\beta}(B^{n+1})}&\geq C_5\|f_j\|_{L^{p_\alpha}(B^{n+1})}\nonumber\\
&\geq C_5|B^{n+1}|^{\frac{1}{p_\alpha}-\frac{1}{p}}\|f_j\|_{L^{p}(B^{n+1})}=C_5|B^{n+1}|^{\frac{1}{p_\alpha}-\frac{1}{p}}.
\end{align}

Now we choose $M>0$ such that $M^{q'}|B^{n+1}|=\frac{C_4^{q'}}{2}$. Combining \eqref{5} with \eqref{5'} yields
\begin{align*}
\frac{C_4^{q'}}{2} &\leq \big(C_5|B^{n+1}|^{\frac{1}{p_\alpha}-\frac{1}{p}}\big)^{q'}|\{T_{\alpha,\beta,p,q}f_j< M\}|^{1-\frac{q'}{q'_\beta}},
\end{align*}
which implies
$$|\{T_{\alpha,\beta,p,q}f_j< M\}|\geq \big(\frac{C_4}{C_5|B^{n+1}|^{\frac{1}{p_\alpha}-\frac{1}{p}}}\big)^{\frac{q'_\beta q'}{q'_\beta-q'}}>0.$$
Hence, there exists $\epsilon_0>0$ such that for any $j$, we can find two points $\zeta_j^1,\zeta_j^2\in \Omega_1=\{\zeta: T_{\alpha,\beta,p,q}f_j< M\}$ satisfying
$|\zeta_j^1-\zeta_j^2|\geq \epsilon_0$.
Then we have
\begin{align*}
& \int_{B^{n+1}}f_j(\eta)d\eta \leq C_6\int_{B_{1-\frac{\epsilon_0}{4}}(x_1)\setminus B_{\frac{\epsilon_0}{4}}(\zeta_j^1)}f_j(\eta)d\eta+C_6\int_{B_{1-\frac{\epsilon_0}{4}}(x_1)\setminus B_{\frac{\epsilon_0}{4}}(\zeta_j^2)}f_j(\eta)d\eta\nonumber\\
\leq& C_7\big(\frac{1}{\epsilon_0}\big)^{\alpha+\beta-\lambda}\int_{B_{1-\frac{\epsilon_0}{4}}(x_1)\setminus B_{\frac{\epsilon_0}{4}}(\zeta_j^1)}\big(\frac{1-|\zeta_j^1-x^1|^2}{2}\big)^\alpha
 \big(\frac{1-|\eta-x^1|^2}{2}\big)^\beta
\frac{ f_j(\eta)}{|\zeta_j^1-\eta|^{\lambda}}d\eta\nonumber\\
&+C_7\big(\frac{1}{\epsilon_0}\big)^{\alpha+\beta-\lambda}\int_{B_{1-\frac{\epsilon_0}{4}}(x_1)\setminus B_{\frac{\epsilon_0}{4}}(\zeta_j^2)}\big(\frac{1-|\zeta_j^1-x^1|^2}{2}\big)^\alpha
 \big(\frac{1-|\eta-x^1|^2}{2}\big)^\beta
\frac{f_j(\eta)}{|\zeta_j^2-\eta|^{\lambda}}d\eta\nonumber\\
\leq& 2C_7M\big(\frac{1}{\epsilon_0}\big)^{\alpha+\beta-\lambda}\quad\  \ \ \ \mbox{uniformly for all}\ j.
\end{align*}

We can find two points $\eta_j^1,\eta_j^2\in \Omega_1=\{\eta: T_{\alpha,\beta,p,q}f_j< M\}$ satisfying
$|\eta_j^1-\eta_j^2|\geq \epsilon_0$.
Then we have
\begin{align*}
& \int_{B^{n+1}}f_j(\zeta)d\zeta \leq C_6\int_{B_{1-\frac{\epsilon_0}{4}}(x_1)\setminus B_{\frac{\epsilon_0}{4}}(\eta_j^1)}f_j(\zeta)d\zeta+C_6\int_{B_{1-\frac{\epsilon_0}{4}}(x_1)\setminus B_{\frac{\epsilon_0}{4}}(\eta_j^2)}f_j(\zeta)d\zeta\nonumber\\
\leq& C_7\big(\frac{1}{\epsilon_0}\big)^{\alpha+\beta-\lambda}\int_{B_{1-\frac{\epsilon_0}{4}}(x_1)\setminus B_{\frac{\epsilon_0}{4}}(\eta_j^1)}
 \big(\frac{1-|\zeta-x^1|^2}{2}\big)^\alpha\big(\frac{1-|\eta_j^1-x^1|^2}{2}\big)^\beta
\frac{ f_j(\zeta)}{|\zeta-\eta_j^1|^{\lambda}}d\zeta\nonumber\\
&+C_7\big(\frac{1}{\epsilon_0}\big)^{\alpha+\beta-\lambda}\int_{B_{1-\frac{\epsilon_0}{4}}(x_1)\setminus B_{\frac{\epsilon_0}{4}}(\eta_j^2)}
 \big(\frac{1-|\zeta-x^1|^2}{2}\big)^\alpha\big(\frac{1-|\eta_j^2-x^1|^2}{2}\big)^\beta
\frac{f_j(\zeta)}{|\zeta-\eta_j^2|^{\lambda}}d\zeta\nonumber\\
\leq& 2C_7M\big(\frac{1}{\epsilon_0}\big)^{\alpha+\beta-\lambda}\quad\  \ \ \ \mbox{uniformly for all}\ j.
\end{align*}
It follows $\|f_j\|_{L^1(B^{n+1})}\leq C_1$. Similarly, we have $\|g_j\|_{L^{1}(B^{n+1})}\leq C_2$.

{\bf Step 2.} There exist two subsequences of $\{f_j^p\}$ and $\{g_j^{p}\}$ (still denoted by $\{f_j^p\}$ and $\{g_j^{p}\}$) and two nonnegative functions $f,g\in L^{1}(B^{n+1})$ such that
\begin{align}\label{8}
&\int_{B^{n+1}}|f_j(\zeta)|^pd\zeta\rightarrow \int_{B^{n+1}}|f(\zeta)|^pd\zeta,\ \ \int_{B^{n+1}}|g_j(\zeta)|^{q}d\zeta\rightarrow \int_{B^{n+1}}|g(\zeta)|^{q}d\zeta\ \ \ \mbox{as}\ j\rightarrow+\infty.
\end{align}

Without loss of generality, we assume that $p\leq q$. By \eqref{4}, we infer that there exist two subsequences of $\{f_j^p\}$ and $\{g_j^{p}\}$ (still denoted by $\{f_j^p\}$ and $\{g_j^{p}\}$) and two nonnegative functions $f,g\in L^{1}(B^{n+1})$ such that
$$f_j^p\rightharpoonup f^p,\ \ \ \ g_j^{p}\rightharpoonup g^{p} \ \mbox{weakly\  in}\  L^{\frac{1}{p}}(B^{n+1}).$$
Since $1\in L^{\frac{1}{1-p}}(B^{n+1})$ and  $g_j^{q-p}\in L^{\frac{1}{q-p}}(B^{n+1})\subset L^{\frac{1}{1-p}}(B^{n+1})$, we immediately derive \eqref{8}.

{ \bf Step 3:} We show that
\begin{align*}
&\int_{B^{n+1}}\int_{B^{n+1}}\big(\frac{1-|\zeta-x^1|^2}{2}\big)^\alpha\big(\frac{1-|\eta-x^1|^2}{2}\big)^\beta
\frac{ f_j(\zeta)g_j(\eta)}{|\zeta-\eta|^{\lambda}}d\zeta d\eta\nonumber\\
&\rightarrow\int_{B^{n+1}}\int_{B^{n+1}}\big(\frac{1-|\zeta-x^1|^2}{2}\big)^\alpha\big(\frac{1-|\eta-x^1|^2}{2}\big)^\beta
\frac{f(\zeta)g(\eta)}{|\zeta-\eta|^{\lambda}}d\zeta d\eta\ \ \ \mbox{as}\ j\rightarrow\infty.
\end{align*}

By \eqref{4} and the interpolation inequality, we have
$$\int_{B^{n+1}}|f_j(\zeta)|^pd\zeta\geq C>0\ \ \mbox{and}\ \int_{B^{n+1}}|g_j(\zeta)|^{q}d\zeta\geq C>0.$$
This implies that
$$\int_{B^{n+1}}|f(\zeta)|^pd\zeta\geq C>0\ \ \mbox{and}\ \int_{B^{n+1}}|g(\zeta)|^{q}d\zeta\geq C>0.$$
Then, we have, for any fixed $\eta\in \overline{B^{n+1}}$, that
$$\big(\frac{1-|\zeta-x^1|^2}{2}\big)^\alpha\big(\frac{1-|\eta-x^1|^2}{2}\big)^\beta
\frac{ f^{1-p}(\zeta)}{|\zeta-\eta|^{\lambda}}\in L^{\frac{1}{1-p}}(B^{n+1}),$$
and for any fixed  $\zeta\in \overline{B^{n+1}}$
$$\big(\frac{1-|\zeta-x^1|^2}{2}\big)^\alpha\big(\frac{1-|\eta-x^1|^2}{2}\big)^\beta
\frac{ g^{1-p}(\eta)}{|\zeta-\eta|^{\lambda}}\in L^{\frac{1}{1-p}}(B^{n+1}).$$
Thus, one can conclude that
\begin{align*}
&\int_{B^{n+1}}\big(\frac{1-|\zeta-x^1|^2}{2}\big)^\alpha\big(\frac{1-|\eta-x^1|^2}{2}\big)^\beta
\frac{f_j^p(\zeta) f^{1-p}(\zeta)}{|\zeta-\eta|^{\lambda}}d\zeta\nonumber\\
&\rightarrow\int_{B^{n+1}}\big(\frac{1-|\zeta-x^1|^2}{2}\big)^\alpha\big(\frac{1-|\eta-x^1|^2}{2}\big)^\beta
\frac{f(\zeta)}{|\zeta-\eta|^{\lambda}}d\zeta\ \ \ \mbox{as}\ j\rightarrow\infty,\nonumber\\
&\int_{B^{n+1}}\big(\frac{1-|\zeta-x^1|^2}{2}\big)^\alpha\big(\frac{1-|\eta-x^1|^2}{2}\big)^\beta
\frac{g_j^p(\eta) g^{1-p}(\eta)}{|\zeta-\eta|^{\lambda}}d\eta\nonumber\\
&\rightarrow\int_{B^{n+1}}\big(\frac{1-|\zeta-x^1|^2}{2}\big)^\alpha\big(\frac{1-|\eta-x^1|^2}{2}\big)^\beta
\frac{g(\eta)}{|\zeta-\eta|^{\lambda}} d\eta\ \ \ \mbox{as}\ j\rightarrow\infty.
\end{align*}
Using similar argument as Lemma 3.2 in \cite{DGZ2020}, we also show that the above convergences are uniformly convergent for all $\zeta,\eta\in \overline{B^{n+1}}$.

Therefore, for any $\epsilon>0$ small enough, there exists $j_0\in N$ such that for all $j>j_0$,
\begin{align}\label{newinqeu}
&\big|\int_{B^{n+1}}\int_{B^{n+1}}\big(\frac{1-|\zeta-x^1|^2}{2}\big)^\alpha\big(\frac{1-|\eta-x^1|^2}{2}\big)^\beta
\frac{f_j^p(\zeta) f^{1-p}(\zeta)g_j^p(\eta) g^{1-p}(\eta)}{|\zeta-\eta|^{\lambda}}d\zeta d\eta-\nonumber\\
&\int_{B^{n+1}}\int_{B^{n+1}}\big(\frac{1-|\zeta-x^1|^2}{2}\big)^\alpha\big(\frac{1-|\eta-x^1|^2}{2}\big)^\beta
\frac{f_j^p(\zeta) f^{1-p}(\zeta)g(\eta)}{|\zeta-\eta|^{\lambda}}d\zeta d\eta\big|\nonumber\\
&\leq \epsilon\int_{B^{n+1}} f_j^p(\zeta) f^{1-p}(\zeta) d\zeta\leq \epsilon C.
\end{align}
Here we obtain the last inequality by using H\"{o}lder inequality.
Notice that $f^{1-p}\in L^{\frac{1}{1-p}}(B^{n+1})$ and
$$\int_{B^{n+1}}\big(\frac{1-|\zeta-x^1|^2}{2}\big)^\alpha\big(\frac{1-|\eta-x^1|^2}{2}\big)^\beta
\frac{ g(\eta)}{|\zeta-\eta|^{\lambda}}d\eta\leq C\int_{B^{n+1}}g(\eta)d\eta \leq C,$$
one has
\begin{align*}
&\int_{B^{n+1}}\int_{B^{n+1}}\big(\frac{1-|\zeta-x^1|^2}{2}\big)^\alpha\big(\frac{1-|\eta-x^1|^2}{2}\big)^\beta
\frac{f_j^p(\zeta) f^{1-p}(\zeta)g(\eta)}{|\zeta-\eta|^{\lambda}}d\zeta d\eta\nonumber\\
&\rightarrow\int_{B^{n+1}}\int_{B^{n+1}}\big(\frac{1-|\zeta-x^1|^2}{2}\big)^\alpha\big(\frac{1-|\eta-x^1|^2}{2}\big)^\beta
\frac{f(\zeta)g(\eta)}{|\zeta-\eta|^{\lambda}} d\zeta d\eta\ \ \ \mbox{as}\ j\rightarrow\infty.
\end{align*}
It follows from the above convergence and \eqref{newinqeu} that
\begin{align*}
&\int_{B^{n+1}}\int_{B^{n+1}}\big(\frac{1-|\zeta-x^1|^2}{2}\big)^\alpha\big(\frac{1-|\eta-x^1|^2}{2}\big)^\beta
\frac{f_j^p(\zeta) f^{1-p}(\zeta)g_j^p(\eta) g^{1-p}(\eta)}{|\zeta-\eta|^{\lambda}}d\zeta d\eta\nonumber\\
&\rightarrow\int_{B^{n+1}}\int_{B^{n+1}}\big(\frac{1-|\zeta-x^1|^2}{2}\big)^\alpha\big(\frac{1-|\eta-x^1|^2}{2}\big)^\beta
\frac{f(\zeta)g(\eta)}{|\zeta-\eta|^{\lambda}}d\zeta d\eta\ \ \ \mbox{as}\ j\rightarrow\infty.
\end{align*}
By H\"{o}lder inequality, we have
\begin{align*}
&\int_{B^{n+1}}\int_{B^{n+1}}\big(\frac{1-|\zeta-x^1|^2}{2}\big)^\alpha\big(\frac{1-|\eta-x^1|^2}{2}\big)^\beta
\frac{f_j^p(\zeta) f^{1-p}(\zeta)g_j^p(\eta) g^{1-p}(\eta)}{|\zeta-\eta|^{\lambda}}d\zeta d\eta\nonumber\\
&\leq \big(\int_{B^{n+1}}\int_{B^{n+1}}\big(\frac{1-|\zeta-x^1|^2}{2}\big)^\alpha\big(\frac{1-|\eta-x^1|^2}{2}\big)^\beta
\frac{f_j(\zeta) g_j(\eta)}{|\zeta-\eta|^{\lambda}}d\zeta d\eta \big)^p \times\nonumber\\
&\quad\big(\int_{B^{n+1}}\int_{B^{n+1}}\big(\frac{1-|\zeta-x^1|^2}{2}\big)^\alpha\big(\frac{1-|\eta-x^1|^2}{2}\big)^\beta
\frac{ f(\zeta) g(\eta)}{|\zeta-\eta|^{\lambda}}d\zeta d\eta\big)^{1-p}.
\end{align*}
Thus,
\begin{align*}
&\int_{B^{n+1}}\int_{B^{n+1}}\big(\frac{1-|\zeta-x^1|^2}{2}\big)^\alpha\big(\frac{1-|\eta-x^1|^2}{2}\big)^\beta
\frac{f(\zeta)g(\eta)}{|\zeta-\eta|^{\lambda}}d\zeta d\eta\nonumber\\
&\leq\liminf_{j\rightarrow+\infty}\int_{B^{n+1}}\int_{B^{n+1}}\big(\frac{1-|\zeta-x^1|^2}{2}\big)^\alpha\big(\frac{1-|\eta-x^1|^2}{2}\big)^\beta
\frac{ f_j(\zeta)g_j(\eta)}{|\zeta-\eta|^{\lambda}}d\zeta d\eta.
\end{align*}
This together with \eqref{3} implies that
\begin{align*}
C^*_{n,\alpha,\beta,\lambda,p}&=\int_{B^{n+1}}\int_{B^{n+1}}\big(\frac{1-|\zeta-x^1|^2}{2}\big)^\alpha\big(\frac{1-|\eta-x^1|^2}{2}\big)^\beta
\frac{f(\zeta)g(\eta)}{|\zeta-\eta|^{\lambda}}d\zeta d\eta\nonumber\\
&=\lim_{j\rightarrow+\infty}\int_{B^{n+1}}\int_{B^{n+1}}\big(\frac{1-|\zeta-x^1|^2}{2}\big)^\alpha\big(\frac{1-|\eta-x^1|^2}{2}\big)^\beta
\frac{ f_j(\zeta)g_j(\eta)}{|\zeta-\eta|^{\lambda}}d\zeta d\eta.
\end{align*}
Therefore, we deduce that $(f,g)\in L^1(B^{n+1})\times L^1(B^{n+1})$ is a minimizer.

 ${\bf (ii).}$ We show that $f$ and $g$ satisfy Euler-Lagrange equation \eqref{integral}.

Since $0<p<1$ and $0<q<1$, we first need to prove $f>0$ and $g>0$ a.e. in $B^{n+1}$. For any positive function $\varphi\in C^\infty(B^{n+1})$ and sufficiently small $t>0$, one can deduce that
$f+t\varphi>0$ in $B^{n+1}$ and
\begin{align}\label{11}
&\quad t\int_{B^{n+1}}\int_{B^{n+1}}\big(\frac{1-|\zeta-x^1|^2}{2}\big)^\alpha\big(\frac{1-|\eta-x^1|^2}{2}\big)^\beta
\frac{ \varphi(\zeta)g(\eta)}{|\zeta-\eta|^{\lambda}}d\zeta d\eta\nonumber\\
&=\int_{B^{n+1}}\int_{B^{n+1}}\big(\frac{1-|\zeta-x^1|^2}{2}\big)^\alpha\big(\frac{1-|\eta-x^1|^2}{2}\big)^\beta
\frac{ (f+t\varphi)(\zeta)g(\eta)}{|\zeta-\eta|^{\lambda}}d\zeta d\eta\nonumber\\
&\quad -\int_{B^{n+1}}\int_{B^{n+1}}\big(\frac{1-|\zeta-x^1|^2}{2}\big)^\alpha\big(\frac{1-|\eta-x^1|^2}{2}\big)^\beta
\frac{ f(\zeta)g(\eta)}{|\zeta-\eta|^{\lambda}}d\zeta d\eta\nonumber\\
&\geq C^*_{n,\alpha,\beta,\lambda,p}\big(\|f+t\varphi\|_{L^p(B^{n+1})}-\|f\|_{L^p(B^{n+1})}\big)\nonumber\\
&= C^*_{n,\alpha,\beta,\lambda,p}t\big(\int_{B^{n+1}}(f+\theta\varphi)^pd\zeta\big)^{\frac{1}{p}-1}\int_{B^{n+1}}(f+\theta\varphi)^{p-1}\varphi d\zeta\ (0<\theta<t)\nonumber\\
&\geq C^*_{n,\alpha,\beta,\lambda,p}t\int_{B^{n+1}}f^{p-1}\varphi d\zeta,
\end{align}
where in the second equality we have used the mean value theorem, and in the last inequality we have used Fatou's Lemma as $t\rightarrow0^{+}$.

Now we claim that $f>0$ a.e. in $B^{n+1}$. Otherwise, for any $\epsilon>0$, there exists $\Omega_\epsilon\subset B^{n+1}$ such that $|\Omega_\epsilon|>0$ and
$f(\zeta)<\epsilon,\ \ \forall\ \zeta\in\Omega_\epsilon.$
This together with \eqref{11} yields that
\begin{align*}
\epsilon^{p-1}|\Omega_\epsilon|&\leq \int_{\Omega_\epsilon}f^{p-1}(\zeta)d\zeta\nonumber\\
&\leq\frac{1}{C^*_{n,\alpha,\beta,\lambda,p}}\int_{B^{n+1}}\int_{B^{n+1}}\big(\frac{1-|\zeta-x^1|^2}{2}\big)^\alpha\big(\frac{1-|\eta-x^1|^2}{2}\big)^\beta
\frac{g(\eta)}{|\zeta-\eta|^{\lambda}}d\zeta d\eta\nonumber\\
&\leq C\int_{B^{n+1}}g(\eta) d\eta\leq C.
\end{align*}
Then we derive a contradiction when $\epsilon>0$ is sufficiently small. Similarly, we have $g>0$ a.e. in $B^{n+1}$. Therefore,
$(f,g)$ is a pair of solutions of \eqref{integral}.

 ${\bf (iii).}$ We prove $(f,g)\in C^\gamma(\overline{B^{n+1}})\times C^\gamma(\overline{B^{n+1}})$ with $\gamma\in(0,1)$.

Since $f\in L^1(B^{n+1})$ and $0<p<p_\alpha<1$, using \eqref{integral}, there exists a constant $C_7>0$ such that $\frac{1}{C_7}<f<C_7$. Similarly, for some constant $C_8>0$, we have $\frac{1}{C_8}<g<C_8$.


Let $F(\zeta)=C^*_{n,\alpha,\beta,\lambda,p}f^{p-1}(\zeta)$. For any given  $\zeta^1,\zeta^2\in \overline{B^{n+1}}$ and arbitrary $\eta\in B^{n+1}$, it is easy to verify that
\begin{align}\label{pineq}
&\Big||\zeta^1-\eta|^{-\lambda}-|\zeta^2-\eta|^{-\lambda}\Big|\leq
\begin{cases}
C|\zeta^1-\zeta^2|^{-\lambda}, \quad \text{if} \,\, -1<\lambda<0,\\
C\big(1+|\eta|^{-\lambda-1}\big)|\zeta^1-\zeta^2|, \quad \text{if} \,\, \lambda\leq-1.
\end{cases}
\end{align}
Since $g$ is bounded, we have, using\eqref{pineq}, that for any given $\zeta^1,\zeta^2\in \overline{B^{n+1}}$ and $\alpha\in(0,1)$
\begin{eqnarray}\label{equicon}
&&|F(\zeta^1)-F(\zeta^2)|\nonumber\\
&=&\big|\int_{B^{n+1}}\big[(\frac{1-|\zeta^1-x^1|^2}{2}\big)^\alpha-(\frac{1-|\zeta^2-x^1|^2}{2}\big)^\alpha\big](\frac{1-|\eta-x^1|^2}{2})^\beta
\frac{g(\eta)}{|\zeta^1-\eta|^{\lambda}} d\eta\big.+\nonumber\nonumber\\
&&\quad\ \  \big.\int_{B^{n+1}} (\frac{1-|\zeta^2-x^1|^2}{2}\big)^\alpha(\frac{1-|\eta-x^1|^2}{2})^\beta
\big[\frac{1}{|\zeta^1-\eta|^{\lambda}}-\frac{1}{|\zeta^2-\eta|^{\lambda}}\big]g(\eta) d\eta\big|\nonumber\\
&\leq& C|\zeta^1-\zeta^2|^\alpha\int_{B^{n+1}}\frac{g(\eta)}{|\zeta^1-\eta|^{\lambda}}d\eta +C\int_{B^{n+1}}
\big|\frac{1}{|\zeta^1-\eta|^{\lambda}}-\frac{1}{|\zeta^2-\eta|^{\lambda}}\big|g(\eta)d\eta\nonumber\\
&\leq& C|\zeta^1-\zeta^2|^\alpha\int_{B^{n+1}}g(\eta)d\eta +C|\zeta^1-\zeta^2|^{\gamma}
\begin{cases}
\int_{B^{n+1}}g(\eta)d\eta,\quad&-1<\lambda<0\nonumber\\
\int_{B^{n+1}}\big(1+|\eta|^{-\lambda-1}\big)g(\eta)d\eta,\quad& \lambda\leq-1
\end{cases}\nonumber\\
&\leq& C|\zeta^1-\zeta^2|^{\gamma},
\end{eqnarray}
where $\gamma\in (0,\alpha]$.
For $\alpha=0$, we have
\begin{eqnarray*}
|F(\zeta^1)-F(\zeta^2)|&=&\big|\int_{B^{n+1}}\big(\frac{1}{2}-\frac{|\eta-x^1|^2}{2}\big)^{\beta}
\big(\frac{1}{|\zeta^1-\eta|^{\lambda}}-\frac{1}{|\zeta^2-\eta|^{\lambda}}\big)g(\eta)d\eta  \big|\nonumber\\
&\leq& C|\zeta^1-\zeta^2|^{\gamma}.
\end{eqnarray*}
For $\alpha\geq1$, it is easy to check that $|F(\zeta^1)-F(\zeta^2)|\leq C|\zeta^1-\zeta^2|^{\gamma}$.
This implies that $F$ is H\"{o}lder continuous in $\overline{B^{n+1}}$.  Thus $f$ is  at least H\"{o}lder continuous in $\overline{B^{n+1}}$. Similarly, we know that $g$ is at least H\"{o}lder continuous in $\overline{B^{n+1}}$.
\end{proof}

\subsection{\textbf{Minimizer of critical reverse weighted  HLS inequality on the ball}}

\begin{lemma}\label{limit}
Let $\{f_p,g_{q}\}\in C^\gamma(\overline{B^{n+1}})\times C^\gamma(\overline{B^{n+1}})$ be defined as in Proposition \ref{pro1}. Then
$$C^*_{n,\alpha,\beta,\lambda,p}\rightarrow N_{\alpha,\beta,\lambda}\ \ \mbox{as}\ p\rightarrow p_\alpha^-,\ q\rightarrow {q_\beta}^-,$$
and the minimizer pair $\{f_p,g_{q}\}$ satisfy
\begin{align}\label{sub0}
& N_{\alpha,\beta,\lambda}=\lim_{p\rightarrow p_\alpha^-, \atop  q\rightarrow q_\beta^-}\frac{\int_{B^{n+1}}\int_{B^{n+1}}(\frac{1-|\zeta-x^1|^2}{2})^\alpha(\frac{1-|\eta-x^1|^2}{2})^\beta\frac{ f_p(\zeta)g_{q}(\eta)}{|\zeta-\eta|^{\lambda}}d\zeta d\eta}
{\|f_p\|_{L^{p_\alpha}(B^{n+1})}\|g_{q}\|_{L^{q_\beta}(B^{n+1})}}.
\end{align}
\end{lemma}

\begin{proof}
Let $\tilde{f}_p=\frac{f_p}{\|f_p\|_{L^{p_\alpha}(B^{n+1})}}$ and $\tilde{g}_{q}=\frac{g_{q}}{\|g_{q}\|_{L^{q_\beta}(B^{n+1})}}$. Since $\|f_p\|_{L^{p}(B^{n+1})}=\|g_{q}\|_{L^{q}(B^{n+1})}=1$, we have, by reversed H\"{o}lder inequality, that
\begin{align*}
C^*_{n,\alpha,\beta,\lambda,p}
&=\|f_p\|_{L^{p_\alpha}(B^{n+1})}\|g_{q}\|_{L^{q_\beta}(B^{n+1})}\int_{B^{n+1}}\int_{B^{n+1}}
\big(\frac{1-|\zeta-x^1|^2}{2}\big)^\alpha
\big(\frac{1-|\eta-x^1|^2}{2}\big)^\beta
\frac{ \tilde{f}_p(\zeta)\tilde{g}_{q}(\eta)}{|\zeta-\eta|^{\lambda}}d\zeta d\eta\nonumber\\
&\geq |B^{n+1}|^{\frac{1}{p_\alpha}-\frac{1}{p}+\frac{1}{q_\beta}-\frac{1}{q}}\int_{B^{n+1}}\int_{B^{n+1}}
\big(\frac{1-|\zeta-x^1|^2}{2}\big)^\alpha\big(\frac{1-|\eta-x^1|^2}{2}\big)^\beta
\frac{ \tilde{f}_p(\zeta)\tilde{g}_{q}(\eta)}{|\zeta-\eta|^{\lambda}}d\zeta d\eta\nonumber\\
&\rightarrow N_{\alpha,\beta,\lambda}\quad  \mbox{as}\ p\rightarrow p_\alpha^-,\ q\rightarrow {q_\beta}^-,
\end{align*}
which yields
\begin{align}\label{es1}
\liminf_{p\rightarrow p_\alpha^-, q\rightarrow {q_\beta}^-} C^*_{n,\alpha,\beta,\lambda,p}\geq N_{\alpha,\beta,\lambda}.
\end{align}
Let $\{f_k,g_k\}\subset L^{p_\alpha}(B^{n+1})\times L^{q_\beta}(B^{n+1})$ be a pair of minimizing sequence of $N_{\alpha,\beta,\lambda}$. Namely,
\begin{align*}
N_{\alpha,\beta,\lambda}
&=\lim_{k\rightarrow+\infty}\frac{\int_{B^{n+1}}\int_{B^{n+1}}(\frac{1-|\zeta-x^1|^2}{2})^\alpha
(\frac{1-|\eta-x^1|^2}{2})^\beta\frac{ f_k(\zeta)g_k(\eta)}{|\zeta-\eta|^{\lambda}}d\zeta d\eta}{\|f_k\|_{L^{p_\alpha}(B^{n+1})}\|g_k\|_{L^{q_\beta}(B^{n+1})}}.
\end{align*}

Set $\tilde{f}_k=\frac{f_k}{\|f_k\|_{L^{p}(B^{n+1})}}$ and $\tilde{g}_k=\frac{g_k}{\|g_k\|_{L^{q}(B^{n+1})}}$.
For $p\in (0,p_\alpha)$ and $q\in (0,q_\beta)$, we deduce that
\begin{align}\label{sub1}
C^*_{n,\alpha,\beta,\lambda,p}&\leq\frac{\int_{B^{n+1}}\int_{B^{n+1}}(\frac{1-|\zeta-x^1|^2}{2})^\alpha(\frac{1-|\eta-x^1|^2}{2})^\beta\frac{ \tilde{f}_k(\zeta)\tilde{g}_k(\eta)}{|\zeta-\eta|^{\lambda}}d\zeta d\eta}
{\|\tilde{f}_k\|_{L^{p}(B^{n+1})}\|\tilde{g}_k\|_{L^{q}(B^{n+1})}}\nonumber\\
&=\frac{\int_{B^{n+1}}\int_{B^{n+1}}(\frac{1-|\zeta-x^1|^2}{2})^\alpha(\frac{1-|\eta-x^1|^2}{2})^\beta\frac{ f_k(\zeta)g_k(\eta)}{|\zeta-\eta|^{\lambda}}d\zeta d\eta}
{\|f_k\|_{L^{p}(B^{n+1})}\|g_k\|_{L^{q}(B^{n+1})}}.
\end{align}
By \eqref{sub1} and dominated convergence theorem, we have
\begin{align*}
\limsup_{p\rightarrow p_\alpha^-, q\rightarrow {q_\beta}^-}C^*_{n,\alpha,\beta,\lambda,p}
&\leq\frac{\int_{B^{n+1}}\int_{B^{n+1}}(\frac{1-|\zeta-x^1|^2}{2})^\alpha
(\frac{1-|\eta-x^1|^2}{2})^\beta\frac{ f_k(\zeta)g_k(\eta)}{|\zeta-\eta|^{\lambda}}d\zeta d\eta}{\|f_k\|_{L^{p_\alpha}(B^{n+1})}\|g_k\|_{L^{q_\beta}(B^{n+1})}}.
\end{align*}
Then, let $k\rightarrow\infty$, one has
\begin{align}\label{sub2}
&\limsup_{p\rightarrow p_\alpha^-, q\rightarrow {q_\beta}^-}C^*_{n,\alpha,\beta,\lambda,p}\leq N_{\alpha,\beta,\lambda}.
\end{align}
Hence, it follows from \eqref{es1} and \eqref{sub2} that
$$\limsup_{p\rightarrow p_\alpha^-, q\rightarrow {q_\beta}^-}C^*_{n,\alpha,\beta,\lambda,p}=N_{\alpha,\beta,\lambda}.$$

Applying H\"{o}lder inequality, we conclude that
\begin{align*}
N_{\alpha,\beta,\lambda}&\leq\frac{\int_{B^{n+1}}\int_{B^{n+1}}(\frac{1-|\zeta-x^1|^2}{2})^\alpha(\frac{1-|\eta-x^1|^2}{2})^\beta\frac{ f_p(\zeta)g_{q}(\eta)}{|\zeta-\eta|^{\lambda}}d\zeta d\eta}
{\|f_p\|_{L^{p_\alpha}(B^{n+1})}\|g_{q}\|_{L^{q_\beta}(B^{n+1})}}\nonumber\\
&\leq\frac{\int_{B^{n+1}}\int_{B^{n+1}}(\frac{1-|\zeta-x^1|^2}{2})^\alpha(\frac{1-|\eta-x^1|^2}{2})^\beta\frac{ f_p(\zeta)g_{q}(\eta)}{|\zeta-\eta|^{\lambda}}d\zeta d\eta}
{|B^{n+1}|^{\frac{1}{p}-\frac{1}{p_\alpha}+\frac{1}{q}-\frac{1}{q_\beta}}}\nonumber\\
&\rightarrow N_{\alpha,\beta,\lambda},\ \ \ \mbox{as}\ p\rightarrow p_\alpha^-, q\rightarrow {q_\beta}^-.
\end{align*}
Therefore, we obtain \eqref{sub0}.
\end{proof}

Now we are ready to complete the proof of Theorem \ref{equi01'}.

\textbf{Proof of Theorem \ref{equi01'}.}
Let $\{f_p,g_{q}\}\in C^\gamma(\overline{B^{n+1}})\times C^\gamma(\overline{B^{n+1}})$ be a minimizing sequence of $N_{\alpha,\beta,\lambda}$. Then, $\{f_p,g_{q}\}$ satisfies \eqref{integral}.
We assume that $f_p(\mathcal{W})=\max\limits_{\zeta\in \overline{B^{n+1}}}f_p(\zeta)$.

We carry out the proof by considering two cases.

{\bf Case 1:} $\mathcal{W}\in B^{n+1}$. By the scale invariance and translation, without loss of generality, we assume that $\mathcal{W}=(0,...,0,-1)\in B^{n+1}$. For some subsequences $p_j\rightarrow p_\alpha$ and $q_j\rightarrow q_\beta$,  we consider all two possibilities of $\max\{\max\limits_{\zeta\in \overline{B^{n+1}}}f_{p_j},\max\limits_{\zeta\in \overline{B^{n+1}}}g_{q_j}\}$.

Indeed, the invariant of system \eqref{integral} is clearly absent under the translation transformation. However, we can still make scaling and translation transformation such that $\mathcal{W}=(0,...,0,-1)\in B^{n+1}$, due to the uniformly boundedness of $(\frac{1-|\zeta-x^1|^2}{2})^\alpha(\frac{1-|\eta-x^1|^2}{2})^\beta$.

{\bf Case 1a:} For some subsequences $p_j\rightarrow p_\alpha$ and $q_j\rightarrow q_\beta$, $\max\{\max\limits_{\zeta\in \overline{B^{n+1}}}f_{p_j},\max\limits_{\zeta\in \overline{B^{n+1}}}g_{q_j}\}$ is uniformly bounded.

Arguing as \eqref{bdd1} and \eqref{equicon}, we have $\{f_{p_j}\}$ and $\{g_{q_j}\}$ are uniformly bounded and equicontinuous in $\overline{B^{n+1}}$. By \eqref{integral}, we know that there exists some constant $C>0$ (independent of $p_j$, $q_j$) such that $f_{p_j}, g_{q_j}\geq C$.  Then, it follows from Arel\`{a}-Ascoli theorem that there exist two subsequences of $\{f_{p_j}\}$ and $\{g_{q_j}\}$ (still denoted by $\{f_{p_j}\}$ and $\{g_{q_j}\}$) and two nonnegative functions $f, g\in C^\gamma(\overline{B^{n+1}})$ such that
$$f_{p_j}\rightarrow f,\ \ \ g_{q_j}\rightarrow g \ \ \mbox{uniformaly\ on}\ \overline{B^{n+1}}.$$
Hence we have
$$\int_{B^{n+1}}f^{p_\alpha}(\zeta)d\zeta=\lim_{p_j\rightarrow p_\alpha }\int_{B^{n+1}}f_{p_j}^{p_j}(\zeta)d\zeta=1,\ \ \int_{B^{n+1}}g^{q_\beta}(\zeta)d\zeta=\lim_{q_j\rightarrow q_\beta }\int_{B^{n+1}}g_{q_j}^{q_j}(\zeta)d\zeta=1.$$
By \eqref{integral} and Lemma \ref{limit}, one can deduce that
\begin{equation*}\begin{cases}
N_{\alpha,\beta,\lambda}f^{p_\alpha-1}(\zeta)=\int_{B^{n+1}}
(\frac{1-|\zeta-x^1|^2}{2})^\alpha(\frac{1-|\eta-x^1|^2}{2})^\beta
\frac{ g(\eta)}{|\zeta-\eta|^{\lambda}} d\eta,\\
N_{\alpha,\beta,\lambda}g^{q_\beta-1}(\zeta)=\int_{B^{n+1}} (\frac{1-|\eta-x^1|^2}{2})^\alpha(\frac{1-|\zeta-x^1|^2}{2})^\beta
\frac{ f(\eta)}{|\zeta-\eta|^{\lambda}} d\eta,
\end{cases}\end{equation*}
as $j\rightarrow+\infty$.
It follows that $f,g$ are minimizers.

{\bf Case 1b:} For any subsequences $p_j\rightarrow p_\alpha$ and $ q_j\rightarrow q_\beta$, $f_{p_j}(\mathcal{W})\rightarrow+\infty$ or $\max\limits_{\zeta\in \overline{B^{n+1}}} g_{q_j}\rightarrow+\infty$. Without loss of generality, we assume that $f_{p_j}(\mathcal{W})\rightarrow+\infty$.

{\bf Case (i):} $\limsup\limits_{j\rightarrow+\infty}\frac{f_{p_j}(\mathcal{W})}{{\max\limits_{\zeta\in \overline{B^{n+1}}}} g_{q_j}}=+\infty$.
Then there exist two subsequences of $p_j$ and  $q_j$ (still denoted by $p_j$ and  $q_j$) such that $f_{p_j}(\mathcal{W})\rightarrow +\infty$ and $\frac{f_{p_j}(\mathcal{W})}{\max\limits_{\zeta\in \overline{B^{n+1}}} g_{q_j}}\rightarrow+\infty$. Let $u_j=f_{p_j}^{p_j-1}$ and $v_j=g_{q_j}^{q_j-1}$. Since $\|f_{p_j}\|_{L^{p_j}(B^{n+1})}=\|g_{q_j}\|_{L^{q_j}(B^{n+1})}=1$, we have
\begin{equation}\label{int001}
\int_{B^{n+1}}u_j^{p'_j}(\zeta)d\zeta=\int_{B^{n+1}}v_j^{q'_j}(\zeta)d\zeta=1.
\end{equation}
By \eqref{integral}, we know
\begin{equation*}\begin{cases}
C_{n,\alpha,\beta,\lambda,p_j}u_j(\zeta)=\int_{B^{n+1}} (\frac{1-|\zeta-x^1|^2}{2})^\alpha(\frac{1-|\eta-x^1|^2}{2})^\beta
\frac{ v_j^{q'_j-1}(\eta)}{|\zeta-\eta|^{\lambda}} d\eta,\\
C_{n,\alpha,\beta,\lambda,p_j}v_j(\zeta)=\int_{B^{n+1}} (\frac{1-|\eta-x^1|^2}{2})^\alpha(\frac{1-|\zeta-x^1|^2}{2})^\beta
\frac{ u_j^{p'_j-1}(\eta)}{|\zeta-\eta|^{\lambda}} d\eta.
\end{cases}\end{equation*}
Then, applying conformal transformation and dilation on $\mathbb R_+^{n+1}$, we have
\begin{equation}\label{intsp00}\begin{cases}
C_{n,\alpha,\beta,\lambda,p_j}u_j(\mathcal{T}^{-1}(\rho(x,t)))
=\rho^{n+1+\alpha+\beta-\lambda}\int_{\mathbb{R}^{n+1}_+}
\frac{t^{\alpha}z^{\beta} v_j^{q'_j-1}(\mathcal{T}^{-1}(\rho(y,z)))}{|(x,t)-(y,z)|^\lambda}dydz,\\
C_{n,\alpha,\beta,\lambda,p_j}v_j(\mathcal{T}^{-1}(\rho(x,t)))
=\rho^{n+1+\alpha+\beta-\lambda}\int_{\mathbb{R}^{n+1}_+}
\frac{t^{\beta}z^{\alpha} u_j^{p'_j-1}(\mathcal{T}^{-1}(\rho(y,z)))}{|(x,t)-(y,z)|^\lambda}dydz.
\end{cases}\end{equation}

Note that $\mathcal{T}(0,...,0,-1)=(0,...,0,2)$. Now we take $\rho=\rho_j$ such that $\rho_j^{\frac{(n+1+\alpha+\beta-\lambda)q'_j}{(q'_j-1)(p'_j-1)-1}}u_j(\mathcal{T}^{-1}(0,...,0,2))=1$
and let
\begin{equation}\label{intsp01}\begin{cases}
U_j(x,t)=\rho_j^{\frac{(n+1+\alpha+\beta-\lambda)q'_j}{(q'_j-1)(p'_j-1)-1}}u_j(\mathcal{T}^{-1}(\rho_j(x,t))),\\
V_j(x,t)=\rho_j^{\frac{(n+1+\alpha+\beta-\lambda)p'_j}{(q'_j-1)(p'_j-1)-1}}v_j(\mathcal{T}^{-1}(\rho_j(x,t))).
\end{cases}\end{equation}
It is easy to see that $U_j$ and $V_j$ satisfy the following renormalized equations
\begin{equation}\label{intsp02}\begin{cases}
C_{n,\alpha,\beta,\lambda,p_j}U_j(x,t)
=\int_{\mathbb{R}^{n+1}_+}\frac{t^{\alpha}z^{\beta } V_j^{q'_j-1}(y,z)}{|(x,t)-(y,z)|^\lambda}dydz,\\
C_{n,\alpha,\beta,\lambda,p_j}V_j(x,t)
=\int_{\mathbb{R}^{n+1}_+}\frac{t^{\beta}z^{\alpha } U_j^{p'_j-1}(y,z)}{|(x,t)-(y,z)|^\lambda}dydz.
\end{cases}\end{equation}
Moreover, $ U_j(x,t)\geq  U_j(0,\frac{2}{\rho_j})=1$ and for $p_j\leq q_j$,
\begin{align}\label{nor1}
V_j(x,t)&\geq \rho_j^{\frac{(n+1+\alpha+\beta-\lambda)p'_j}{(q'_j-1)(p'_j-1)-1}}\min\limits_{\zeta\in \overline{B^{n+1}}} v_j(\zeta)\nonumber\\
& =\rho_j^{\frac{(n+1+\alpha+\beta-\lambda)(p'_j-q'_j)}{(q'_j-1)(p'_j-1)-1}}\frac{\min\limits_{\zeta\in \overline{B^{n+1}}} v_j}{u_j(\mathcal{T}^{-1}(0,...,0,2))}\nonumber\\
& \rightarrow +\infty \quad \  \mbox{uniformaly\ for\ any}\ (x,t)\in\mathbb{R}^{n+1}_+\ \mbox{as}\ j\rightarrow +\infty.
\end{align}

Next we show that there exists a constant $C_1\geq1$ such that for any $(x,t)\in\mathbb{R}^{n+1}_+$,
\begin{align}\label{nor2}
&\frac{1}{C_1}(1+|(x,t)|^{-\lambda})\leq \frac{U_j(x,t)}{t^\alpha}\leq C_1(1+|(x,t)|^{-\lambda}).
\end{align}
Once the inequality \eqref{nor2} holds, by the change of polar coordinates,  we have that for $\alpha p'_j+1>0$,
\begin{align}\label{contra}
&C_{n,\alpha,\beta,\lambda,p_j}V_j(0,1)\nonumber\\
&=\int_{\mathbb{R}^{n+1}_+}\frac{z^{\alpha } U_j^{p'_j-1}(y,z)}{|(0,1)-(y,z)|^\lambda}dydz\nonumber\\
&\leq C\int_{\mathbb{R}^{n+1}_+\setminus B_{1}^+(0)}z^{\alpha }(1+|(y,z)|^{-\lambda})U_j^{p'_j-1}(y,z)dydz +C\int_{ B_{1}^+(0)}z^{\alpha }(1+|(y,z)|^{-\lambda})U_j^{p'_j-1}(y,z)dydz\nonumber\\
&\leq C\int_{\mathbb{R}^{n+1}_+\setminus B_{1}^+(0)}z^{\alpha p'_j}|(y,z)|^{-\lambda p'_j}dydz +C\int_{ B_{1}^+(0)}z^{\alpha }(1+|(y,z)|^{-\lambda})U_j^{p'_j-1}(y,z)dydz\nonumber\\
& \leq C\bar{J}_\alpha \int_{1}^\infty \rho^{(\alpha-\lambda)p'_j}\rho^nd\rho+C\int_{ B_{1}^+(0)}z^{\alpha }(1+|(y,z)|^{-\lambda})dydz\nonumber\\
&\leq C,
\end{align}
where \begin{eqnarray*}
\bar{J}_\alpha&=&\frac{1}{n+1+(\alpha-\lambda)p'_j}\int_{0}^{\pi}(sin\theta_1)^{n-1+\alpha p'_j}d\theta_1\int_{0}^{\pi}(sin\theta_2)^{n-2+\alpha p'_j }d\theta_2\cdots\int_{0}^{\pi}(sin\theta_{n})^{\alpha p'_j} d\theta_{n}\\
&=&\frac{\pi^{\frac{n}{2}}}{n+1+(\alpha-\lambda)p'_j}\frac{\Gamma(\frac{\alpha p'_j+1}{2})}{\Gamma(\frac{n+\alpha p'_j+1}{2})},
\end{eqnarray*}
and we have used $n+1+(\alpha-\lambda)p'_j<0$ and $U_j(x,t)\geq 1$ in the last inequality.
This contradicts \eqref{nor1}.  Therefore, {\bf Case (i)} is impossible.

Then what left is to prove the inequality \eqref{nor2} holds.
Combining \eqref{int001}, \eqref{intsp01} with conformal transformation, we have
\begin{align*}
&\int_{B^{n+1}}u_j^{p'_j}(\zeta)d\zeta=\int_{\mathbb{R}^{n+1}_+} U_j^{p'_j}(x,t)\rho_j^{n+1-\frac{(n+1+\alpha+\beta-\lambda)p'_jq'_j}{(p'_j-1)(q'_j-1)-1}}dxdt=1,\nonumber\\
&\int_{B^{n+1}}v_j^{q'_j}(\zeta)d\zeta=\int_{\mathbb{R}^{n+1}_+} V_j^{q'_j}(x,t)\rho_j^{n+1-\frac{(n+1+\alpha+\beta-\lambda)p'_jq'_j}{(p'_j-1)(q'_j-1)-1}}dxdt=1.
\end{align*}
Hence we conclude that
\begin{equation*}
\text{meas}\{(x,t)\in\mathbb{R}^{n+1}_+ \mid U_j(x,t)<+\infty\}>0,\ \ \ \text{meas}\{(y,z)\in\mathbb{R}^{n+1}_+ \mid V_j(y,z)<+\infty\}>0.
\end{equation*}
It follows that there exists $R>1$ large enough and a measurable set $E$ such that
\begin{equation*}
E\subset\{(y,z)\in\mathbb{R}^{n+1}_+ \mid V_j(y,z)<R\}\cap B_R^{+}(0)
\end{equation*}
with $meas(E)>\frac{1}{R}$.

For any $(x,t)\in\mathbb{R}^{n+1}_+$, it follows from $\beta\geq0$ and $q'_j<0$ that
\begin{equation*}\begin{split}
C_{n,\alpha,\beta,\lambda,p_j}U_j(x,t)
&=\int_{\mathbb{R}^{n+1}_+}\frac{t^\alpha z^{\beta} V_j^{q'_j-1}(y,z)}{|(x,t)-(y,z)|^\lambda}dydz\\
&\geq \int_{E}\frac{t^\alpha z^{\beta} V_j^{q'_j-1}(y,z)}{|(x,t)-(y,z)|^\lambda}dydz\\
&\geq R^{q'_j-1}t^\alpha\int_{E}z^{\beta}|(x,t)-(y,z)|^{-\lambda} dydz.
\end{split}\end{equation*}
Since $\lambda<0$, there exists a constant $C_1\geq1$ such that
\begin{equation*}
\frac{U_j(x,t)}{t^\alpha}\geq\frac{1}{C_1}(1+|(x,t)|^{-\lambda}), \qquad \forall \, (x,t)\in \mathbb{R}^{n+1}_+.
\end{equation*}
Similarly, for any $(x,t)\in \mathbb{R}^{n+1}_+$, there exists a constant $C_2\geq1$ such that
\begin{equation}\label{nor11}
\frac{V_j(x,t)}{t^\beta}\geq\frac{1}{C_2}(1+|(x,t)|^{-\lambda}).
\end{equation}

On the other hand, there exists a point $(\hat{x},\hat{t})\in\mathbb{R}^{n+1}_+$ such that
\begin{equation}\label{nor12}
C_{n,\alpha,\beta,\lambda,p_j}U_j(\hat{x},\hat{t})=\int_{\mathbb{R}^{n+1}_{+}}\frac{\hat{t}^\alpha z^{\beta }V_j^{q'_j-1}(y,z)}{|(\hat{x},\hat{t})-(y,z)|^{\lambda}}dydz<+\infty.
\end{equation}
Then combining \eqref{nor11}, \eqref{nor12} with $\beta q'_j+1>0$, we have
\begin{align}\label{nor13}
\int_{\mathbb{R}^{n+1}_{+}}z^{\beta }(1+|(y,z)|^{-\lambda})V_j^{q'_j-1}(y,z)dydz&\leq C_{\hat{x},\hat{t}}\int_{\{|(y,z)|<\frac{1}{2}|(\hat{x},\hat{t})|\}}z^{\beta }|(\hat{x},\hat{t})-(y,z)|^{-\lambda} V_j^{q'_j-1}(y,z)dydz\nonumber\\
&\ \ \ +C_{\hat{x},\hat{t}}\int_{\{|(y,z)|>2|(\hat{x},\hat{t})|\}}z^{\beta }|(\hat{x},\hat{t})-(y,z)|^{-\lambda} V_j^{q'_j-1}(y,z)dydz \nonumber\\
&\ \ \ +\int_{\{\frac{1}{2}|(\hat{x},\hat{t})|\leq|(y,z)|\leq2|(\hat{x},\hat{t})|\}}z^{\beta }(1+|(y,z)|^{-\lambda})V_j^{q'_j-1}(y,z)dydz\nonumber\\
&< +\infty.
\end{align}
Therefore, by \eqref{nor13} and $\beta q'_j+1>0$, we derive that for any $(x,t)\in\mathbb{R}^{n+1}_+$,
\begin{equation}\label{nor14}\begin{split}
\frac{U_j(x,t)}{t^\alpha(1+|(x,t)|^{-\lambda})}&=\int_{\mathbb{R}^{n+1}_+}\frac{|(x,t)-(y,z)|^{-\lambda}}{1+|(x,t)|^{-\lambda}}z^{\beta }V_j^{q'_j-1}(y,z)dydz\\
&\leq C\int_{\mathbb{R}^{n+1}_+}(1+|(y,z)|^{-\lambda})z^{\beta }V_j^{q'_j-1}(y,z)dydz
<+\infty.
\end{split}\end{equation}
Combining \eqref{nor11} with \eqref{nor14}, we obtain \eqref{nor2}.

{\bf Case (ii):} $\limsup\limits_{j\rightarrow+\infty}\frac{f_{p_j}(\mathcal{W})}{\max\limits_{\zeta\in \overline{B^{n+1}}} g_{q_j}}=0$.
Then there exist two subsequences of $p_j$ and  $q_j$ (still denoted by $p_j$ and  $q_j$) such that $f_{p_j}(\mathcal{W})\rightarrow +\infty$ and $\frac{f_{p_j}(\mathcal{W})}{\max\limits_{\zeta\in \overline{B^{n+1}}} g_{q_j}}\rightarrow0$, which implies that $\max\limits_{\zeta\in \overline{B^{n+1}}} g_{q_j}\rightarrow+\infty$.
Using similar arguments as {\bf Case (i)}, we prove that {\bf Case (ii)} does not hold.

{\bf Case (iii):} $\limsup\limits_{j\rightarrow+\infty}\frac{f_{p_j}(\mathcal{W})}{\max\limits_{\zeta\in \overline{B^{n+1}}} g_{q_j}}=c_0\in (0,+\infty)$.
Then there exist two subsequences of $p_j$ and  $q_j$ (still denoted by $p_j$ and  $q_j$) such that $f_{p_j}(\mathcal{W})\rightarrow +\infty$, $\max\limits_{\zeta\in \overline{B^{n+1}}} g_{q_j}\rightarrow+\infty$ and $\frac{f_{p_j}(\mathcal{W})}{\max\limits_{\zeta\in \overline{B^{n+1}}} g_{q_j}}\rightarrow c_0$. Similar to {\bf Case (i)}, we choose $\{U_j,V_j\}$ defined as \eqref{intsp01}. It is easy to see that $U_j$ and $V_j$ satisfy \eqref{intsp02}. Moreover, $ U_j(x,t)\geq U_j(0,\frac{2}{\rho_j})=1$  and for $p_j\leq q_j$,
\begin{align*}
V_j(x,t)&\geq \rho_j^{\frac{(n+1+\alpha+\beta-\lambda)p'_j}{(q'_j-1)(p'_j-1)-1}}\min\limits_{\zeta\in \overline{B^{n+1}}} v_j(\zeta)\nonumber\\
& =\rho_j^{\frac{(n+1+\alpha+\beta-\lambda)(p'_j-q'_j)}{(q'_j-1)(p'_j-1)-1}}
\frac{\min\limits_{\zeta\in \overline{B^{n+1}}} v_j}{u_j(\mathcal{T}^{-1}(0,...,0,2))}\nonumber\\
& \geq c_1>0 \quad \  \mbox{uniformaly\ for\ any}\ (x,t)\in\mathbb{R}^{n+1}_+\ \mbox{as}\ j\rightarrow +\infty.
\end{align*}
Hence $V_j(x,t)$ has uniformly lower bound $c_1>0$.

Using similar arguments as \eqref{nor2}, there exist constants $C_3\geq1$ and $C_4\geq1$ such that for any $(x,t)\in\mathbb{R}^{n+1}_+$,
\begin{align}\label{nor16}
&\frac{1}{C_3}(1+|(x,t)|^{-\lambda})\leq \frac{U_j(x,t)}{t^{\alpha}}\leq C_3(1+|(x,t)|^{-\lambda}),
\end{align}
\begin{align}\label{nor17}
&\frac{1}{C_4}(1+|(x,t)|^{-\lambda})\leq \frac{V_j(x,t)}{t^{\beta}}\leq C_4(1+|(x,t)|^{-\lambda}).
\end{align}
For any given constant $R_0>0$ and $(x,t)\in B_{R_0}^+(0)$, using \eqref{nor13}, \eqref{nor17} and the lower bound of $V_j(x,t)$, we have
\begin{align*}
C_{n,\alpha,\beta,\lambda,p_j}U_j(x,t)&=\int_{\mathbb{R}^{n+1}_+}\frac{t^{\alpha}z^{\beta } V_j^{q'_j-1}(y,z)}{|(x,t)-(y,z)|^\lambda}dydz\nonumber\\
&\leq Ct^{\alpha}\int_{\mathbb{R}^{n+1}_+\setminus B_{|(x,t)|}^+(0) }z^{\beta }V_j^{q'_j-1}(y,z)(1+|(y,z)|^{-\lambda})dydz\nonumber\\
&\quad +Ct^{\alpha}\int_{B_{|(x,t)|}^+(0)}z^{\beta }(1+|(y,z)|^{-\lambda})dydz\nonumber\\
&\leq Ct^{\alpha}.
\end{align*}
This implies that $U_j(x,t)$ is uniformly bounded in $B_{R_0}^+(0)$. Similarly, $V_j(x,t)$ is uniformly bounded in $B_{R_0}^+(0)$.

Similar  to \eqref{equicon}, we can show that
$U_j(x,t)$ is equicontinuous in $\overline{\mathbb{R}^{n+1}_{+}}$  and $V_j(x,t)$ is equicontinuous in $\overline{\mathbb{R}^{n+1}_{+}}$.
It follows from Arzel\`{a}-Ascoli theorem that there exist two subsequences of $\{U_j\}$ and $\{V_j\}$ (still denoted by $\{U_j\}$ and $\{V_j\}$) and two functions $U$ and $V$ with lower bound $C>0$ such that
\begin{equation*}
U_j\rightarrow U\ \ \mbox{and}\ V_j\rightarrow V,\ \ \ \ \mbox{as}\  j\rightarrow+\infty\ \mbox{uniformaly\ on}\ B_{R_0}^+(0).
\end{equation*}
By the arbitrariness of $R_0$, one can see that $U(x)$ and $V(x)$ satisfy
\begin{equation*}\begin{cases}
N_{\alpha,\beta,\lambda}U(x,t)
=\int_{\mathbb{R}^{n+1}_+}\frac{t^{\alpha}z^{\beta } V^{q'_\beta-1}(y,z)}{|(x,t)-(y,z)|^\lambda}dydz\\
N_{\alpha,\beta,\lambda}V(x,t)
=\int_{\mathbb{R}^{n+1}_+}\frac{t^\beta z^{\alpha } U_j^{p'_\alpha-1}(y,z)}{|(x,t)-(y,z)|^\lambda}dydz.
\end{cases}\end{equation*}
Since
\begin{align*}
1&=\int_{B^{n+1}}u_j^{p'_j}(\zeta)d\zeta\\
&=\int_{\mathbb{R}^{n+1}_+} U_j^{p'_j}(x,t)\rho_j^{n+1-\frac{(n+1+\alpha+\beta-\lambda)p'_jq'_j}{(p'_j-1)(q'_j-1)-1}}dxdt\\
&\leq \int_{\mathbb{R}^{n+1}_+} U_j^{p'_j}(x,t)dxdt
\end{align*}
and $$ U_j^{p'_j}\rightarrow  U^{p'_\alpha}$$
uniformly on any compact domain, we know from \eqref{nor16} that
$$\int_{\mathbb{R}^{n+1}_+} U^{p'_\alpha}(x,t)dxdt=\lim_{j\rightarrow+\infty}\int_{\mathbb{R}^{n+1}_+} U_j^{p'_j}(x,t)dxdt\geq1.$$
Similarly, using \eqref{nor17}, we have
$$\int_{\mathbb{R}^{n+1}_+} V^{q'_\beta}(x,t)dxdt=\lim_{j\rightarrow+\infty}\int_{\mathbb{R}^{n+1}_+} V_j^{q'_j}(x,t)dxdt\geq1.$$
Let $F(x,t)= U^{p'_\alpha-1}(x,t)$ and $G(x,t)= V^{q'_\beta-1}(x,t)$. Then we derive
$$\int_{\mathbb{R}^{n+1}_+}F^{p_\alpha}(x,t)dxdt\geq1,\ \ \ \ \int_{\mathbb{R}^{n+1}_+}G^{q_\beta}(x,t)dxdt\geq1,$$
and $F, G$ satisfy
\begin{equation*}\begin{cases}
N_{\alpha,\beta,\lambda}F^{p_\alpha-1}(x,t)
=\int_{\mathbb{R}^{n+1}_+}\frac{t^\alpha z^{\beta} G(y,z)}{|(x,t)-(y,z)|^\lambda}dydz,\\
N_{\alpha,\beta,\lambda}G^{q_\beta-1}(x,t)
=\int_{\mathbb{R}^{n+1}_+}\frac{t^\beta z^{\alpha} F(y,z)}{|(x,t)-(y,z)|^\lambda}dydz.
\end{cases}\end{equation*}
Since $p_\alpha<1$ and $q_\beta<1$, using conformal transformation, we conclude that
\begin{align*}
N^2_{\alpha,\beta,\lambda}&=\frac{(\int_{\mathbb{R}^{n+1}_+}\int_{\mathbb{R}^{n+1}_+}\frac{t^\alpha z^{\beta} F(x,t) G(y,z)}{|(x,t)-(y,z)|^\lambda}dxdtdydz)^2}{\int_{\mathbb{R}^{n+1}_+}F^{p_\alpha}(x,t)dxdt\int_{\mathbb{R}^{n+1}_+}G^{q_\beta}(y,z)dydz}\\
&\geq\frac{(\int_{\mathbb{R}^{n+1}_+}\int_{\mathbb{R}^{n+1}_+}\frac{t^\alpha z^{\beta} F(x,t) G(y,z)}{|(x,t)-(y,z)|^\lambda}dxdtdydz)^2}{(\int_{\mathbb{R}^{n+1}_+}F^{p_\alpha}(x,t)dxdt)^{\frac{2}{p_\alpha}}
(\int_{\mathbb{R}^{n+1}_+}G^{q_\beta}(y,z)dydz)^{\frac{2}{q_\beta}}}\\
&=\frac{(\int_{B^{n+1}}\int_{B^{n+1}}
(\frac{1-|\zeta-x^1|^2}{2})^\alpha(\frac{1-|\eta-x^1|^2}{2})^\beta
\frac{ f(\zeta)g(\eta)}{|\zeta-\eta|^{\lambda}} d\zeta d\eta)^2}{(\int_{B^{n+1}}f^{p_\alpha}(\zeta)d\zeta)^{\frac{2}{p_\alpha}}
(\int_{B^{n+1}}g^{q_\beta}(\eta)d\eta)^{\frac{2}{q_\beta}}}.
\end{align*}
Therefore, $\{f(\xi), g(\eta)\}$ is a pair of minimizers of sharp constant $N_{\alpha,\beta,\lambda}$.

{\bf Case 2:} $\mathcal{W}_p\in \partial B^{n+1}$. By the scale and rotation invariance of \eqref{integral} for $\zeta\in \partial B^{n+1}$, without loss of generality, we assume $\mathcal{W}_p=(0,..,0)\in \partial B^{n+1}$. For some subsequences $p_j\rightarrow p_\alpha$ and $q_j\rightarrow q_\beta$,  we consider all possible values of $\max\{\max\limits_{\zeta\in \overline{B^{n+1}}}f_{p_j},\max\limits_{\zeta\in \overline{B^{n+1}}}g_{q_j}\}$.

It is easy to see that {\bf Case 1a} still holds here.
We only need to consider the following case:\\

For any subsequences $p_j\rightarrow p_\alpha$ and $ q_j\rightarrow q_\beta$, $f_{p_j}(\mathcal{W})\rightarrow+\infty$ or $\max\limits_{\zeta\in \overline{B^{n+1}}} g_{q_j}\rightarrow+\infty$. Without loss of generality, we assume that $f_{p_j}(\mathcal{W})\rightarrow+\infty$.

{\bf Case (2i):} $\limsup\limits_{j\rightarrow+\infty}\frac{f_{p_j}(\mathcal{W})}{{\max\limits_{\zeta\in \overline{B^{n+1}}}} g_{q_j}}=+\infty$.
Then there exist two subsequences of $p_j$ and  $q_j$ (still denoted by $p_j$ and  $q_j$) such that $f_{p_j}(\mathcal{W})\rightarrow +\infty$ and $\frac{f_{p_j}(\mathcal{W})}{\max\limits_{\zeta\in \overline{B^{n+1}}} g_{q_j}}\rightarrow+\infty$.

Using similar arguments as {\bf Case (i)}, applying conformal transformation and dilation, we know that $u_j(\mathcal{T}^{-1}(\rho(x,t)))$ and $v_j(\mathcal{T}^{-1}(\rho(x,t)))$ satisfy \eqref{intsp00} for $(x,t)\in \overline{\mathbb R^{n+1}_+}$.
Then we take $\rho=\rho_j$ satisfy $\rho_j^{\frac{(n+1+\alpha+\beta-\lambda)q'_j}{(q'_j-1)(p'_j-1)-1}}u_j(\mathcal{T}^{-1}(0,...,0,0))=1$,
and  let
\begin{equation}\label{ints001}\begin{cases}
U_j(x,t)=\rho_j^{\frac{(n+1+\alpha+\beta-\lambda)q'_j}{(q'_j-1)(p'_j-1)-1}}u_j(\mathcal{T}^{-1}(\rho_j(x,t))),\\
V_j(x,t)=\rho_j^{\frac{(n+1+\alpha+\beta-\lambda)p'_j}{(q'_j-1)(p'_j-1)-1}}v_j(\mathcal{T}^{-1}(\rho_j(x,t)))
\end{cases}\end{equation}
for $(x,t)\in \overline{\mathbb R^{n+1}_+}$.
It is easy to see that $U_j$ and $V_j$ satisfy the renormalized equations \eqref{intsp02}.
Moreover, $ U_j(x,t)\geq  U_j(0,...,0,0)=1$
and for $p_j\leq q_j$,
\begin{align*}
V_j(x,t)&\geq \rho_j^{\frac{(n+1+\alpha+\beta-\lambda)p'_j}{(q'_j-1)(p'_j-1)-1}}\min\limits_{\zeta\in \overline{B^{n+1}}} v_j(\zeta)\nonumber\\
& =\rho_j^{\frac{(n+1+\alpha+\beta-\lambda)(p'_j-q'_j)}{(q'_j-1)(p'_j-1)-1}}\frac{\min\limits_{\zeta\in \overline{B^{n+1}}} v_j}{u_j(\mathcal{T}^{-1}(0,...,0,0))}\nonumber\\
& \rightarrow +\infty \quad \  \mbox{uniformaly\ for\ any}\ (x,t)\in\mathbb{R}^{n+1}_+\ \mbox{as}\ j\rightarrow +\infty.
\end{align*}
Computing similarly to \eqref{contra}, we find that this case does not hold.

{\bf Case (2ii):} $\limsup\limits_{j\rightarrow+\infty}\frac{f_{p_j}(\mathcal{W})}{\max\limits_{\zeta\in \overline{B^{n+1}}} g_{q_j}}=0$.
Then there exist two subsequences of $p_j$ and  $q_j$ (still denoted by $p_j$ and  $q_j$) such that $f_{p_j}(\mathcal{W})\rightarrow +\infty$ and $\frac{f_{p_j}(\mathcal{W})}{\max\limits_{\zeta\in \overline{B^{n+1}}} g_{q_j}}\rightarrow0$, which implies that $\max\limits_{\zeta\in \overline{B^{n+1}}} g_{q_j}\rightarrow+\infty$.
Using similar arguments as {\bf Case (2i)}, we prove that {\bf Case (2ii)} does not hold.

{\bf Case (2iii):} $\limsup\limits_{j\rightarrow+\infty}\frac{f_{p_j}(\mathcal{W})}{\max\limits_{\zeta\in \overline{B^{n+1}}} g_{q_j}}=c_0\in (0,+\infty)$.
Then there exist two subsequences of $p_j$ and  $q_j$ (still denoted by $p_j$ and  $q_j$) such that $f_{p_j}(\mathcal{W})\rightarrow +\infty$, $\max\limits_{\zeta\in \overline{B^{n+1}}} g_{q_j}\rightarrow+\infty$ and $\frac{f_{p_j}(\mathcal{W})}{\max\limits_{\zeta\in \overline{B^{n+1}}} g_{q_j}}\rightarrow c_0$. Similar to {\bf Case (2i)}, we choose $\{U_j,V_j\}$ defined as \eqref{ints001}. It is easy to see that $U_j$ and $V_j$ satisfy \eqref{intsp02}. Moreover, $ U_j(x,t)\geq  U_j(0,...,0,0)=1$  and for $p_j\leq q_j$,
\begin{align*}
V_j(x,t)&\geq \rho_j^{\frac{(n+1+\alpha+\beta-\lambda)p'_j}{(q'_j-1)(p'_j-1)-1}}\min\limits_{\zeta\in \overline{B^{n+1}}} v_j(\zeta)\nonumber\\
& =\rho_j^{\frac{(n+1+\alpha+\beta-\lambda)(p'_j-q'_j)}{(q'_j-1)(p'_j-1)-1}}\frac{\min\limits_{\zeta\in \overline{B^{n+1}}} v_j}{u_j(\mathcal{T}^{-1}(0,...,0,0))}\nonumber\\
& \geq c_1>0 \quad \  \mbox{uniformaly\ for\ any}\ (x,t)\in\mathbb{R}^{n+1}_+\ \mbox{as}\ j\rightarrow +\infty.
\end{align*}
Hence $V_j(x,t)$ has uniformly lower bound $c_1>0$.

Arguing similarly to the proof of {\bf Case (iii)}, we show that $\{f(\xi), g(\eta)\}$ is a pair of minimizers of sharp constant $N_{\alpha,\beta,\lambda}$.
This completes the proof of Theorem \ref{equi01'}.
\hfill$\Box$

{\bf Proof of Theorem \ref{theorem2}.} Theorem \ref{theorem2} is an easy consequence of Theorem \ref{equi01'}.
 \hfill$\Box$

\section{Classification of extremal functions\label{Section 4}}

In this section, we investigate the regularity and radical symmetry of solutions to equation \eqref{Eulereq0},  and then classify the extremal function of inequality \eqref{RHLSD-1} by using the method of moving spheres.

Let
$$u(x,t)= f^{p-1}(x,t), \quad v(y,z)=g^{q-1}(y,z), \quad \theta=\frac{1}{1-p}>1\ \ \mbox{and}\ \kappa=\frac{1}{1-q}>1.$$
The Euler-Lagrange equation \eqref{Eulereq0} can be rewritten as the following integral system
\begin{equation}\label{intsp}\begin{cases}
u(x,t)=\int_{\mathbb{R}^{n+1}_+}\frac{t^{\alpha}z^{\beta} v^{-\kappa}(y,z)}{|(x,t)-(y,z)|^\lambda}dydz \quad (x,t)\in\mathbb{R}^{n+1}_+,\\
v(y,z)=\int_{\mathbb{R}^{n+1}_+}\frac{t^{\alpha}z^{\beta} u^{-\theta}(x,t)}{|(x,t)-(y,z)|^\lambda}dxdt \quad (y,z)\in\mathbb{R}^{n+1}_+,
\end{cases}\end{equation}
where $\alpha,\beta,\lambda,\kappa,\theta$ satisfy
\begin{equation}\label{RWH-exp-2}
\begin{cases}
&-n-1<\lambda<0, \kappa,\theta>1,\\
&0\le\alpha<\frac{1}{\theta-1},0\le\beta<\frac{1}{\kappa-1},\\
&\frac{1}{\kappa-1}+\frac{1}{\theta-1}=\frac{\alpha+\beta-\lambda}{n+1}.
\end{cases}
\end{equation}

\medskip

Theorem \ref{regularity1} and \ref{theoremfenlei1}  are equivalent to discuss the regularity, radical symmetry of system \eqref{intsp} as follows.

\begin{theorem}\label{regularity}
Let $\alpha,\beta,\lambda,\kappa,\theta$ satisfy \eqref{RWH-exp-2}
and $(u,v)$ be a pair of Lebesgue measurable positive solutions to system \eqref{intsp}. Then $u,v\in C^\infty(\mathbb{R}^{n+1}_+)\cap C^\gamma(\overline{\mathbb{R}^{n+1}_+})$ for $\gamma\in(0,1)$.
\end{theorem}

\begin{theorem}\label{theoremfen}
Let $\alpha,\beta,\lambda,\kappa,\theta$ satisfy \eqref{RWH-exp-2} and $\alpha,\beta>0$.
Suppose that $(u,v)$ is a pair of Lebesgue measurable positive solutions to system \eqref{intsp}, then
$u(x,t)$ and $v(x,t)$ are radically symmetric with respect to $x$ about some $x_0\in\mathbb{R}^{n}$.
Moreover, assume that
$$\kappa=\frac{2(n+1)+2\beta-\lambda}{2\beta-\lambda}, \quad \theta=\frac{2(n+1)+2\alpha-\lambda}{2\alpha-\lambda}.$$
Then, $u$ and $v$ must take the following form on the boundary $\partial \mathbb{R}^{n+1}_{+}$
\begin{equation*}
u(x,0)=c_1(\frac{d}{1+d^{2}|x-\xi_0|^2})^{\frac{\lambda-2\alpha}{2}}, \ \ v(x,0)=c_2(\frac{d}{1+d^{2}|x-\xi_0|^2})^{\frac{\lambda-2\beta}{2}} \quad \ \forall \, x\in\partial\mathbb{R}^{n+1}_{+}
\end{equation*}
for some $\xi_0\in\partial \mathbb{R}^{n+1}_{+}$, $c_1>0$, $c_2>0$ and $d>0$.
\end{theorem}

To show the classification of the solutions of \eqref{intsp}, we employ the method of moving spheres.
Below are some well-known notation in this process. For any $r>0$, denote
$$B_{r}(x,t):=\{(y,z)\in \mathbb{R}^{n+1} \mid |(y,z)-(x,t)|<r,\ (x,t)\in \mathbb{R}^{n+1}\},$$
$$B_{r}^{+}(x,t):=\{(y,z)=(y_1,y_2,\cdots,y_n,z)\in B_{r}(x,t) \mid  z>0,\ (x,t)\in \partial\mathbb{R}^{n+1}_{+}\}.$$

For $\xi\in\partial\mathbb{R}^{n+1}_{+}$ and $r>0$, set
$$(x,t)^{\xi,r}:=\frac{r^2((x,t)-\xi)}{|(x,t)-\xi|^2}+\xi, \ \quad \forall\ (x,t)\in \mathbb{R}^{n+1}_{+}.$$
Let $(u,v)$ be a pair of positive functions defined on $\mathbb{R}^{n+1}_{+}\times \mathbb{R}^{n+1}_{+}$. Define the Kelvin transforms
\begin{equation*}
u_{\xi,r}(x,t)=\big(\frac{r}{|(x,t)-\xi|}\big)^{\lambda-2\alpha} u((x,t)^{\xi,r}),\ \
v_{\xi,r}(y,z)=\big(\frac{r}{|(y,z)-\xi|}\big)^{\lambda-2\beta} v((y,z)^{\xi,r}) \qquad \forall \, (x,t), (y,z)\in \mathbb{R}^{n+1}_{+}.
\end{equation*}

\begin{lemma}\label{lemma1}
Assume that $(u,v)$ is a pair of positive solutions to system \eqref{intsp}. Then, for any $\xi\in\partial\mathbb{R}^{n+1}_{+}$ and $r>0$,
\begin{equation}\begin{split}\label{equ11}
&\quad u(x,t)-u_{\xi,r}(x,t)=\int_{B^{+}_{r}(\xi)}t^{\alpha}z^{\beta}K(\xi,r,(y,z),(x,t))
\big(\big(\frac{r}{|(y,z)-\xi|}\big)^{\mu_1}
v^{-\kappa}_{\xi,r}(y,z)-v^{-\kappa}(y,z)\big)dydz\\
\end{split}\end{equation}
for all $(x,t)\in \mathbb{R}^{n+1}_{+}$ and
\begin{equation}\label{equ12}\begin{split}
&\quad v(y,z)-v_{\xi,r}(y,z)=\int_{B^{+}_{r}(\xi)}t^{\alpha}z^{\beta}K(\xi,r,(y,z),(x,t))
\big(\big(\frac{r}{|(x,t)-\xi|}\big)^{\mu_2}
u^{-\theta}_{\xi,r}(x,t)-u^{-\theta}(x,t)\big)dxdt\\
\end{split}\end{equation}
for all $(y,z)\in\mathbb{R}^{n+1}_{+}$, where  $\mu_1=2(n+1)+2\beta-\lambda+(\lambda-2\beta)\kappa,\mu_2=2(n+1)+2\alpha-\lambda+(\lambda-2\alpha)\theta$, and
$$K(\xi,r,(y,z),(x,t)):=\big(\frac{r}{|(x,t)-\xi|}\big)^{\lambda}\frac{1}{|(x,t)^{\xi,r}-(y,z)|^{\lambda}}-\frac{1}{|(x,t)-(y,z)|^{\lambda}}.$$
Moreover, for any $\xi\in\partial\mathbb{R}^{n+1}_{+}$ and $r>0$,
\begin{equation}\label{equ224}
K(\xi,r,(y,z),(x,t))>0 \qquad \forall \, (x,t), (y,z)\in B^{+}_{r}(\xi).
\end{equation}
\end{lemma}

The proof of Lemma \ref{lemma1} is similar to that in \cite{L2004}. We shall skip the details here.
Now we use the idea of \cite{ ChenL2017, L2004} to prove the following result, which is a crucial ingredient for the rest paper.

\begin{lemma}\label{lem2}
Assume that $(u,v)$ is a pair of positive Lebesgue
measurable solutions to system \eqref{intsp}. Then

\noindent (i) $$\int_{\mathbb{R}^{n+1}_{+}}t^{\alpha}(1+|(x,t)|^{-\lambda})u^{-\theta}(x,t)dxdt <+\infty,$$
$$\int_{\mathbb{R}^{n+1}_{+}}z^{\beta}(1+|(y,z)|^{-\lambda})v^{-\kappa}(y,z)dydz <+\infty;$$

\noindent (ii) there exist constants $C_1\geq1$ and $C_2\geq1$ such that
$$\frac{1}{C_1}(1+|(x,t)|^{-\lambda})\leq \frac{u(x,t)}{t^\alpha}\leq C_1(1+|(x,t)|^{-\lambda}),$$
$$\frac{1}{C_2}(1+|(y,z)|^{-\lambda})\leq \frac{v(y,z)}{z^\beta}\leq C_2(1+|(y,z)|^{-\lambda});$$

\noindent (iii) the following asymptotic properties hold:
$$a=\lim_{t^\alpha|(x,t)|^{-\lambda}\rightarrow\infty}\frac{u(x,t)}{t^\alpha|(x,t)|^{-\lambda}}=\int_{\mathbb{R}^{n+1}_{+}}z^{\beta}v^{-\kappa}(y,z)dydz<+\infty,$$
$$b=\lim_{z^\beta|(y,z)|^{-\lambda}\rightarrow\infty}\frac{v(y,z)}{z^\beta|(y,z)|^{-\lambda}}=\int_{\mathbb{R}^{n+1}_{+}}t^{\alpha}u^{-\theta}(x,t)dxdt<+\infty.$$
Moreover, $u,v\in C^\infty(\mathbb{R}^{n+1}_+)\cap C^\gamma(\overline{\mathbb{R}^{n+1}_+})$ for $\gamma\in(0,1)$.
\end{lemma}

\begin{proof}
Using similar arguments as \eqref{nor13} and \eqref{nor2}, we derive (i) and (ii).

For $|(x,t)|>1$ and $|(y,z)|>1$, using (i) and (ii) gives
\begin{equation*}\begin{split}
\frac{u(x,t)}{t^\alpha|(x,t)|^{-\lambda}}&=|(x,t)|^{\lambda}\int_{\mathbb{R}^{n+1}_{+}}\frac{z^{\beta}v^{-\kappa}(y,z)}{|(x,t)-(y,z)|^{\lambda}}dydz\\
&\leq C \int_{\mathbb{R}^{n+1}_{+}}z^{\beta}(1+|(y,z)|^{-\lambda})v^{-\kappa}(y,z)dydz<+\infty
\end{split}\end{equation*}
and
\begin{equation*}\begin{split}
\frac{v(y,z)}{z^\beta|(y,z)|^{-\lambda}}&=|(y,z)|^{\lambda}\int_{\mathbb{R}^{n+1}_{+}}\frac{t^{\alpha}u^{-\theta}(x,t)}{|(x,t)-(y,z)|^{\lambda}}dxdt\\
&\leq C \int_{\mathbb{R}^{n+1}_{+}}t^{\alpha}(1+|(x,t)|^{-\lambda})u^{-\theta}(x,t)dxdt<+\infty.
\end{split}\end{equation*}
Hence, using dominated convergence theorem, we arrive at (iii).

\medskip

By conformal transformation \eqref{cftrans}, we have
\begin{align*}
&\int_{\mathbb{R}^{n+1}_{+}}z^{\beta}(1+|(x,t)|^{-\lambda})v^{-k}(x,t)dxdt\\
&=\int_{B^{n+1}}\big(\frac{1}{2}-\frac{|\eta-x^1|^2}{2}\big)^{\beta}
\big(\frac{2}{|\eta-x^0|}\big)^{2(n+1)+(\lambda-2\beta)k+2\beta}v^{-k}_{x^0,2}(\eta)\left(1+\left|\frac{2^2(\eta-x^0)}{|\eta-x^0|^2}
+x^0\right|^{-\lambda}\right)d\eta.
\end{align*}
Then using (i) gives
\begin{align*}
& \int_{B^{n+1}}\big(\frac{1}{2}-\frac{|\eta-x^1|^2}{2}\big)^{\beta}
\big(\frac{2}{|\eta-x^0|}\big)^{2(n+1)+(\lambda-2\beta)(k-1)}(1+|\eta|^{-\lambda-1})v^{-k}_{x^0,2}(\eta)d\eta\\
&\leq C\int_{B^{n+1}}\big(\frac{1}{2}-\frac{|\eta-x^1|^2}{2}\big)^{\beta}
\big(\frac{2}{|\eta-x^0|}\big)^{2(n+1)+(\lambda-2\beta)k+2\beta}v^{-k}_{x^0,2}(\eta)\left(1+\left|\frac{2^2(\eta-x^0)}{|\eta-x^0|^2}
+x^0\right|^{-\lambda}\right)d\eta\\
&<\infty,
\end{align*}
where in the first inequality we have used the fact
$$||\vec{a}|-|\vec{b}||\leq|\vec{a}-\vec{b}|.$$

A simple calculation gives
\begin{align*}
u_{x^0,2}(\zeta)
&=\int_{B^{n+1}}\big(\frac{1}{2}-\frac{|\zeta-x^1|^2}{2}\big)^{\alpha}\big(\frac{1}{2}-\frac{|\eta-x^1|^2}{2}\big)^{\beta}
\big(\frac{2}{|\eta-x^0|}\big)^{2(n+1)+(\lambda-2\beta)(k-1)}\frac{v^{-k}_{x^0,2}(\eta)}{|\zeta-\eta|^\lambda}d\eta,
\end{align*}
\begin{align*}
v_{x^0,2}(\eta)
&=\int_{B^{n+1}}\big(\frac{1}{2}-\frac{|\zeta-x^1|^2}{2}\big)^{\alpha}\big(\frac{1}{2}-\frac{|\eta-x^1|^2}{2}\big)^{\beta}
\big(\frac{2}{|\zeta-x^0|}\big)^{2(n+1)+(\lambda-2\alpha)(\theta-1)}\frac{u^{-\theta}_{x^0,2}(\zeta)}{|\zeta-\eta|^\lambda}d\zeta
\end{align*}
for all $\zeta\in B^{n+1}$.

For any given  $\zeta^1,\zeta^2\in \overline{B^{n+1}}$ and arbitrary $\eta\in B^{n+1}$, it is easy to verify that
\begin{align}\label{pineq}
&\Big||\zeta^1-\eta|^{-\lambda}-|\zeta^2-\eta|^{-\lambda}\Big|\leq
\begin{cases}
C|\zeta^1-\zeta^2|^{-\lambda}, \quad \text{if} \,\, -1<\lambda<0,\\ \\
C\big(1+|\eta|^{-\lambda-1}\big)|\zeta^1-\zeta^2|, \quad \text{if} \,\, \lambda\leq-1.
\end{cases}
\end{align}
Then using\eqref{pineq}, we deduce that for any given $\zeta^1,\zeta^2\in \overline{B^{n+1}}$ and $\alpha\in(0,1)$,
\begin{eqnarray*}
&&|u_{x^0,2}(\zeta^1)-u_{x^0,2}(\zeta^2)|\nonumber\\
&=&\big|\int_{B^{n+1}}\big(\frac{1}{2}-\frac{|\eta-x^1|^2}{2}\big)^{\beta}
\big(\frac{2}{|\eta-x^0|}\big)^{2(n+1)+(\lambda-2\beta)(k-1)}v^{-k}_{x^0,2}(\eta)\times\big.\nonumber\nonumber\\
&&\quad\  \big.\big[\big(\frac{1}{2}-\frac{|\zeta^1-x^1|^2}{2}\big)^{\alpha}\frac{1}{|\zeta^1-\eta|^\lambda}-
\big(\frac{1}{2}-\frac{|\zeta^2-x^1|^2}{2}\big)^{\alpha}\frac{1}{|\zeta^2-\eta|^\lambda}\big]d\eta
\big|\nonumber\\
&=&\big|\int_{B^{n+1}}\big(\frac{1}{2}-\frac{|\eta-x^1|^2}{2}\big)^{\beta}
\big(\frac{2}{|\eta-x^0|}\big)^{2(n+1)+(\lambda-2\beta)(k-1)}\frac{v^{-k}_{x^0,2}(\eta)}{|\zeta^2-\eta|^\lambda}\times\big.\nonumber\nonumber\\
&&\quad\  \big.\big[\big(\frac{1}{2}-\frac{|\zeta^1-x^1|^2}{2}\big)^{\alpha}-
\big(\frac{1}{2}-\frac{|\zeta^2-x^1|^2}{2}\big)^{\alpha}\big]d\eta\nonumber\\
&& +\int_{B^{n+1}}\big(\frac{1}{2}-\frac{|\zeta^1-x^1|^2}{2}\big)^{\alpha}\big(\frac{1}{2}-\frac{|\eta-x^1|^2}{2}\big)^{\beta}
\big(\frac{2}{|\eta-x^0|}\big)^{2(n+1)+(\lambda-2\beta)(k-1)}v^{-k}_{x^0,2}(\eta)\times\big.\nonumber\nonumber\\
&&\quad\  \big.\big[\frac{1}{|\zeta^1-\eta|^\lambda}-
\frac{1}{|\zeta^2-\eta|^\lambda}\big]
d\eta\big|\nonumber\\
&\leq& C|\zeta^1-\zeta^2|^\alpha\int_{B^{n+1}}\big(\frac{1}{2}-\frac{|\eta-x^1|^2}{2}\big)^{\beta}
\big(\frac{2}{|\eta-x^0|}\big)^{2(n+1)+(\lambda-2\beta)(k-1)}v^{-k}_{x^0,2}(\eta)d\eta\nonumber\\
&& + C|\zeta^1-\zeta^2|^{\gamma}
\begin{cases}
\int_{B^{n+1}}\big(\frac{1}{2}-\frac{|\eta-x^1|^2}{2}\big)^{\beta}
\big(\frac{2}{|\eta-x^0|}\big)^{2(n+1)+(\lambda-2\beta)(k-1)}v^{-k}_{x^0,2}(\eta)d\eta,\quad -1<\lambda<0\nonumber\\
\int_{B^{n+1}}\big(\frac{1}{2}-\frac{|\eta-x^1|^2}{2}\big)^{\beta}
\big(\frac{2}{|\eta-x^0|}\big)^{2(n+1)+(\lambda-2\beta)(k-1)}\big(1+|\eta|^{-\lambda-1}\big)v^{-k}_{x^0,2}(\eta)d\eta,\quad \lambda\leq-1
\end{cases}\nonumber\\
&\leq& C|\zeta^1-\zeta^2|^{\gamma},
\end{eqnarray*}
where $\gamma\in (0,\alpha]$.
For $\alpha=0$, we have
\begin{eqnarray*}
&&|u_{x^0,2}(\zeta^1)-u_{x^0,2}(\zeta^2)|\nonumber\\
&=&\big|\int_{B^{n+1}}\big(\frac{1}{2}-\frac{|\eta-x^1|^2}{2}\big)^{\beta}
\big(\frac{2}{|\eta-x^0|}\big)^{2(n+1)+(\lambda-2\beta)(k-1)}v^{-k}_{x^0,2}(\eta)\big[\frac{1}{|\zeta^1-\eta|^\lambda}-
\frac{1}{|\zeta^2-\eta|^\lambda}\big]
d\eta \big|\nonumber\\
&\leq& C|\zeta^1-\zeta^2|^{\gamma}.
\end{eqnarray*}
For $\alpha\geq1$, it is easy to check that $|u_{x^0,2}(\zeta^1)-u_{x^0,2}(\zeta^2)|\leq C|\zeta^1-\zeta^2|^{\gamma}$.
It follows that
$u_{x^0,2}(\zeta)\in C^\gamma(\overline{B^{n+1}})$ with $\gamma\in(0,1)$.
This implies that $u(x)$ is H\"{o}lder continuous in $\overline{\mathbb{R}^{n+1}}$. That is,
$$u(x)\in C^\gamma(\overline{\mathbb{R}^{n+1}}).$$
Similarly, one can deduce that $v(x)\in C^\gamma(\overline{\mathbb{R}^{n+1}})$.
By bootstrap, we obtain $u(x),v(x)\in C^\infty(\mathbb{R}^{n+1}_+)$.
\end{proof}

The following lemma shows that $u-u_{\xi,r}$ and $v-v_{\xi,r}$ are strictly positive in a small neighborhood of $\xi$.
\begin{lemma}\label{lemaa}
Assume that $(u,v)$ is a pair of positive Lebesgue measurable solutions to system \eqref{intsp}. Then, for any $\xi\in\partial\mathbb{R}^{n+1}_{+}$, there exists $\delta_{0}(\xi)>0$ small enough such that, for any $0<r\leq\delta_{0}$, there exist constants $C_3>0$ and $C_4>0$ such that
\begin{align*}
&u(x,t)-u_{\xi,r}(x,t)>C_3t^\alpha>0 ,\ \ v(y,z)-v_{\xi,r}(y,z)>C_4z^\beta>0 \qquad \forall \, (x,t), (y,z)\in B^+_{r^{2}}(\xi).
\end{align*}
\end{lemma}
\begin{proof}
Without loss of generality, we assume $\xi=0$. For any $(x,t)\in B^{+}_{r^{2}}(\xi)$, it is easy to see that $|(x,t)^{{0,r}}|\geq1$. From Lemma \ref{lem2}, we have
\begin{equation*}\begin{split}
u_{0,r}(x,t)&=\big(\frac{r}{|(x,t)|}\big)^{\lambda-2\alpha}u((x,t)^{0,r})\\
&\leq \big(\frac{r}{|(x,t)|}\big)^{\lambda-2\alpha}\big(\frac{r^2t}{|(x,t)|^{ 2}}\big)^{\alpha} \frac{2C_1}{|(x,t)^{0,r}|^{\lambda}}\\
&\leq \frac{2C_1t^{\alpha}}{r^{\lambda}}.
\end{split}\end{equation*}
Hence there exists $\delta_{0}>0$ sufficiently small, such that for any $0<r\leq\delta_{0}$, there exists constant $C_3>0$ such that
\begin{align*}
u(x,t)-u_{0,r}(x,t)&\geq u(x,t)-\frac{2C_1t^{\alpha}}{r^{\lambda}}\nonumber\\
&=t^\alpha\big(\int_{\mathbb{R}^{n+1}_{+}}\frac{ z^{\beta}v^{-\kappa}(y,z)}{|(x,t)-(y,z)|^{\lambda}}dydz-\frac{2C_1}{r^{\lambda}})\nonumber\\
&\geq C_3t^\alpha>0,\ \qquad \forall \, (x,t)\in B^{+}_{r^{2}}(0).
\end{align*}

Similarly, we derive that there exists $\delta_{0}>0$ small enough, such that for any $0<r\leq\delta_{0}$, there exists constant $C_4>0$ such that
\begin{equation*}\begin{split}
v(y,z)-v_{0,r}(y,z)\geq C_4z^\beta>0 \qquad \forall \, (y,z)\in B^{+}_{r^{2}}(0).
\end{split}\end{equation*}
This completes the proof of Lemma \ref{lemaa}.
\end{proof}

\medskip

For $\xi\in\partial\mathbb{R}^{n+1}_{+}$ and $r>0$, define
$$B_{r,u}^-=\{(x,t)\in B^{+}_{r}(\xi) \mid u(x,t)<u_{\xi,r}(x,t)\},\ \ B_{r,v}^-=\{(y,z)\in B_r^{+}(\xi) \mid v(y,z)<v_{\xi,r}(y,z)\}.$$
We start the moving sphere procedure by showing that $B_{r,u}^-=B_{r,v}^-=\emptyset$ for sufficiently small $r$.
\begin{lemma}\label{lemmastart}
For any $\xi\in\partial\mathbb{R}^{n+1}_{+}$, there exists $\epsilon_0(\xi)>0$ such that, for $r\in (0,\epsilon_0(\xi)]$,
\begin{equation*}
u(x,t)\geq u_{\xi,r}(x,t),\ \
v(y,z)\geq v_{\xi,r}(y,z) \qquad \forall \, (x,t), (y,z)\in B_{r}^{+}(\xi).
\end{equation*}
\end{lemma}
\begin{proof}
Using \eqref{equ11} and \eqref{equ224}, we have,
for any $(x,t) \in B_{r,u}^-$ and $\kappa\leq\frac{2(n+1)+2\beta-\lambda}{2\beta-\lambda}$, that
\begin{align}\label{revise0}
0&<u_{\xi,r}(x,t)-u(x,t)\nonumber\\
&=\int_{B^{+}_{r}(\xi)}t^\alpha z^{\beta}K(\xi,r,(y,z),(x,t))\big(v^{-\kappa}(y,z)-\big(\frac{r}{|(y,z)-\xi|}\big)^{\mu_1}
v^{-\kappa}_{\xi,r}(y,z)\big)dydz\nonumber\\
&\leq\int_{B_{r,v}^-}t^\alpha z^{\beta}K(\xi,r,(y,z),(x,t))\big(v^{-\kappa}(y,z)-v^{-\kappa}_{\xi,r}(y,z)\big)dydz\nonumber\\
&\leq \kappa\int_{B_{r,v}^-}t^\alpha z^{\beta}K(\xi,r,(y,z),(x,t))v^{-\kappa-1}(y,z)\big(v_{\xi,r}(y,z)-v(y,z)\big)dydz.
\end{align}

From Lemma \ref{lemaa}, we infer that, for any $r\in(0,\delta_{0}(\xi)]$,
$$B_{r,v}^-\subseteq B^+_r(\xi)\setminus B^+_{r^{2}}(\xi).$$
By the continuity of $v$,  there exists constant $C'>0$ independent of $r$ such that, for any $(x,t)\in B_{r,u}^{-}$ and $(y,z)\in B_{r,v}^-$,
\begin{align}\label{fvshangjie}
&\kappa t^\alpha z^{\beta}v^{-\kappa-1}(y,z)K(\xi,r,(y,z),(x,t))\nonumber\\
&\leq\kappa t^\alpha z^{\beta}v^{-\kappa-1}(y,z)\big(\frac{r}{|(x,t)-\xi|}\big)^{\lambda}
\big|\frac{r^2((x,t)-\xi)}{|(x,t)-\xi|^2}-(y,z)+\xi\big|^{-\lambda}\nonumber\\
&=\kappa t^\alpha z^{\beta} v^{-\kappa-1}(y,z)\big|\frac{r((x,t)-\xi)}{|(x,t)-\xi|}-\frac{((y,z)-\xi)|(x,t)-\xi|}{r}\big|^{-\lambda}\nonumber\\
&\leq \frac{C'}{r^{\lambda-\alpha-\beta}}.
\end{align}
Then, for sufficiently small $r$, it follows from \eqref{revise0} and \eqref{fvshangjie} that
\begin{align}\label{equ20}
\|u_{\xi,r}-u\|_{L^1(B_{r,u}^-)}
&\leq C'r^{-\lambda+\alpha+\beta}|B_{r,u}^{-}|\cdot\big\|v_{\xi,r}-v\big\|_{L^1(B_{r,v}^-)}\nonumber\\
&\leq \frac{1}{4}\big\|v_{\xi,r}-v\big\|_{L^1(B_{r,v}^-)}.
\end{align}

A similar computation shows that there exists a constant $C''>0$ independent of $r$ such that for sufficiently small $r$,
\begin{align}\label{equ21}
\big\|v_{\xi,r}-v\big\|_{L^1(B_{r,v}^-)}
&\leq C''r^{-\lambda+\alpha+\beta}|B_{r,v}^{-}|\cdot\|u_{\xi,r}-u\|_{L^1(B_{r,u}^-)}\nonumber\\
&\leq \frac{1}{4}\|u_{\xi,r}-u\|_{L^1(B_{r,u}^-)}.
\end{align}

Combining \eqref{equ20} with \eqref{equ21}, we find that there exists $\epsilon_{0}(\xi)\in(0,\delta_{0}(\xi)]$ small enough such that, for any $0<r\leq\epsilon_{0}(\xi)$,
$$\|u_{\xi,r}-u\|_{L^1(B_{r,u}^-)}=\big\|v_{\xi,r}-v\big\|_{L^1(B_{r,v}^-)}=0.$$
Therefore, $B_{r,u}^-=B_{r,v}^-=\emptyset$ for any $r\in(0,\epsilon_{0}(\xi)]$.
\end{proof}

\medskip

For each fixed $\xi\in \partial\mathbb{R}^{n+1}_+$, define
\begin{equation*}
  \bar{r}(\xi)=\sup\{r>0 \mid u\geq u_{\xi,\mu},\
  v\geq v_{\xi,\mu}\,\, \text{in} \,\, B^+_{\mu}(\xi),\,\, \forall \,\, 0<\mu\leq r\}.
\end{equation*}
By Lemma \ref{lemmastart}, $\bar{r}(\xi)$ is well-defined and $0<\bar{r}(\xi)\leq+\infty$ for any $\xi\in \partial\mathbb{R}^{n+1}_+$. The following lemma indicates that if the sphere stops for some $\xi\in \partial\mathbb{R}^{n+1}_+$, then $u$ and $v$ are symmetric about $\partial B^+_{\bar{r}(\xi)}$.

\begin{lemma}\label{lemmasequali}
If there exists some  $\bar{\xi}\in \partial\mathbb{R}^{n+1}_+$ satisfies $\bar{r}(\bar{\xi})<+\infty$, then
\begin{equation*}
u(x,t)=u_{\bar{\xi},\bar{r}(\bar{\xi})}(x,t) ,\ \
v(y,z)=v_{\bar{\xi},\bar{r}(\bar{\xi})}(y,z) \qquad \forall \, (x,t), (y,z)\in B^+_{\bar{r}(\bar{\xi})}(\bar{\xi}).
\end{equation*}
Furthermore, we must have
$$\kappa=\frac{2(n+1)+2\beta-\lambda}{2\beta-\lambda}, \qquad \theta=\frac{2(n+1)+2\alpha-\lambda}{2\alpha-\lambda}.
$$
\end{lemma}
\begin{proof}
Without loss of generality, we assume $\bar{\xi}=0$, and denote $\bar{r}:=\bar{r}(0)$.

Suppose on the contrary that $u-u_{0,\bar{r}}\geq0$ in $B^+_{\bar{r}}(0)$ and $v-v_{0,\bar{r}}\geq0$ in $B^+_{\bar{r}}(0)$ but at least one of them is not identically zero. Without loss of generality, we assume that $v-v_{0,\bar{r}}\geq0$ but $v-v_{0,\bar{r}}$ is not identically zero in $B^+_{\bar{r}}(0)$. Then we will derive the desired contradiction for the definition of $\bar{r}$.

We first show that
\begin{equation}\label{positive1}
  u(x,t)-u_{0,\bar{r}}(x,t)>0 \quad\,\,\,\,\, \forall \, (x,t)\in B^+_{\bar{r}}(0)
\end{equation}
and
\begin{equation}\label{positive2}
  v(y,z)-v_{0,\bar{r}}(y,z)>0 \quad\,\,\,\,\, \forall \, (y,z)\in B^{+}_{\bar{r}}(0).
\end{equation}
Indeed, choose a point $(y^{0},z^0)\in B^{+}_{\bar{r}}(0)$ such that $v(y^{0},z^0)-v_{0,\bar{r}}(y^{0},z^0)>0$. Due to the continuity of $v$, there exists $\eta>0$ and   $c_{0}>0$ such that
\begin{equation}\label{vbigger}
B_{\eta}(y^{0},z^0)\subset B^+_{\bar{r}}(0) \,\,\,\,\,\, \text{and} \,\,\,\,\,\,
v(y,z)-v_{0,\bar{r}}(y,z)\geq c_{0}>0 \quad\,\,\, \forall \, (y,z)\in B_{\eta}(y^{0},z^0).
\end{equation}

By Lemma \ref{lemma1}, we have $K(0,\bar{r},(y,z),(x,t))>0$ for any $(x,t), (y,z)\in B^+_{\bar{r}}(0)$.  Noting that $\mu_1\le0$  and $\kappa\leq\frac{2(n+1)+2\beta-\lambda}{2\beta-\lambda}$, we know from \eqref{equ11} and \eqref{vbigger} that
\begin{equation*}\begin{split}
u(x,t)-u_{0,\bar{r}}(x,t)
&=\int_{B^{+}_{\bar{r}}(0)}t^\alpha z^{\beta}K(0,\bar{r},(y,z),(x,t))\big(\big(\frac{\bar{r}}{|(y,z)|}\big)^{\mu_1}
v^{-\kappa}_{0,\bar{r}}(y,z)-v^{-\kappa}(y,z)\big)dydz\\
&\geq\int_{B_\eta(y^{0},z^0)}t^\alpha z^{\beta} K(0,\bar{r},(y,z),(x,t))\big(v^{-\kappa}_{0,\bar{r}}(y,z)-v^{-\kappa}(y,z)\big)dydz\\
&>c_1t^\alpha>0.
\end{split}\end{equation*}
Thus, we arrive at \eqref{positive1}.
The inequality \eqref{positive2} follows from a similar computation.

Using similar arguments as \eqref{equ20} and \eqref{equ21}, we know that there exist constants $C_0, C_1>0$ such that for all $r>0$,
$$\|u_{\xi,r}-u\|_{L^1(B_{r,u}^-)}\leq C_0r^{-\lambda+\alpha+\beta}|B_{r,u}^{-}|\cdot\big\|v_{\xi,r}-v\big\|_{L^1(B_{r,v}^-)},$$
$$\big\|v_{\xi,r}-v\big\|_{L^1(B_{r,v}^-)}
\leq C_1r^{-\lambda+\alpha+\beta}|B_{r,v}^{-}|\cdot\|u_{\xi,r}-u\|_{L^1(B_{r,u}^-)}.$$
Based on  the above estimates, we will show that $|B_{r,u}^{-}|\rightarrow0$ as $r\rightarrow \bar{r}^{+}$, and hence obtain the desired contradiction.

For $0\leq\rho<\eta\leq+\infty$, $r>0$ and $\delta\geq0$, we denote
$$A(\rho,\eta)=\{(x,t)\in\mathbb R^{n+1}_+: \rho<|(x,t)|<\eta\},$$
$$A(\rho,\eta; r, \delta)=\{(x,t)\in A(\rho,\eta): u_{0,r}(x,t)-u(x,t)>\delta\}.$$
We first note that $B_{\bar{r},u}^{-}=A(0,\bar{r}; \bar{r}, 0)$.
For any sufficiently small $\epsilon_1\in(0,\bar{r})$, if $\bar{r}\leq r\leq\bar{r}+\epsilon_1$, then we have
$$B_{r,u}^{-}\subset A(\bar{r}-\epsilon_1, \bar{r}+\epsilon_1)\cup (B^+_{\bar{r}}\cap\{t<\epsilon_1\})\cup (A(0,\bar{r}-\epsilon_1; r, 0)\cap\{t>\epsilon_1\})$$
Moreover, for such $r$, we obtain the following two estimates
$$|A(\bar{r}-\epsilon_1, \bar{r}+\epsilon_1)|\leq C(n+1)\bar{r}^{n}\epsilon_1,\ \ |B^+_{\bar{r}}\cap\{t<\epsilon_1\}|\leq C(n+1)\bar{r}^{n}\epsilon_1.$$
Thus, it suffices to show that for every $0<\epsilon_1<\bar{r}$,
$$\lim_{r\rightarrow \bar{r}^{+}}|A(0,\bar{r}-\epsilon_1; r, 0)\cap\{t>\epsilon_1\}|=0.$$
We prove this by contradiction. Suppose there is $0<\epsilon_1<\bar{r}$, $l>0$ and a sequence $r_k\rightarrow \bar{r}^{+}$ such that for all $k$,
\begin{equation}\label{set1}
|A(0,\bar{r}-\epsilon_1; r_k, 0)\cap\{t>\epsilon_1\}|>l.
\end{equation}
For $(x,t)\in B^{+}_{\bar{r}}(0)\setminus\{0\}$, we have
\begin{equation*}\begin{split}
u_{0,\bar{r}}(x,t)-u(x,t)&=\int_{B^{+}_{\bar{r}}(0)}t^\alpha z^{\beta}K(0,\bar{r},(y,z),(x,t))\big(v^{-\kappa}(y,z)-\big(\frac{\bar{r}}{|(y,z)|}\big)^{\mu_1}
v^{-\kappa}_{0,\bar{r}}(y,z)\big)dydz\\
&\leq\int_{B_{\bar{r},v}^-}t^\alpha z^{\beta}K(0,\bar{r},(y,z),(x,t))\big(v^{-\kappa}(y,z)-\big(\frac{\bar{r}}{|(y,z)|}\big)^{\mu_1}
v^{-\kappa}_{0,\bar{r}}(y,z)\big)dydz.
\end{split}\end{equation*}
The above estimate and \eqref{set1} imply that there exists $c_1>0$ depending only on $n,\alpha,\beta,\lambda,\epsilon_1$ and the distribution function of $v^{-\kappa}(y,z)-\big(\frac{\bar{r}}{|(y,z)|}\big)^{\mu_1}
v^{-\kappa}_{0,\bar{r}}(y,z)$ such that
$$u(x,t)-u_{0,\bar{r}}(x,t)\geq c_1\ \ \ \mbox{for\ all}\ (x,t)\in A(0,\bar{r}-\epsilon_1)\cap\{t>\epsilon_1\}.$$
Thus, for $(x,t)\in A(0,\bar{r}-\epsilon_1; r_k, 0)\cap\{t>\epsilon_1\}$, we have
$$u_{0,r_k}(x,t)-u_{0,\bar{r}}(x,t)\geq u(x,t)-u_{0,\bar{r}}(x,t)\geq c_1.$$
For any $(x,t)\in A(0,\bar{r}-\epsilon_1; r_k, 0)\cap\{t>\epsilon_1\}$ and $h\in C^{0}\cap L^1(\mathbb R^{n+1}_+)$, there exists $k(h)\in N$ such that for all
$k>k(h)$,
\begin{align*}
c_1&\leq u_{0,r_k}(x,t)-u_{0,\bar{r}}(x,t)\nonumber\\
&=\big(\big(\frac{r_k}{|(x,t)|}\big)^{\lambda-2\alpha}-\big(\frac{\bar{r}}{|(x,t)|}\big)^{\lambda-2\alpha}\big)u((x,t)^{0,\bar{r}})
+\big(\frac{r_k}{|(x,t)|}\big)^{\lambda-2\alpha}\big(u((x,t)^{0,r_k})-u((x,t)^{0,\bar{r}})\big)\nonumber\\
&\leq \big(\frac{r_k}{|(x,t)|}\big)^{\lambda-2\alpha}\big(u((x,t)^{0,r_k})-u((x,t)^{0,\bar{r}})\big)\nonumber\\
&\leq |u((x,t)^{0,r_k})-h((x,t)^{0,r_k})|+|h((x,t)^{0,r_k})-h((x,t)^{0,\bar{r}})|+|h((x,t)^{0,\bar{r}})-u((x,t)^{0,\bar{r}})|\nonumber\\
&\leq \frac{c_1}{2}+|u((x,t)^{0,r_k})-h((x,t)^{0,r_k})|+|h((x,t)^{0,\bar{r}})-u((x,t)^{0,\bar{r}})|.
\end{align*}
Then, for $k>k(h)$, we choose $h\in C^{0}\cap L^1(\mathbb R^{n+1}_+)$ sufficiently close to $u$ in $L^1(\mathbb R^{n+1}_+)$ such that
\begin{align*}
l&\leq 2 \big|\{(x,t)\in A(0,\bar{r}-\epsilon_1):  |u((x,t)^{0,r_k})-h((x,t)^{0,r_k})|\geq\frac{c_1}{4}\}\big|\nonumber\\
&\quad +2 \big|\{(x,t)\in A(0,\bar{r}-\epsilon_1): |h((x,t)^{0,\bar{r}})-u((x,t)^{0,\bar{r}})|\geq\frac{c_1}{4}\}\big|\nonumber\\
&\leq 4 \big|\{(x,t)\in \mathbb R^{n+1}_+: |u((x,t)^{0,r_k})-h((x,t)^{0,r_k})|\geq\frac{c_1}{4}\}\big|\nonumber\\
&\leq C(n)c_1^{-1}\|u-h\|_{L^1(\mathbb R^{n+1}_+)}\leq \frac{l}{2}.
\end{align*}
Therefore, for every $0<\epsilon_1<\bar{r}$, we conclude that
$$\lim_{r\rightarrow \bar{r}^{+}}|A(0,\bar{r}-\epsilon_1; r, 0)\cap\{t>\epsilon_1\}|=0.$$
This completes the proof of Lemma \ref{lemmasequali}.
\end{proof}

\medskip

Now we are ready to give a complete proof of Theorem \ref{theoremfen}.

{\bf Proof of Theorem \ref{theoremfen}.} We carry out the proof by considering two different possible cases.

{\bf Case (i).}  $\bar{r}(\xi)=+\infty$ for all $\xi\in \partial\mathbb{R}^{n+1}_+$. For all $\xi\in \partial\mathbb{R}^{n+1}_+$ and $0<r<+\infty$, we have
$$u_{\xi,r}(x,t)\leq u(x,t),\ \ v_{\xi,r}(y,z)\leq v(y,z) \qquad \forall \,\, (x,t), (y,z)\in B^{+}_{r}(\xi).$$
Then, we know from Lemma 3.7 in \cite{DZ2015a} that $u(x,t)$ only depends on $t$ and $v(y,z)$ only depends on $z$.
It follows from \eqref{intsp} that, for any $(x,t)\in\mathbb{R}^{n+1}_{+}$,
\begin{align*}
u(0,t)&=\int_{\mathbb{R}^{n+1}_+}\frac{ t^\alpha z^{\beta} v^{-\kappa}(0,z)}{|(0,t)-(y,z)|^\lambda}dydz \\
&=\int_0^\infty\frac{ t^\alpha z^{\beta} v^{-\kappa}(0,z)}{|t-z|^{\lambda}}dz\int_0^\infty\frac{\rho^{n-1}}{(\rho^2+1)^\frac{\lambda}{2}} d\rho=+\infty.
\end{align*}
Thus, \emph{Case (i)} is impossible.

\medskip

{\bf Case (ii).} There exists $\hat{\xi}\in\partial\mathbb{R}^{n+1}_{+}$ such that $\bar{r}(\hat{\xi})<+\infty$. Then, by the definition of $\bar{r}(\hat{\xi})$, one can see that for any $0<r<\bar{r}(\hat{\xi})$,
$$u_{\hat{\xi},r}(x,t)\leq u(x,t) \qquad \forall \,\, (x,t)\in B^{+}_{r}(\hat{\xi}).$$
Moreover, Lemma \ref{lemmasequali} indicates that $\kappa=\frac{2(n+1)+2\beta-\lambda}{2\beta-\lambda}, \theta=\frac{2(n+1)+2\alpha-\lambda}{2\alpha-\lambda}$ and
\begin{equation*}
  u_{\hat{\xi},\bar{r}(\hat{\xi})}(x,t)=u(x,t)  \qquad \forall \ (x,t)\in B^{+}_{\bar{r}(\hat{\xi})}(\hat{\xi}).
\end{equation*}

For any $\xi\in\partial\mathbb{R}^{n+1}_{+}$, we know from the definition of $\bar{r}(\xi)$ that for any $0<r\leq\bar{r}(\xi)$,
$$u_{\xi,r}(x,t)\leq u(x,t) \qquad \forall \,\, (x,t)\in B^{+}_{r}(\xi),$$
that is,
$$u(x,t)\leq u_{\xi,r}(x,t) \qquad  \forall \,\,|(x,t)-\xi|\geq r, \quad \forall \,\, 0<r\leq\bar{r}(\xi).$$
Therefore, for any $r\in (0,\bar{r}(\xi)]$,
\begin{eqnarray*}
&& \big[\bar{r}(\hat{\xi})\big]^{\lambda-2\alpha}u(\hat{\xi})=\liminf_{|(x,t)|\rightarrow+\infty}|(x,t)|^{\lambda-2\alpha}u_{\hat{\xi},\bar{r}(\hat{\xi})}(x,t)
  =\liminf_{|(x,t)|\rightarrow+\infty}|(x,t)|^{\lambda-2\alpha}u(x,t) \nonumber\\
&&\leq\liminf_{|(x,t)|\rightarrow+\infty}|(x,t)|^{\lambda-2\alpha}u_{\xi,r}(x,t)=r^{\lambda-2\alpha}u(\xi),
\end{eqnarray*}
which implies that $\bar{r}(\xi)<+\infty$ for all $\xi\in\partial\mathbb{R}^{n+1}_{+}$.

Applying Lemma \ref{lemmasequali}, we infer that, for all $\xi\in\partial\mathbb{R}^{n+1}_{+}$,
\begin{align*}
&u_{\xi,\bar{r}(\xi)}(x,t)=u(x,t) ,\ \ v_{\xi,\bar{r}(\xi)}(y,z)=v(y,z) \qquad  \forall \ (x,t), (y,z)\in\mathbb{R}^{n+1}_{+}.
\end{align*}
Then from Lemma 7.2 in \cite{L2004}, we have, for any $x\in\mathbb{R}^{n}$, that
\begin{equation*}
u(x,0)=c_1\big(\frac{d}{1+d^{2}|x-\xi_0|^2}\big)^{\frac{\lambda-2\alpha}{2}} \quad \text{and} \quad v(x,0)=c_2\big(\frac{d}{1+d^{2}|x-\xi_0|^2}\big)^{\frac{\lambda-2\beta}{2}}
\end{equation*}
for some $c_1>0$, $c_2>0$, $d>0$ and $\xi_0\in \partial\mathbb{R}^{n+1}_+$.
\medskip

We next state the following calculus lemma from \cite{QLX2008}, which gives the symmetry property of a function through the investigation of its Kelvin transformation.
\begin{lemma}[\cite{QLX2008}]\label{lemmasymm}
If $u\in C^{1}(\mathbb{R}^{n}\setminus \{0\})$ is a function such that for each $y\neq0$, there holds
$$u_{y,\lambda}(x)\leq u(x), \qquad  \forall \ 0<\lambda<|y|\ \mbox{and}\ |y-x|\geq \lambda\ \mbox{with}\ x\neq0,$$
then $u$ must be radically symmetric about the origin, and $u'(r)\leq0$ for $0<r<\infty$.
\end{lemma}

Now we can complete the proof of Theorem \ref{theoremfen}.
From Lemma \ref{lemmastart}, we know that there exists $\hat{\xi}\in\partial\mathbb{R}^{n+1}_{+}$ such that, for all $0<r<|\hat{\xi}|$,
\begin{align*}
&u(x,t)\geq u_{\hat{\xi},r}(x,t),\ \ v(y,z)\geq v_{\hat{\xi},r}(y,z) \qquad \forall \, (x,t), (y,z)\in B_{r}^{+}(\hat{\xi}).
\end{align*}
Define
\begin{equation*}
  \bar{r}(\xi)=\sup\{0<r\leq|\xi| : u\geq u_{\xi,\mu},\
  v\geq v_{\xi,\mu}\,\, \text{in} \,\, B^+_{\mu}(\xi),\,\, \forall \,\, 0<\mu\leq r\}.
\end{equation*}
By Lemma \ref{lemmasymm}, it suffices to show that
\begin{equation}\label{symmwequi}
\bar{r}(\xi)=|\xi|.
\end{equation}
Suppose \eqref{symmwequi} is not true, then there exists $\xi_0\in\partial\mathbb{R}^{n+1}_{+}$ such that $\bar{r}(\xi_0)<|\xi_0|$. For simplicity, we let $\bar{r}(\xi_0)=\bar{r}$. By the definition of $\bar{r}$, we have
\begin{align*}
&u(x,t)\geq u_{\xi_0,\bar{r}}(x,t) ,\ \ v(y,z)\geq v_{\xi_0,\bar{r}}(y,z) \qquad \forall \, (x,t), (y,z)\in B_{\bar{r}}^{+}(\xi_0).
\end{align*}
Using similar arguments as \eqref{positive1} and \eqref{positive2}, we derive
\begin{align*}
&u(x,t)> u_{\xi_0,\bar{r}}(x,t) ,\ \ v(y,z)> v_{\xi_0,\bar{r}}(y,z) \qquad \forall \, (x,t), (y,z)\in B_{\bar{r}}^{+}(\xi_0).
\end{align*}
Similar to proof process of Lemma \ref{lemmasequali}, one can conclude that there exists $\hat{\delta}>0$ such that, for all $r\in[\bar{r},\bar{r}+\hat{\delta}]$,
\begin{equation*}
u(x,t)-u_{\xi_0,r}(x,t)\geq0,\ \ v(y,z)-v_{\xi_0,r}(y,z)\geq0 \qquad \forall \, (x,t), (y,z)\in B^{+}_{r}(\xi_0),
\end{equation*}
which contradicts the definition of $\bar{r}$. Therefore, we must have $\bar{r}(\xi)=|\xi|$.

Then one has, for any $0<r<|\xi|$,
$$u_{\xi,r}(x,t)\leq u(x,t), \qquad \forall \,\, (x,t)\in B^{+}_{r}(\xi),$$
that is,
$$u(x,t)\leq u_{\xi,r}(x,t), \qquad  \forall \,\,|(x,t)-\xi|\geq r, \quad \forall \,\, 0<r<|\xi|.$$
Arguing as the proof of Lemma \ref{lemmasymm}, we know that $u(x,t)$ must be radically symmetric with respect to $x$.
Similarly, $v(y,z)$ is radically symmetric with respect to $y$. \hfill
$\square$

\section{The proof of Theorem \ref{theorem4}\label{Section 5}}

In this section, we will give the proof of Theorem \ref{theorem4} via  Pohozaev type identities. From Section \ref{Section 4}, we know system \eqref{Eulereq0} can be written as  system  \eqref{intsp}.
We first give the definition of weak positive solutions of system \eqref{intsp}.

\begin{df} We say that $(u,v)$  is pair of weak positive solutions of system \eqref{intsp}, if  $(u,v)$ is a pair of positive Lebesgue measurable solutions, and satisfies
\begin{eqnarray*}
& &\int_{\mathbb{R}^{n+1}_+}u(x,t)\varphi(x,t)dxdt+ \int_{\mathbb{R}^{n+1}_+}v(x,t)\varphi(x,t)dxdt\nonumber\\
&=&\int_{\mathbb{R}^{n+1}_+}\int_{\mathbb{R}^{n+1}_+}
  \frac{t^{\alpha}z^{\beta}v^{-\kappa}(y,z)\varphi(x,t)}{|(x,t)-(y,z)|^\lambda }dydzdxdt
  +\int_{\mathbb{R}^{n+1}_+}\int_{\mathbb{R}^{n+1}_+}
  \frac{z^{\alpha}t^{\beta}u^{-\theta}(y,z)\varphi(x,y)}{|(x,t)-(y,z)|^\lambda}dydz dxdt~~~~~~~~~~~~~~
\end{eqnarray*}
 for any $\varphi\in C^\infty_0(\mathbb{R}^{n+1}_+)$.
\end{df}

Theorem \ref{theorem4} can be proved through the following theorem.

\begin{theorem}\label{theorem4-1}
For $-n-1<\lambda<0$, $\alpha,\beta\geq0$ and $\kappa,\theta>1$, suppose that there exists a pair of weak positive solutions $(u,v)$ satisfying \eqref{intsp}, then a necessary condition for $k$ and $\theta$ is
$$\frac{n+1}{\theta-1}+\frac{n+1}{\kappa-1}=\alpha+\beta-\lambda.$$
\end{theorem}

\begin{proof}
For $-n-1<\lambda<0$, $\alpha,\beta\geq0$ and $\kappa,\theta>1$, assume that $(u,v)$ is a pair of positive Lebesgue measurable solutions to system \eqref{intsp}.  By Lemma \ref{lem2}, we have
\begin{align*}
0<\int_{\mathbb{R}^{n+1}_{+}}u^{1-\theta}(x,t)dxdt&\leq C_{1}\int_{\mathbb{R}^{n+1}_{+}}t^{\alpha}(1+|(x,t)|^{-\lambda})u^{-\theta}(x,t)dxdt<+\infty,\nonumber\\
0<\int_{\mathbb{R}^{n+1}_{+}}v^{1-\kappa}(y,z)dydz&\leq C_{2}\int_{\mathbb{R}^{n+1}_{+}}z^{\beta}(1+|(y,z)|^{-\lambda})v^{-\kappa}(y,z)dydz<+\infty.
\end{align*}
That is, $(u,v)\in L^{1-\theta}(\mathbb{R}^{n+1}_{+})\times L^{1-\kappa}(\mathbb{R}^{n+1}_{+})$.

By \eqref{intsp}, we obtain
\begin{equation*}\begin{cases}
u^{1-\theta}(x,t)=u^{-\theta}(x,t)\int_{\mathbb{R}^{n+1}_+}\frac{t^{\alpha}z^{\beta}}{|(x,t)-(y,z)|^{\lambda}}v^{-\kappa}(y,z)dydz, \qquad (x,t)\in\mathbb{R}^{n+1}_+,\\
v^{1-\kappa}(y,z)=v^{-\kappa}(y,z)\int_{\mathbb{R}^{n+1}_+}\frac{t^{\alpha}z^{\beta}}{|(x,t)-(y,z)|^{\lambda}}u^{-\theta}(x,t)dxdt, \qquad (y,z)\in\mathbb{R}^{n+1}_+.
\end{cases}\end{equation*}
Then, it follows from Fubini's theorem that
\begin{eqnarray}\label{Fubini}
\int_{\mathbb{R}_{+}^{n+1}}v^{1-\kappa}(y,z)dydz
&=&\int_{\mathbb{R}_{+}^{n+1}} \int_{\mathbb{R}_{+}^{n+1}}\frac{t^{\alpha}z^{\beta}}{|(x,t)-(y,z)|^{\lambda}}
u^{-\theta}(x,t)v^{-\kappa}(y,z)dxdtdydz\nonumber\\
&=&\int_{\mathbb{R}_{+}^{n+1}} u^{1-\theta}(x,t)dxdt.
\end{eqnarray}

For $\epsilon>0$ and $R>0$, we define
$$\phi_{\epsilon,R}(x,t)=\phi_{\epsilon}(x,t)\psi_{R}(x,t),$$
where $\phi_{\epsilon}(x,t)=\phi(\frac{|(x,t)|}{\epsilon})$ and $\psi_{R}(x,t)=\psi(\frac{|(x,t)|}{R})$, $\phi$ and $\psi$ are smooth functions in $\mathbb R$ satisfying $0\leq\phi,\psi\leq1$, $supp\ \phi\subset(1,+\infty)$, $supp\ \psi\subset(-\infty,2)$, and $\phi(\rho)=1$ for $\rho\geq2$ and $\psi(\rho)=1$ for $\rho\leq1$.
Thus, $supp\ \phi_{\epsilon,R}$ is contained in $\{\epsilon<|(x,t)|<2R\}$ for $\epsilon<2R$.

Since $(u,v)$ is a pair of  positive Lebesgue measurable to system \eqref{intsp}, $u$ and $v$ are smooth away from the singular set, and
$((x,t)\cdot\nabla u(x,t))\phi_{\epsilon,R}(x,t),((x,t)\cdot\nabla v(x,t))\phi_{\epsilon,R}(x,t)\in C^2_0(\mathbb{R}_{+}^{n+1})$.
Multiplying  the first equation in system \eqref{intsp} by $((x,t)\cdot\nabla u(x,t))\phi_{\epsilon,R}(x,t)$, we derive
\begin{align}\label{eq-a11}
&\quad \int_{B_{2R}^{+}(0)} u^{-\theta}(x,t)((x,t)\cdot\nabla u(x,t))\phi_{\epsilon,R}(x,t)dxdt\nonumber\\
&=\frac{1}{1-\theta}\int_{B_{2R}^{+}(0)} (x,t)\cdot\nabla (u^{1-\theta}(x,t))\phi_{\epsilon,R}(x,t)dxdt\nonumber\\
&=\frac{2R}{1-\theta}\int_{\partial B_{2R}(0)\cap\mathbb{R}_{+}^{n+1} }u^{1-\theta}(x,t)\phi_{\epsilon,R}(x,t)d\sigma-\frac{n+1}{1-\theta}\int_{B_{2R}^{+}(0)} u^{1-\theta}(x,t)\phi_{\epsilon,R}(x,t)dxdt\nonumber\\
&\quad -\frac{1}{1-\theta}\int_{B_{2R}^{+}(0)} u^{1-\theta}(x,t)((x,t)\cdot\nabla \phi_{\epsilon,R}(x,t))dxdt.
\end{align}
Similarly, we have
\begin{align}\label{eq-a12}
&\quad \int_{B_{2R}^{+}(0)} v^{-\kappa}(x,t)\big((x,t)\cdot\nabla  v(x,t)\big)\phi_{\epsilon,R}(x,t)dxdt\nonumber\\
&=\frac{1}{1-\kappa}\int_{B_{2R}^{+}(0)} (x,t)\cdot\nabla (v^{1-\kappa}(x,t))\phi_{\epsilon,R}(x,t)dxdt\nonumber\\
&=\frac{2R}{1-\kappa}\int_{\partial B_{2R}(0)\cap\mathbb{R}_{+}^{n+1} }v^{1-\kappa}(x,t)\phi_{\epsilon,R}(x,t)d\sigma-\frac{n+1}{1-\kappa}\int_{B_{2R}^{+}(0)} v^{1-\kappa}(x,t)\phi_{\epsilon,R}(x,t)dxdt\nonumber\\
&\quad -\frac{1}{1-\kappa}\int_{B_{2R}^{+}(0)} v^{1-\kappa}(x,t)\big((x,t)\cdot\nabla \phi_{\epsilon,R}(x,t)\big)dxdt.
\end{align}

Since $(u,v)\in L^{1-\theta}(\mathbb{R}^{n+1}_+)\times L^{1-\kappa}(\mathbb{R}^{n+1}_+)$ and $\phi_{\epsilon,R}$ is bounded, there exists a sequence $\{R_{j}\}$ with $R_j\rightarrow+\infty$ such that
$$R_j\int_{\partial B_{2R_j}(0)\cap\mathbb{R}_{+}^{n+1}} u^{1-\theta}(x,t)\phi_{\epsilon,R_j}(x,t)d\sigma\rightarrow0,\ \ R_j\int_{\partial B_{2R_j}(0)\cap\mathbb{R}_{+}^{n+1}} v^{1-\kappa}(x,t)\phi_{\epsilon,R_j}(x,t)d\sigma\rightarrow0.$$
Note that $\nabla \phi_\epsilon$ and $\nabla \psi_R$ have supports in $\{\epsilon<|(x,t)|<2\epsilon\}$ and $\{R<|(x,t)|<2R\}$, respectively. Therefore, $|(x,t)\cdot\nabla \phi_{\epsilon,R}(x,t)|\leq C$ on its compact support. Choose $R=R_{j}$ and let $\epsilon\rightarrow0$, $j\rightarrow+\infty$ in \eqref{eq-a11} and \eqref{eq-a12}, then using Lebesgue dominated convergence theorem we have
\begin{align}\label{poh}
&\quad \int_{\mathbb{R}^{n+1}_+} u^{-\theta}(x,t)\big((x,t)\cdot\nabla  u(x,t)\big)dxdt+\int_{\mathbb{R}^{n+1}_+}v^{-\kappa}(x,t)((x,t)\cdot\nabla v(x,t))dxdt\nonumber\\
&=-\frac{n+1}{1-\theta}\int_{\mathbb{R}^{n+1}_+} u^{1-\theta}(x,t)dxdt-\frac{n+1}{1-\kappa}\int_{\mathbb{R}^{n+1}_+} v(x,t)^{1-\kappa}dxdt.
\end{align}

On the other hand, using Lebesgue dominated convergence theorem and \eqref{intsp}, one can calculate that
\begin{align}\label{LHS-1}
&\int_{\mathbb{R}^{n+1}_+} u^{-\theta}(x,t)((x,t)\cdot\nabla u(x,t))\phi_{\epsilon,R}(x,t)dxdt\nonumber\\
&=\int_{\mathbb{R}^{n+1}_+} u^{-\theta}(x,t)\phi_{\epsilon,R}(x,t)\big[\sum_{j=1}^n x_j\frac{\partial}{\partial x_j} \big(\int_{\mathbb{R}^{n+1}_+}\frac{t^\alpha z^{\beta} v^{-\kappa}(y,z) }{|(x,t)-(y,z)|^{\lambda}} dydz\big)\nonumber\\
&\quad +t\frac{\partial}{\partial t} \big(\int_{\mathbb{R}^{n+1}_+}\frac{t^\alpha z^{\beta} v^{-\kappa}(y,z) }{|(x,t)-(y,z)|^{\lambda}}dydz \big)\big]     dxdt\nonumber\\
 &=-\lambda\int_{\mathbb{R}^{n+1}_+}\int_{\mathbb{R}^{n+1}_+} u^{-\theta}(x,t)\phi_{\epsilon,R}(x,t) \frac{t^\alpha z^{\beta} v^{-\kappa}(y,z) }{|(x,t)-(y,z)|^{\lambda+2}}\big[((x,t)-(y,z))\cdot (x,t)\big]dydzdxdt\nonumber\\
&\quad +\alpha\int_{\mathbb{R}^{n+1}_+}\int_{\mathbb{R}^{n+1}_+} u^{-\theta}(x,t)\phi_{\epsilon,R}(x,t) \frac{t^\alpha z^{\beta} v^{-\kappa}(y,z)}{|(x,t)-(y,z)|^{\lambda}} dydzdxdt
\end{align}
 and

 \begin{align}\label{LHS-2}
& \int_{\mathbb{R}^{n+1}_+} v^{-\kappa}(x,t)((x,t)\cdot\nabla v(x,t))\phi_{\epsilon,R}(x,t)dxdt\nonumber\\
&=\int_{\mathbb{R}^{n+1}_+}  v^{-\kappa}(x,t)\phi_{\epsilon,R}(x,t)\big[\sum_{j=1}^n x_j\frac{\partial}{\partial x_j} \big(\int_{\mathbb{R}^{n+1}_+}\frac{z^\alpha t^{\beta} u^{-\theta}(y,z) }{|(x,t)-(y,z)|^{\lambda}} dydz\big)\nonumber\\
&\quad +t\frac{\partial}{\partial t} \big(\int_{\mathbb{R}^{n+1}_+}\frac{z^\alpha t^{\beta} u^{-\theta}(y,z) }{|(x,t)-(y,z)|^{\lambda}} dydz\big)\big]    dxdt\nonumber\\
 &=-\lambda\int_{\mathbb{R}^{n+1}_+}\int_{\mathbb{R}^{n+1}_+}v^{-\kappa}(x,t)\phi_{\epsilon,R}(x,t) \frac{z^\alpha t^{\beta} u^{-\theta}(y,z) }{|(x,t)-(y,z)|^{\lambda+2}}\big[((x,t)-(y,z))\cdot (x,t)\big]dydzdxdt\nonumber\\
&\quad +\beta\int_{\mathbb{R}^{n+1}_+}\int_{\mathbb{R}^{n+1}_+} v^{-\kappa}(x,t)\phi_{\epsilon,R}(x,t) \frac{z^\alpha t^{\beta} u^{-\theta}(y,z)}{|(x,t)-(y,z)|^{\lambda}} dydzdxdt.
\end{align}
Let $R\to\infty$ and $\epsilon\to0$, combining \eqref{LHS-1} with \eqref{LHS-2}, we arrive at
\begin{align}\label{poh000}
&\quad \int_{\mathbb{R}^{n+1}_+} u^{-\theta}(x,t)\big((x,t)\cdot\nabla  u(x,t)\big)dxdt+\int_{\mathbb{R}^{n+1}_+}v^{-\kappa}(y,z)\big((y,z)\cdot\nabla v(y,z)\big)dx\nonumber\\
&=(\alpha+\beta-\lambda)\int_{\mathbb{R}_{+}^{n+1}} \int_{\mathbb{R}_{+}^{n+1}}\frac{t^{\alpha}z^{\beta}}{|(x,t)-(y,z)|^{\lambda}}u^{-\theta}(x,t)v^{-\kappa}(y,z)dxdtdydz.
\end{align}
Substituting \eqref{Fubini}, \eqref{poh} into \eqref{poh000}, we have
$$\frac{n+1}{\theta-1}+\frac{n+1}{\kappa-1}=\alpha+\beta-\lambda.$$
\end{proof}

As a consequence, we obtain the following Liouville type theorem for positive solutions of system \eqref{intsp}.
\begin{crl}\label{coro01}
For $-n-1<\lambda<0$, $\alpha,\beta\geq0$, $\kappa,\theta>1$, assume that
$$\frac{n+1}{\theta-1}+\frac{n+1}{\kappa-1}\neq\alpha+\beta-\lambda,$$
then there does not exist a pair of positive Lebesgue measurable solutions $(u,v)$ satisfying \eqref{intsp}.
\end{crl}

 \vskip 1cm

\noindent {\bf Acknowledgements}\\
This project is supported by  the
National Natural Science Foundation of China (Grant No. 12101380, 12071269), China Postdoctoral Science Foundation (Grant No. 2021M700086), Youth Innovation Team of Shaanxi Universities  and the Fundamental Research Funds for the Central Universities (Grant No. GK202307001, GK202202007)

\medskip

\end{document}